\documentclass[11pt]{article}
\usepackage[utf8]{inputenc}
\usepackage[english]{babel}
\usepackage[T1]{fontenc}
\usepackage{float}
\usepackage{listings}
\usepackage{amsmath}
\usepackage{amsfonts}
\usepackage{amssymb}
\usepackage{amsthm}
\usepackage{enumerate}
\usepackage{anysize}
\usepackage{graphicx}
\usepackage{xcolor}
\usepackage{color}
\usepackage{fancyhdr}
\usepackage{makeidx}
\usepackage{hyperref}
\usepackage{appendix}
\usepackage{tikz}
\usepackage[matrix,arrow]{xy}
\usetikzlibrary{arrows} 
\usepackage{url}
\usepackage[square,numbers]{natbib}
\usepackage{subcaption}
\usepackage{stackrel}
\marginsize{2cm}{2cm}{0cm}{2cm}

\definecolor
{gray75}{gray}{0.75}
\definecolor
{gray85}{gray}{0.85}
\numberwithin{equation}{section}

\newtheorem{teo}{Theorem}[section]
\newtheorem{lem}{Lemma}[section]
\newtheorem{defi}{Definition}[section]
\newtheorem{obs}{Remark}[section]
\newtheorem{col}{Corollary}[section]

\begin{document}
\title{\textbf{Theoretical and numerical analysis for a hybrid tumor model with diffusion depending on vasculature}}
	\author{ A. Fernández-Romero$^{2}$, F. Guillén-González$^{2}$\footnote{ORCID: 0000-0001-5539-5888}, A. Suárez$^{1\;2}$\footnote{ORCID: 0000-0002-6407-7758}.\\ \small{$^{1}$Corresponding author.}\\
	\small{$^{2}$Dpto. Ecuaciones Diferenciales y Análisis Numérico,}\\
	\small{Facultad de Matemáticas, Universidad de Sevilla. Sevilla, Spain.}\\
	\small{\texttt{afernandez61@us.es, guillen@us.es, suarez@us.es}}	
}
\date{}
\maketitle
% \setcounter{tocdepth}{3}
% \tableofcontents
% \vspace*{0.3cm}
\section*{\centering{\textbf{Abstract}}}

In this work we analyse a PDE-ODE problem modelling the evolution of a Glioblastoma, which includes an anisotropic nonlinear diffusion term with a diffusion velocity increasing with respect to vasculature. First, we prove the existence of global in time weak-strong solutions using a regularization technique via an artificial diffusion in the ODE-system and a fixed point argument. In addition, stability results of the critical points are given under some constraints on parameters. Finally, we design a fully discrete finite element scheme for the model which preserves the pointwise and energy estimates of the continuous problem.
\\
\\
\textbf{Mathematics Subject Classification.} $35\text{A}01,\;35\text{B}40,\;35\text{M}10,\;35\text{Q}92,\;47\text{J}35,\;92\text{B}05$\\
\textbf{Keywords:} Tumor model, Glioblastoma, PDE-ODE system, Numerical scheme.

\subsection*{Funding}
The authors were supported by PGC2018-098308-B-I00 (MCI/AEI/FEDER, UE).

\section{Introduction}

In this paper we investigate the following parabolic PDE-ODE system

%Specifically, let $\Omega\subseteq\mathbb{R}^3$ be a bounded domain of and $\left(0,T_f\right)$ a time interval, with $0<T_f<+\infty$ and we analyse the following PDE-ODEs system where $T\left(t,x\right)\geq0$, $N\left(t,x\right)\geq0$ and $\Phi\left(t,x\right)\geq0$ represent the tumor density, necrotic density and vasculature concentration, respectively, at the point $x\in\Omega$ at the time $t\in\left(0,T_f\right)$. 

\begin{equation}\label{probOriginal}
\left\{\begin{array}{ccl}
\dfrac{\partial T}{\partial t} -\nabla\cdot\left(\left(\kappa_1\;P\left(\Phi,T\right)+\kappa_0\right)\nabla\;T\right)& = & f_1\left(T,N,\Phi\right)\quad\text{in}\;\;\left(0,T_f\right)\times\Omega\\ %\;\;\text{en}\;\;\left(0,T_f\right)\times\Omega\\
&&\\
\dfrac{\partial N}{\partial t}& = & f_2\left(T,N,\Phi\right)\quad\text{in}\;\;\left(0,T_f\right)\times\Omega\\%\;\;\text{en}\;\;\left(0,T_f\right)\times\Omega \\
&&\\
\dfrac{\partial \Phi}{\partial t} & = &f_3\left(T,N,\Phi\right)\quad\text{in}\;\;\left(0,T_f\right)\times\Omega%\;\;\text{en}\;\;\left(0,T_f\right)\times\Omega \\
\\
%\hspace{0.5cm}\dfrac{\partial T}{\partial \text{n}}\Bigg\vert_{\partial\Omega}=0\;\;\text{en}\;\;\left(0,T_f\right)\times\partial\Omega\\
%\\
\end{array}\right.\end{equation}
\\
endowed with non-flux boundary condition
\begin{equation}\label{condifronte}
\dfrac{\partial T}{\partial n}\Bigg\vert_{\partial\Omega}=0\;\;\text{on}\;\;\left(0,T_f\right)\times\partial\Omega
\end{equation} 
where $n$ is the outward unit normal vector to $\partial\Omega$ and initial conditions
\begin{equation}\label{condinicio}
T\left(0,\cdot\right)=T_0,\;N\left(0,\cdot\right)=N_0,\;\Phi\left(0,\cdot\right)=\Phi_0
\;\;\text{in}\;\;\Omega.
\end{equation}

Here $\Omega\subset\mathbb{R}^3$ is a smooth bounded domain, $T_f>0$ the final time, and $T(t,x), N(t,x)$ and $\Phi(t,x)$ represent the tumor and necrotic densities and the vasculature concentration at the point $x\in\Omega$ and time $t>0$, respectively. We have chosen $\Omega\subset\mathbb{R}^3$ although the dimension of the domain will not have influence on our study as we see along the paper. The nonlinear reaction terms are:

%We assume initial datas $T_0,N_0,\Phi_0\geq0$ in $\Omega$ such that $0\leq T_0(x)+N_0(x)+\Phi_0(x)\leq K$ where $K>0$ is the capacity. 
%\\
%
%In $\left(\ref{probOriginal}\right)$, the unknowns are tumor density $T\left(t,x\right)\geq0$, necrotic density $N\left(t,x\right)\geq0$ and vasculature concentration $\Phi\left(t,x\right)\geq0$ every in point $x\in\Omega$ at time $t\in\left(0,T_f\right)$. The functions $f_i\;:\mathbb{R}^3\rightarrow\mathbb{R}$ para $i=1,2,3$, which appear in the right side of $\left(\ref{probOriginal}\right)$, denote the reaction functions and they have the following definition

%The nonlinear functions $f_i\;:\mathbb{R}^3\rightarrow\mathbb{R}$ for $i=1,2,3$ denote the reaction functions and they have the following form

\begin{equation}\label{funciones}
\left\{\begin{array}{ccl}
f_1\left(T,N,\Phi\right) &:=&\rho\;T\;P\left(\Phi,T\right) \left(1-\dfrac{T+N+\Phi}{K}\right)-\alpha\;T\;\sqrt{1-P\left(\Phi,T\right)^2}-\beta_1\; 
N\;T,%:=\\
\\
%&:=&T\;\widetilde{f_1}\left(T,N,\Phi\right)\\
\\
f_2\left(T,N,\Phi\right) &:=& \alpha\;T\;\sqrt{1-P\left(\Phi,T\right)^2}+\beta_1\;N\;T +\delta\;T\;\Phi +\beta_2\;N\;\Phi,%:=\\
\\
%&:=&\beta\;N\left(T+\Phi\right)+T\;\widetilde{f_2}\left(T,\Phi\right)\\
\\
f_3\left(T,N,\Phi\right) &:=& \gamma\;T\;\sqrt{1-P\left(\Phi,T\right)^2}\;\dfrac{\Phi}{K}\left(1-\dfrac{T+N+\Phi}{K}\right)-\delta\;T\;\Phi-\beta_2\;N\;\Phi,%:=\\
%\\
%&:=&\Phi\;\widetilde{f_3}\left(T,N,\Phi\right)
\end{array}\right.
\end{equation}
where $\kappa_1,\kappa_0>0$ are diffusion coefficients, $\rho,\;\alpha,\;\beta_1,\;\beta_2,\;\delta,\;\gamma>0$ are reaction coefficients and $K>0$ is the maximum population size that can be sustained by the environment (see Table \ref{Table1} and \cite{Klank_2018, Alicia_2015, Alicia_2012} for a description of the parameters):

\begin{table}[H]
	\centering
	\begin{tabular}{c|c|c}
		%\hline
		\textbf{Variable} & \textbf{Description} & \textbf{Value} \\
		\hline
		$\kappa_1$	& Anisotropic speed diffusion & $cm^2/\text{day}$     \\
		\hline
		$\kappa_0$	& Isotropic diffusion & $cm^2/\text{day}$     \\
		\hline
		$\rho$	& Tumor proliferation rate  &  $\text{day}^{-1}$    \\
		\hline
		$\alpha$ & Hypoxic death rate  & $cell/\text{day}$    \\
		\hline
		$\beta_1$	& Change rate from tumor to necrosis &   $\text{day}^{-1}$    \\
		\hline
		$\beta_2$	& Change rate from vasculature to necrosis &   $\text{day}^{-1}$    \\
		\hline
		$\gamma$	& Vasculature proliferation rate  &  $\text{day}^{-1}$    \\
		\hline
		$\delta$	& Vasculature destruction by tumor   &  $\text{day}^{-1}$    \\
		\hline
		$K$	& Carrying capacity & $\text{cell}/\text{cm}^3$ \\
		%\hline    
	\end{tabular}
	\caption{\label{Table1} Coefficients.}
\end{table}

The vasculature volume fraction function $P\left(\Phi,T\right)$ is defined as follows

\begin{equation}\label{funcion_P}
P\left(\Phi, T\right)=%\left\{\begin{array}{cl}
\dfrac{\Phi_+}{\left(\dfrac{\Phi_++K}{2}\right)+T_+}
\end{equation}
%\;\; \text{  if  }\; \left(\Phi,T\right)\neq\left(0,0\right)$$
%
%\\
%0&\quad \text{otherwise}\end{array} \right.$$
with $T_+=\max\{0,T\}$ and similar to $\Phi_+$. Notice that $P\left(\Phi,T\right)$ is continuous in $\mathbb{R}^2$, satisfies the pointwise estimates 
%$$0\leq P\left(\Phi,T\right)\leq2\quad\forall\left(T,\Phi\right)\in\mathbb{R}^2,$$%\left[0,K\right]\times\left[0,K\right]%\backslash\left\{\left(0,0\right)\right\}$$
\begin{equation}\label{cota_funcion_P}
0\leq P\left(\Phi,T\right)\leq1\quad\forall\left(T,\Phi\right)\in\left[0,K\right]\times\left[0,K\right]
\end{equation}
and $P\left(\Phi,T\right)=0$ for $\Phi=0$ (without vasculature) and $P\left(\Phi,T\right)=1$ for $\left(\Phi,T\right)=\left(K,0\right)$ (maximum of vasculature).
\\

It is well-known that Glioblastoma (GBM) presents pathologically important differences with respect to other brain tumors of lesser malignancy. Given the great difficulties presented by the treatment of GBM, the mathematical modelling of GBM has been a relatively broad topic in the applied mathematics community. However, the applicability of the results has been very reduced \cite {Baldock_2014, Alicia_2015, Molina_2016, Julian_2016}.
\\

Our system $\left(\ref{probOriginal}\right)$-$\left(\ref{condinicio}\right)$ models some biological aspects of the evolution of the Glioblastoma (GBM). In \cite{Romero_2020}, we analysed a simplified model of $\left(\ref{probOriginal}\right)$-$\left(\ref{condinicio}\right)$ with linear diffusion ($\kappa_1=0$) where we did a study of classical solution. Moreover, we explained the relation between biological effects and reaction terms.
\\
%Our system $\left(\ref{probOriginal}\right)$-$\left(\ref{condinicio}\right)$ models some biological aspects of the evolution of the Glioblastoma (GBM). In \cite{Romero_2020}, we explained the relation between biological effects and reaction terms and that GBM presents pathologically important differences with respect to other brain tumors of lesser malignancy. Given the great difficulties presented by the treatment of GBM, the mathematical modelling of GBM has been a relatively broad topic in the applied mathematics community. However, the applicability of the results has been very reduced \cite {Baldock_2014, Alicia_2015, Molina_2016, Julian_2016}. In \cite{Romero_2020} we analysed a simplified model of $\left(\ref{probOriginal}\right)$-$\left(\ref{condinicio}\right)$ with linear diffusion ($\kappa_1=0$) where we did a study of classical solution.
%\\
%Consequently, we will model some aspects about this biological phenomenon such as the regular or irregular tumor growth with the effect of nonlinnear diffusion term depending mainly on the vasculature in the tumor equation.
%\\
%The mathematical model $\left(\ref{probOriginal}\right)$-$\left(\ref{condinicio}\right)$ was proposed including three differential equations which related the density of tumor cells, the density of necrosis and the concentration of vasculature.
%\\

Following the recommendation of Molab\footnote{\href{Molab}{http://matematicas.uclm.es/molab/}} group, which classifies the GBM depending on the width of the tumor ring and/or the tumor surface regularity (see \cite{Julian_2016, Victor_2018} respectively) using image treatment about GBM or more recently, working with a PDE-ODE system with linear diffusion in \cite{Victor_2020}, we study $\left(\ref{probOriginal}\right)$-$\left(\ref{condinicio}\right)$ where we have included a nonlinear diffusion term in the spatial mobility of the tumor with the diffusion velocity increasing with vasculature. In fact, the diffusion term in $\left(\ref{probOriginal}\right)$-$\left(\ref{condinicio}\right)$ includes the nonlinear term, $\kappa_1 P(\Phi,T)$, and the linear self-diffusion term with coefficient $\kappa_0>0$, what makes the diffusion non-degenerate (although from a biological point of view $\kappa_0$ must be small). Therefore, tumor cells show a random movement when there is not nutrient limitation (included in the linear self-diffusion term) whereas they have a kinematic movement when there exists a nutrient limitation. Thus, we express this possible lack of nutrient through function $P\left(\Phi,T\right)$ which measures the quotient between the amount of vasculature and the amount of vasculature and tumor together.
%	
%	$$\nabla\cdot\left(\left(\kappa_1\;P\left(\Phi,T\right)+\kappa_0\right)\cdot\nabla\;T\right).$$
\\

The inclusion of this nonlinear diffusion term makes our model more realistic than model studied in \cite{Romero_2020}, but it entails technical complications that we try to overcome in this work. Specifically, the main contributions of the paper are the following:
\begin{enumerate} 
	\item The existence of global in time weak solutions of $\left(\ref{probOriginal}\right)$-$\left(\ref{condinicio}\right)$. For that, we regularize $\left(\ref{probOriginal}\right)$-$\left(\ref{condinicio}\right)$ including an artificial diffusion in the ODE-system. This regularized problem maintains the same pointwise estimates as $\left(\ref{probOriginal}\right)$-$\left(\ref{condinicio}\right)$ and it is solved by a fixed point argument. Finally, we get some estimates for the solution of regularized problem which let us to pass to the limit arriving at one solution of $\left(\ref{probOriginal}\right)$-$\left(\ref{condinicio}\right)$.
	\item We investigate the asymptotic behaviour of $\left(\ref{probOriginal}\right)$-$\left(\ref{condinicio}\right)$. Mainly, we prove that the vasculature tends to zero "pointwisely" as time goes to infinity and, under some constraints on the parameters, tumor also goes to the extinction and necrosis grows to a upper limit. Looking at the asymptotic behaviour of the linear diffusion problem (see \cite[Section 4]{Romero_2020}), we conclude that the nonlinear diffusion model has a similar behaviour.
	\item The construction of an uncoupled and linear numerical scheme of $\left(\ref{probOriginal}\right)$-$\left(\ref{condinicio}\right)$ by means of a implicit-explicit finite difference scheme in time and a finite element with "mass-lumping" approximation in space which preserves the pointwise and energy estimates of the continuous model whenever an acute triangulation is considered. 
	
\end{enumerate}

There exist many studies dedicated to the analysis of PDE-ODE systems in the literature, see for instance \cite{Bitsouni_2017, Cruz_2018, Anderson_2015,Romero_2020, Kubo_2016, Lou_2014, Pang_2018,  Tello_2018} and the references therein. Some of them such as \cite{Kubo_2016, Tello_2018}, use results of classical solutions given in \cite{Amann_1991,Amann_1993}. On the other hand, in \cite{Lou_2014}, the study of a PDE-ODE system is based on approximating regularized problems pointwise estimates. Moreover, the results obtained in \cite{Lou_2014} are used in a recent work of the same authors for other PDE-ODE system, see \cite{Lou_2020}. In a previous paper \cite{Romero_2020}, we have studied the PDE-ODE system $\left(\ref{probOriginal}\right)$-$\left(\ref{condinicio}\right)$ with linear diffusion ($\kappa_1=0$) where we get existence and uniqueness of classical solution using a fixed point argument. Now, due to the complexity of the nonlinear diffusion term, we will prove  existence of a so-called weak-strong solution (see Definition \ref{defi_Probl_Dif_No_Lineal} below). Roughly speaking, it will be a variational solution for the tumor-PDE and pointwise for the ODE system with necrosis and vasculature variables.
\\

There are multiple results according to the analysis of the Finite Element (FE) scheme which preserves pointwisely and energy estimates related to parabolic PDEs with maximum principle, see for instance \cite{Elliot_1993}. Specifically, in order to obtain pointwisely estimates for FE numerical scheme of nonlinear PDE-ODE systems with maximum principle, we highlight works such as \cite{Farago_2012,Thomee_2015,Thomee_2008} or \cite{Ciarlet_1973} where it is considered an acute triangulation to have the pointwisely estimates and \cite{Kisko_2019} for energy estimates. Another relevant paper in the study of FE method for nonlinear PDE is \cite{Nie_1985} where the authors use a mass-lumping technique with quadrature formula. 
\\

The outline of the paper is as follows. In Section $2$, we present preliminary results which we will use along the study of system $\left(\ref{probOriginal}\right)$-$\left(\ref{condinicio}\right)$. In Section $3$ we prove the existence of weak-strong solutions of $\left(\ref{probOriginal}\right)$-$\left(\ref{condinicio}\right)$. Section $4$ is dedicated to the long time behaviour of the solution. Finally, in Section $5$ we present a numerical scheme of our model $\left(\ref{probOriginal}\right)$-$\left(\ref{condinicio}\right)$ which preserves the same estimates as the continuous model.

%The paper is organized as follows: In Section $2$, we present preliminary results which we will use along the study of system $\left(\ref{probOriginal}\right)$-$\left(\ref{condinicio}\right)$. In Section $3$ we prove the existence of weak-strong solution of $\left(\ref{probOriginal}\right)$-$\left(\ref{condinicio}\right)$. Section $4$ is dedicated to the long time behaviour of the classical solution where the main results are proven under certain relations between some parameters. 
\section{Preliminaries}

In this section we include some necessary results to study the existence of solutions of the system $\left(\ref{probOriginal}\right)$-$\left(\ref{condinicio}\right)$. 
\\

The regularity of the functions $P\left(\Phi,T\right)$ and $f_i\left(T,N,\Phi\right)$ for $i=1,2,3$ (see \cite[Lemma 2]{Romero_2020}) can be summarized in the following lemma
\begin{lem}\label{lemafi}
	The functions $P\;:\mathbb{R}^2\rightarrow\mathbb{R}$ and $f_i\;:\mathbb{R}^3\rightarrow\mathbb{R}$ for $i=1,2,3$ defined in $\left(\ref{funcion_P}\right)$ and $\left(\ref{funciones}\right)$, are continuous and locally lipschitz.
\end{lem}
%\begin{proof}
%	Rewriting the definition of $f_i\left(T,N,\Phi\right)$ for every $i=1,2,3$ according to the functions $B\left(\Phi,T\right)$ and $D\left(\Phi,T\right)$, it is easy to deduce that functions $f_i\left(T,N,\Phi\right)$ are continuous and their partial derivates are bounded in compact sets of $\mathbb{R}^3$ for every $i=1,2,3$, because they are products and sums of the globally lipschitz functions $B\left(\Phi,T\right)$ and $D\left(\Phi,T\right)$ and polynomials.
%\end{proof}
In order to define the concepts of weak and strong solution for a parabolic problem, we introduce the following "weak" space
	\begin{equation}\label{W}
\mathcal{W}_2=\left\{u\in L^\infty\left(0,T_f;L^2\left(\Omega\right)\right)\cap L^2\left(0,T_f;H^1\left(\Omega\right)\right):\;\;u_t\in L^2\left(0,T_f;\left(H^1\left(\Omega\right)\right)'\right)\right\},
\end{equation}
and the "strong" space 
\begin{equation}\label{S_2}
\mathcal{S}_2=\left\{u\in{L}^\infty\left(0,T_f;H^1\left(\Omega\right)\right)\cap L^2\left(0,T_f;H^2\left(\Omega\right)\right),\;\;u_t\in{L}^2\left(0,T_f;L^2\left(\Omega\right)\right)\right\}.
\end{equation}
$\mathcal{W}_2$ and $\mathcal{S}_2$ are Banach spaces with the respective norms:

	$$\| u\|_{\mathcal{W}_2}=\|u\|_{L^\infty\left(0,T_f;L^2\left(\Omega\right)\right)}+\|u\|_{L^2\left(0,T_f;H^1\left(\Omega\right)\right)}+\|u_t\|_{L^2\left(0,T_f;\left(H^1\left(\Omega\right)\right)'\right)},$$

$$\| u\|_{\mathcal{S}_2}=\|u\|_{L^\infty\left(0,T_f;H^1\left(\Omega\right)\right)}+\|u\|_{L^2\left(0,T_f;H^2\left(\Omega\right)\right)}+\|u_t\|_{L^2\left(0,T_f;L^2\left(\Omega\right)\right)}.$$

Along the paper the constant $C$ will denote different constants which will appear in the work.
\\

In these circumstances, we will use the following result about existence and uniqueness of weak and strong solution for a linear parabolic problem, see for instance \cite{Evans_1998}.

\begin{teo}\label{teoexisunicweak}
	Given $\Omega\subseteq\mathbb{R}^3$ a bounded open set and $\left(0,T_f\right)$ a time interval for a fixed time $T_f>0$, we consider the following linear parabolic problem
	
	\begin{equation}\label{prob_parabolico}
	\left\{\begin{array}{lcr}
      u_t+L\;u=f&\quad\text{in}\;\;\left(0,T_f\right)\times\Omega,\\
      \\
     u\left(0,\cdot\right)=u_0&\quad\text{in}\;\; \Omega,\\
     \\
      \dfrac{\partial u}{\partial \text{n}}\Bigg\vert_{\partial\Omega}=0&\quad\text{on}\;\;\left(0,T_f\right)\times\partial\Omega
	\end{array}
	\right.
		\end{equation}
	where $f\in L^{2}\left(0,T_f;L^{2}\left(\Omega\right)\right)$, %$u_0\in L^2\left(\Omega\right)$,
	
	$$\displaystyle L\;u=-\sum_{i,j=1}^{3}\left(a_{ij}\left(t,x\right)\;u_{x_{j}}\right)_{x_{i}}+\sum_{i=1}^{3}b_{i}\left(t,x\right)u_{x_{i}}+c\left(t,x\right)u$$
	denotes a second-order partial elliptic differential operator with $a_{ij},b_i,c\in L^{\infty}\left(0,T_f;L^{\infty}\left(\Omega\right)\right)$, $a_{ij}=a_{ji}$ and there exists $C>0$ such that 
	$$\displaystyle\sum_{i,j=1}^{3}a_{ij}\left(t,x\right)p_i\; p_j\geq C\| p\|^2,\;\;\text{a.e.}\; \left(t,x\right)\in\left(0,T_f\right)\times\Omega,\;\; \forall p\in\mathbb{R}^3.$$
	
	Then:
	
	\begin{enumerate}[a)]
		\item For every $u_0\in L^2\left(\Omega\right)$, $\left(\ref{prob_parabolico}\right)$ has a unique weak solution $u\in \mathcal{W}_2$
		%$$u\in L^2\left(0,T_f;H^1\left(\Omega\right)\right)\cap L^\infty\left(0,T_f;L^2\left(\Omega\right)\right),\quad u_t\in L^2\left(0,T_f;H^1\left(\Omega\right)'\right)$$
		and
		%$$\parallel u\parallel_{{L}^\infty\left(0,T_f;L^2\left(\Omega\right)\right)}+\parallel u\parallel_{{L}^2\left(0,T_f;H^1\left(\Omega\right)\right)}+\parallel u_t\parallel_{{L}^2\left(0,T_f;\left(H^1\left(\Omega\right)\right)'\right)}
		$$\| u\|_{\mathcal{W}_2}\leq C\left(\| u_0\|_{L^2\left(\Omega\right)},\;\| f\|_{{L}^2\left(0,T_f;L^2\left(\Omega\right)\right)}\right).$$
		%with $\overline{C}>0$. 
		
		\item Assume $a_{ij}=\delta_{ij}$ (Kronecker delta) for $i,j=1,2,3$ hence
		$$\displaystyle L\;u=-\Delta\;u+\sum_{i=1}^{3}b_{i}\left(t,x\right)u_{x_{i}}+c\left(t,x\right)u.$$ 
		Then, for every $u_0\in H^1\left(\Omega\right)$, $\left(\ref{prob_parabolico}\right)$ has a unique strong solution $u\in \mathcal{S}_2$ 
		%$$u\in L^2\left(0,T_f;H^2\left(\Omega\right)\right)\cap L^\infty\left(0,T_f;H^1\left(\Omega\right)\right),\quad u_t\in L^2\left(0,T_f;L^2\left(\Omega\right)\right)$$
		and
		%$$\parallel u\parallel_{{L}^\infty\left(0,T_f;H^1\left(\Omega\right)\right)}+\parallel u\parallel_{{L}^2\left(0,T_f;H^2\left(\Omega\right)\right)}+\parallel u_t\parallel_{{L}^2\left(0,T_f;L^2\left(\Omega\right)\right)}
		$$\| u\|_{\mathcal{S}_2}\leq C\left(\| u_0\|_{H^1\left(\Omega\right)},\;\| f\|_{{L}^2\left(0,T_f;L^2\left(\Omega\right)\right)}\right).$$
		%with $\overline{C}>0$. 
	\end{enumerate}

\end{teo}

%In particular, if $b_i=0$ for $i=1,2,3$, $c=0$ and $a_{ij}=\delta_{ij}$ (Kronecker delta) for $i,j=1,2,3$, then the problem
%
%\begin{equation}\label{prob_parabolico2}
%\left\{\begin{array}{lcr}
%u_t-\Delta\;u=f&\quad\text{in}\;\;\Omega\times\left(0,T_f\right]\\
%\\
%u\left(0,\cdot\right)=u_0&\quad\text{in}\;\; \Omega\\
%\\
%\dfrac{\partial u}{\partial \text{n}}\Bigg\vert_{\partial\Omega}=0&\quad\text{on}\;\;\left(0,T_f\right)\times\partial\Omega
%\end{array}
%\right.
%\end{equation}
%\\
%where $f\in L^{2}\left(0,T_f;L^{2}\left(\Omega\right)\right)$ and 
%\\

Finally, we will use the following fixed point theorem to obtain solution of $\left(\ref{probOriginal}\right)$-$\left(\ref{condinicio}\right)$

\begin{teo}[\textit{Leray-Schauder's} theorem] \label{Leray}
	Let $V$ a Banach space, $\lambda\in\left[0,1\right]$ and $\mathcal{R}:V\rightarrow V$ a continuous and compact map such that for every $v\in V$ with $v=\lambda\;\mathcal{R}(v)$, satisfies that $\|v\|_V\leq \mathcal{C}$ with $\mathcal{C}>0$ independent of $\lambda\in\left[0,1\right]$. 
	Then, there exists a fixed point $v$ of $\mathcal{R}$. 
\end{teo}

\section{Existence of Solution of Problem $\boldsymbol{\left(\ref{probOriginal}\right)$-$\left(\ref{condinicio}\right)}$}
 We assume along the paper the following assumptions on the initial data
\begin{equation}\label{hipotesis0}
0\leq T_0(x),N_0(x),\Phi_0(x)\leq K,\;\; \text{  a.e.}\;x\in\Omega.
\end{equation}

First of all, we define the concept of solution used in the paper.
\begin{defi}[Weak-Strong solution of $\left(\ref{probOriginal}\right)$-$\left(\ref{condinicio}\right)$]\label{defi_Probl_Dif_No_Lineal}
	Given $T_0\in L^\infty\left(\Omega\right)$ and $N_0,\Phi_0\in H^1\left(\Omega\right)\cap L^\infty(\Omega)$ satisfying $\left(\ref{hipotesis0}\right)$, then $\left(T,N,\Phi\right)$ is called a weak-strong solution of problem $\left(\ref{probOriginal}\right)$-$\left(\ref{condinicio}\right)$ if 
	%$$0\leq T,\Phi\leq K,\quad0\leq N\leq C\left(T_f\right),$$
	$T\in \mathcal{W}_2$, 
	$N,\Phi\in{L}^\infty\left(0,T_f;H^1\left(\Omega\right)\right),\;\;N_t,\Phi_t\in{L}^2\left(0,T_f;L^2\left(\Omega\right)\right)$
	and they satisfy
	%satisfies the variational formulation
	$$\displaystyle \int_{0}^{T_f} \langle T_t,v\rangle_{\left(H^1\left(\Omega\right)\right)'}\;dt+\int_{0}^{T_f}\int_\Omega \left(\kappa_1\;P\left(\Phi,T\right)+\kappa_0\right)\nabla T\cdot\nabla v\;dx\;dt=\int_{0}^{T_f}\int_\Omega f_1\left(T,N,\Phi\right)\;v\;dx\;dt,$$
	$\forall v\in L^2\left(0,T_f;H^1\left(\Omega\right)\right)$ and %\quad\text{where the space}\;\;W_2\;\;\text{is defined by } \ref{W}$$
%	and
%	$N,\Phi\in{L}^\infty\left(0,T_f;H^1\left(\Omega\right)\right),\;\;N_t,\Phi_t\in{L}^2\left(0,T_f;L^2\left(\Omega\right)\right)$
%	verifying

$$\left\{\begin{array}{rl}
	\displaystyle N_t=f_2\left(T,N,\Phi\right)&\\
	&\quad\text{a.e. in}\;\; %\left(t,x\right)\in
	\left(0,T_f\right)\times\Omega\\
	\displaystyle \Phi_t=f_3\left(T,N,\Phi\right)&
\end{array}\right.$$
	%$$\displaystyle N_t=f_2\left(T,N,\Phi\right)$$
	%$$\displaystyle \Phi_t=f_3\left(T,N,\Phi\right)$$
	%a.e. $\left(t,x\right)\in\left(0,T_f\right)\times\Omega$ 
	and the boundary and initial conditions $\left(\ref{condifronte}\right)$ and $\left(\ref{condinicio}\right)$ are satisfied by $T$ and $(T,N,\Phi)$, respectively.
\end{defi}

\subsection{Truncated problem}
In order to obtain a solution of $\left(\ref{probOriginal}\right)$-$\left(\ref{condinicio}\right)$, we define the following truncated system of $\left(\ref{probOriginal}\right)$:
\begin{equation}\label{problin}
\left\{\begin{array}{ccl}
\dfrac{\partial T}{\partial t} -\nabla\cdot\left(\left(\kappa_1\;P\left(\Phi_+^K,T_+^K\right)+\kappa_0\right)\nabla\;T\right)& = & f_1\left(T_+^{K},N_+^{C\left(T_{f}\right)},\Phi_+^{K}\right)\\
&&\\
\dfrac{\partial N}{\partial t} & = &  f_2\left(T_+^{K},N_+^{C\left(T_{f}\right)},\Phi_+^{K}\right)\\
&&\\
\dfrac{\partial \Phi}{\partial t} & = & f_3\left(T_+^{K},N_+^{C\left(T_{f}\right)},\Phi_+^{K}\right)\\
\end{array}\right.
\end{equation}
\\
subject to $\left(\ref{condifronte}\right)$
%\begin{equation}\label{condifronte2}
%\dfrac{\partial T}{\partial n}\Bigg\vert_{\partial\Omega}=0,
%\end{equation}
%\\
and $\left(\ref{condinicio}\right)$.
%\begin{equation}\label{condinicio2}
%T\big\vert_{t=0}=T_0,\;\;N\big\vert_{t=0}=N_0,\;\;\Phi\big\vert_{t=0}=\Phi_0\;\;\text{in}\;\;\Omega
%\end{equation}
We have denoted $T_+^{K}=\min\left\{K,\max\left\{0,T\right\}\right\}$
%$$
%T_0^K=\max\left\{0,\min\left\{T,K\right\}\right\}=\left\{\begin{array}{cl}
%0&\quad si\;\;T\leq0 \\
%\\
%T&\quad si\;\;0\leq T\leq K\\
%\\
%K&\quad  si\;\;T\leq K
%\end{array} \right.
%$$
and similar to $ \Phi_+^{K}$ and $N_+^{C\left(T_f\right)}$ with $C\left(T_f\right)$ an exponential positive constant which depends on the final time $T_f>0$ and the carrying capacity $K$ (see \cite[Lemma 5]{Romero_2020}). 
\\

Once we obtain the existence of solution of the truncated problem $\left(\ref{problin}\right)$, we will prove that this solution is also a positive solution of $\left(\ref{probOriginal}\right)$-$\left(\ref{condinicio}\right)$, due to the following estimates for any possible weak-strong solution of $\left(\ref{problin}\right)$.

\begin{lem}\label{estimaciones}%[Pointwise a priori estimates]\label{estimaciones}
	%Under assumptions of Definition $\ref{defi_Probl_Dif_No_Lineal}$, 
	Any weak-strong solution $\left(T,N,\Phi\right)$ of $\left(\ref{problin}\right)$ with initial data satisfying $\left(\ref{hipotesis0}\right)$ satisfies the following bounds:
	
	\begin{enumerate}[a)]
		
	\item Pointwise estimates: \begin{equation}\label{cotas}
	0\leq T,\;\Phi\leq K\;\;\text{and}\;\; 0\leq N\leq C\left(T_{f}\right),\;\; \text{a.e. in}\;\; \left(0,T_f\right)\times\Omega.\end{equation}
%	\begin{equation}\label{cotas}
%	\left\{\begin{array}{l}\
%	0\leq T,\;\Phi\leq K,\\%&\text{ a.e. }\left(t,x\right)\in\left(0,T_f\right)\times\Omega,\\
%	\\
%	0\leq N\leq C\left(T_{f}\right),\\%&\text{ a.e. }\left(t,x\right)\in\left(0,T_f\right)\times\Omega,\\
%%	\\
%%	0\leq \Phi\leq K,\\%&\text{ a.e. }\left(t,x\right)\in\left(0,T_f\right)\times\Omega,
%	\end{array} \right.\end{equation}
	%a.e. in $\left(t,x\right)\in\left(0,T_f\right)\times\Omega$ 
	%where $C\left(T_f\right)$ is a positive constant depending exponentially on the final time $0<T_f<+\infty$.
	
	\item Energy estimates:
	$$\| T\|_{{L}^\infty\left(0,T_f;L^2\left(\Omega\right)\right)}+\| T\|_{{L}^2\left(0,T_f;H^1\left(\Omega\right)\right)}\leq C\left(\| T_0\|_{L^2\left(\Omega\right)},K,|\Omega|,T_f\right).$$
	
\end{enumerate}
\end{lem}

\begin{proof}
	
	\begin{enumerate}[a)]
		\item 
	Let $\left(T,N,\Phi\right)$ a weak-strong solution of $\left(\ref{problin}\right)$. Since one can rewrite  $f_1(T,N,\Phi)=T\;\widetilde{f_1}(T,N,\Phi)$, multiplying the first equation of $\left(\ref{problin}\right)$ by $T_-=\min\left\{T,0\right\}$ and integrating in $\Omega$, we get

$$\dfrac{1}{2}\dfrac{d}{dt}\int_{\Omega}(T_-)^2\;dx+\int_{\Omega}\left(\kappa_1\;P\left(\Phi_+^K,T_+^K\right)+\kappa_0\right)\mid\nabla T_-\mid^2\;dx=$$
$$=\int_{\Omega}T_-\;T_+^{K}\;\widetilde{f_1}\left(T_+^{K},N_+^{C\left(T_{f}\right)},\Phi_+^{K}\right)\;dx=0,\quad\text{  a.e. in}\;\left(0,T_f\right).$$

Hence, since $T_-\left(0,x\right)=0$, then $T_{-}\left(t,x\right)=0$ a.e. $\left(t,x\right)\in\left(0,T_f\right)\times\Omega$.
%We repeat the same argument for the other two equations of $\left(\ref{problin}\right)$ using that 
%$$\Phi_-\;f_3\left(T_+^{K},N_+^{C\left(T_{f}\right)},\Phi_+^{K}\right)=0\;\;\text{and}\;\;N_-\;f_2\left(T_+^{K},N_+^{C\left(T_{f}\right)},\Phi_+^{K}\right)\leq0.$$% . On the one hand, we have
To obtain the upper bound $T\le K$, we multiply the first equation of $\left(\ref{problin}\right)$ by $\left(T-K\right)_+=\max\left\{0,T-K\right\}$ and integrate in $\Omega$ %On the one hand,

$$\dfrac{1}{2}\dfrac{d}{dt}\int_{\Omega}\left(\left(T-K\right)_+\right)^2\;dx+\int_{\Omega}\left(\kappa_1\;P\left(\Phi_+^K,T_+^K\right)+\kappa_0\right)\mid\nabla \left(T-K\right)_+\mid^2 \;dx=$$
$$=\int_{\Omega}f_1\left(T_+^{K},N_+^{C\left(T_{f}\right)},\Phi_+^{K}\right)\left(T-K\right)_+\;dx,\quad\text{a.e. in}\;\left(0,T_f\right).$$
Since $f_1(T_+^{K},N_+^{C(T_{f})},\Phi_+^{K})\leq \rho\; T_+^K\;(1-\frac{T_+^{K}}{K})$, then $$f_1\left(T_+^{K},N_+^{C(T_{f})},\Phi_+^{K}\right)\;\left(T-K\right)_+\leq\rho\; T_+^K\;\left(1-\dfrac{T_+^{K}}{K}\right)\;\left(T-K\right)_+=0.$$
 %$$\displaystyle\dfrac{d}{dt}\int_{\Omega}\left(\left(T-K\right)_+\right)^2\;dx\leq0.$$ 
Since $\left(T\left(0,x\right)-K\right)_+=0$, then $\left(T\left(t,x\right)-K\right)_+=0$  a.e. $\left(t,x\right)\in\left(0,T_f\right)\times\Omega$. 
\\
%We repeat the same argument for the third equation of $\left(\ref{problin}\right)$ using that $\left(\Phi-K\right)_+\;f_3\left(T_+^{K},N_+^{C\left(T_{f}\right)},\Phi_+^{K}\right)\leq0$.
%\\
%Finally, from the second equation of $\left(\ref{problin}\right)$, we are going to obtain a bound in finite time with an exponential growth. Given a fixed final time $T_f>0$, for any $t\leq T_f$ and $x\in\overline{\Omega}$,
%$$\dfrac{\partial N}{\partial t} =\alpha\;B\left(\Phi_+^{K},T_+^{K}\right)+\delta\;T_+^{K}\;\Phi_+^{K}+ \beta\; N_+^{C\left(T_f\right)}\left(T_+^{K}+\Phi_+^{K}\right)\leq \mathcal{C}_1+\mathcal{C}_2\;N$$
%where $\mathcal{C}_1$ and $\mathcal{C}_2$ depend on $\alpha$, $\beta$, $\delta$ and $K$. Hence,
%\begin{equation}\label{*}
%N\left(t,x\right)\leq \dfrac{\mathcal{C}_1}{\mathcal{C}_2}\left(e^{\mathcal{C}_2\;t}-1\right)+e^{\mathcal{C}_2\;t}\;N_0\left(x\right)\leq C\left(T_f\right)=e^{\mathcal{C}_2\;T_f}\left(\dfrac{\mathcal{C}_1}{\mathcal{C}_2}+K\right).
%\end{equation}
%In particular, $C\left(T_f\right)>0$ is an upper bound with an exponential growth depending on the final time $\left(T_f\right)$, the bound for the initial data of necrosis $\left(\parallel N_0\left(x\right)\parallel_{L^\infty\left(\Omega\right)}\leq K\right)$ and the upper bounds for $T,\;\Phi\leq K$.

For the corresponding bounds of $N$ and $\Phi$ given in $\left(\ref{cotas}\right)$, we can use the same argument as in \cite[Lemma 5]{Romero_2020}.
\item Using the pointwise bounds for $\left(T,N,\Phi\right)$ given in a), multiplying the first equation of $\left(\ref{problin}\right)$ by $T$ and integrating in $\Omega$, we get

$$\dfrac{1}{2}\dfrac{d}{dt}\int_{\Omega}T^2\;dx+\int_{\Omega}\left(\kappa_1\;P\left(\Phi_+^K,T_+^K\right)+\kappa_0\right)\mid\nabla T\mid^2 \;dt\;dx=\int_{\Omega}T^2\;\widetilde{f_1}\left(T,N,\Phi\right)\leq$$
\vspace{0.3cm}
%$$+\dfrac{1}{2}\int_{\Omega}\left(T_0\left(x\right)\right)^2\leq$$
$$\leq\int_{\Omega}\rho T^2\;dx\leq\rho\;K^2\mid\Omega\mid.$$%-\dfrac{2}{3}\dfrac{\rho}{K}\int_{\Omega}T^3\;dx.$$

Integrating in time, the proof is finished.
	\end{enumerate}
\end{proof}
By Lemma \ref{estimaciones} a), for any $\left(T,N,\Phi\right)$ a weak-strong solution of $\left(\ref{problin}\right)$, we deduce that $T_+^{K}=T$, $N_+^{C\left(T_f\right)}=N$ and $\Phi_+^{K}=\Phi$ and then, $f_i\left(T_+^{K},N_+^{C\left(T_{f}\right)},\Phi_+^{K}\right)=f_i\left(T,N,\Phi\right)$ for $i=1,2,3$. Hence, we obtain the following crucial corollary
\begin{col}\label{solTrunc-solOrigin}
%Under hypotheses of Lemma $\ref{estimaciones}$, 
If $\left(T,N,\Phi\right)$ is a weak-strong solution of the truncated problem $\left(\ref{problin}\right)$, then $\left(T,N,\Phi\right)$ is also a weak-strong solution of %non truncated problem 
$\left(\ref{probOriginal}\right)$-$\left(\ref{condinicio}\right)$ and $\left(T,N,\Phi\right)$ satisfies the pointwise bounds $\left(\ref{cotas}\right)$.
\end{col}

%\begin{lem}\label{cotaL2}
%	Under hypotheses of Lemma $\ref{estimaciones}$, let $\left(T,N,\Phi\right)$ a weak-strong solution of $\left(\ref{problin}\right)$. Then,
%	
%	$$\parallel T\parallel_{{L}^\infty\left(0,T_f;L^2\left(\Omega\right)\right)}+\parallel T\parallel_{{L}^2\left(0,T_f;H^1\left(\Omega\right)\right)}\leq\mathcal{C}\left(\parallel T_0\parallel_{L^2\left(\Omega\right)},K,|\Omega|,T_f\right)$$
%\end{lem}
%
%\begin{proof}
%	Multiplying the first equation of $\left(\ref{problin}\right)$ by $T$ and integrating in $\Omega$, we get
%	
%	$$\dfrac{1}{2}\int_{\Omega}T^2\;dx+\int_{\Omega\times(0,T_f)}\left(\kappa_1\;P\left(\Phi,T\right)+\kappa_0\right)\mid\nabla T\mid^2 \;dt\;dx=\int_{\Omega\times(0,T_f)}T^2\;\widetilde{f_1}\left(T,N,\Phi\right)\leq$$
%	\vspace{0.3cm}
%	%$$+\dfrac{1}{2}\int_{\Omega}\left(T_0\left(x\right)\right)^2\leq$$
%	$$\leq\int_{\Omega}\rho T^2\left(1-\dfrac{T}{K}\right)\;dx\leq\dfrac{2}{3}\;\rho\;K^2\mid\Omega\mid-\dfrac{2}{3}\dfrac{\rho}{K}\int_{\Omega}T^3\;dx.$$
%	
%	Integrating in time, the proof is finished.
%
%\end{proof}

\subsection{Existence of Weak-Strong Solution of Problem $\boldsymbol{\left(\ref{problin}\right)}$}

\begin{teo}%[Existence of Weak-Strong solution of $\left(\ref{problin}\right)$] 
	There exists a weak-strong solution $\left(T,N,\Phi\right)$ of $\left(\ref{problin}\right)$ in the sense of Definition $\ref{defi_Probl_Dif_No_Lineal}$.
\end{teo}

\begin{obs}
	We can not guarantee the uniqueness of the weak-strong solution of $\left(\ref{probOriginal}\right)$-$\left(\ref{condinicio}\right)$ due to $T$ is not sufficiently regular by the influence of the nonlinear diffusion. % and the function $P\left(\Phi,T\right)$, which appears in the nonlinnear diffusion, is not locally lipschitz in a neighbourhood of $\left(0,0\right)$.
	Notice that, unlike in \cite{Romero_2020} and due to nonlinear diffusion, we are not able to prove that $T$ is a pointwise solution of $\left(\ref{probOriginal}\right)$. Consequently, the uniqueness of weak-strong solution is an open problem.
	%Notice that whereas in \cite{Romero_2020} we could prove existence of strong solution for the coupled system similar to $\left(\ref{probOriginal}\right)$ but with linear diffusion, now due to the PDE with nonlinear diffusion which satisfies the variable $T$, the strong regularity for $T$ is not possible to achieve.
\end{obs}
\begin{proof}

The proof of this theorem follows the next steps:

\begin{enumerate}
	\item Regularize the problem via an artificial diffusion with parameter $\epsilon>0$ for $\left(N,\Phi\right)$.
	\item Solve the regularized problem for any fixed value of $\epsilon$.
	\item Taking limits $\epsilon\rightarrow0$ to get solution of the non-regularized problem $\left(\ref{problin}\right)$.
\end{enumerate}

\subsubsection*{Step 1. Regularizing the problem $\boldsymbol{\left(\ref{problin}\right)}$}

We will study the following family of regularized problems related to system $\left(\ref{problin}\right)$. For any $\epsilon>0$. We define $\left(T_\epsilon,N_\epsilon,\Phi_\epsilon\right)$ as the solution of

\begin{equation}\label{probregul}
\left\{\begin{array}{ccl}
\dfrac{\partial T}{\partial t} -\nabla\cdot\left(\left(\kappa_1\;P\left(\Phi_+^K,T_+^K\right)+\kappa_0\right)\nabla\;T\right)& = & f_1\left(T_+^{K},N_+^{C\left(T_{f}\right)},\Phi_+^{K}\right)\\
&&\\
\dfrac{\partial N}{\partial t} -\epsilon\;\Delta\;N& = & f_2\left(T_+^{K},N_+^{C\left(T_{f}\right)},\Phi_+^{K}\right)\\
&&\\
\dfrac{\partial \Phi}{\partial t} -\epsilon\;\Delta\;\Phi& = &f_3\left(T_+^{K},N_+^{C\left(T_{f}\right)},\Phi_+^{K}\right)
\end{array}\right.\end{equation}
with the boundary conditions

\begin{equation}\label{condifronte3}
\dfrac{\partial T}{\partial n}\Bigg\vert_{\partial\Omega}=\epsilon\;\dfrac{\partial N}{\partial n}\Bigg\vert_{\partial\Omega}=\epsilon\;\dfrac{\partial \Phi}{\partial n}\Bigg\vert_{\partial\Omega}=0
\end{equation}
and the initial conditions
\begin{equation}\label{condinicio3}
T\big\vert_{t=0}=T_0,\;\;N\big\vert_{t=0}=N_0,\;\;,\Phi\big\vert_{t=0}=\Phi_0\;\;\text{in}\;\;\Omega.
\end{equation}
%where $T_0^{\varepsilon}\left(x\right)\in H^1\left(\Omega\right)\cap L^\infty(\Omega)$ is a regularization of $T_0$ such that $T_0^{\varepsilon}\rightarrow T_0$ in $L^2\left(\Omega\right)$ and $0\leq T_0^{\varepsilon}\leq K$. 
%\\

Now, we can define the kind of solution which we will obtain

\begin{defi}[Weak-Strong solution of $\left(\ref{probregul}\right)$-$\left(\ref{condinicio3}\right)$]\label{defi_Probl_Dif_No_Lineal2}
	Given $T_0\in L^\infty\left(\Omega\right)$ and $N_0,\Phi_0\in L^\infty\left(\Omega\right)\cap H^1\left(\Omega\right)$, then $\left(T,N,\Phi\right)$ is called a weak-strong solution of problem $\left(\ref{probregul}\right)$-$\left(\ref{condinicio3}\right)$ if 
	%$$0\leq T,\Phi\leq K,\quad0\leq N\leq C\left(T_f\right),$$
	$T\in \mathcal{W}_2$ and $N,\Phi\in \mathcal{S}_2$ and they satisfy %satisfying the variational formulation for $T$ equation
	$$\displaystyle \int_{0}^{T_f} \langle T_t,v\rangle_{\left(H^1\left(\Omega\right)\right)'}\;dt+\int_{0}^{T_f}\int_\Omega \left(\kappa_1\;P\left(\Phi_+^K,T_+^K\right)+\kappa_0\right)\nabla T\cdot\nabla v\;dx\;dt=\int_{0}^{T_f}\int_{\Omega} f_1(T_+^{K},N_+^{C\left(T_{f}\right)},\Phi_+^{K})\;v\;dx\;dt,$$
	$\forall v\in L^2\left(0,T_f;H^1\left(\Omega\right)\right)$, %\quad\text{where}\;\;W_2\;\;\text{is the space defined in}\;\left(\ref{W}\right),$$ 
	the PDE system 
	%\quad\text{where}\;\;S_2\;\;\text{is the space defined in}\;\left(\ref{S_2}\right),$$
	%the following variational formulation holds for $T$ equation
	%$$\displaystyle \langle T_t,v\rangle+\langle \left(\kappa_1\;P\left(\Phi,T\right)+\kappa_0\right)\nabla T,\nabla v\rangle=\langle f_1(T_+^{K},N_+^{C\left(T_{f}\right)},\Phi_+^{K}),v\rangle\;\;\text{a.e.}\;\; t\in\left(0,T_f\right)$$
	%$\forall v\in L^2\left(0,T_f;H^1\left(\Omega\right)'\right)$ and the $N$ and $\Phi$ equations 
	$$\left\{\begin{array}{rl}
		\displaystyle N_t-\epsilon\;\Delta\;N=f_2\left(T_+^{K},N_+^{C\left(T_{f}\right)},\Phi_+^{K}\right)&\\
		&\quad\text{a.e. in}\;\; %\left(t,x\right)\in
		\left(0,T_f\right)\times\Omega\\
	\displaystyle \Phi_t-\epsilon\;\Delta\;\Phi=f_3\left(T_+^{K},N_+^{C\left(T_{f}\right)},\Phi_+^{K}\right)&
	\end{array}\right.$$
	%are satisfied a.e. $\left(t,x\right)\in\left(0,T_f\right)\times\Omega$ 
	%are satisfied 
	and the boundary and initial conditions $\left(\ref{condifronte3}\right)$ and $\left(\ref{condinicio3}\right)$.
\end{defi}

\begin{obs}
	It is easy to prove for $T$ and $\Phi$ the estimates $\left(\ref{cotas}\right)$ following the same argument as in Lemma $\ref{estimaciones}$. For $N$, the following differential inequality is satisfied
	\begin{equation}\label{cotaN}
		\dfrac{\partial N}{\partial t}-\epsilon\;\Delta\;N\leq C_1+C_2\;N.
	\end{equation}

	Hence, $N\leq\widetilde{N}$ where $\widetilde{N}$ is the solution of the ODE equation
	$$	\dfrac{\partial\widetilde{N}}{\partial t}=C_1+C_2\;\widetilde{N}.$$
	
	Thus, any solution of $\left(\ref{probregul}\right)$-$\left(\ref{condinicio3}\right)$ satisfies that
	\begin{equation}\label{esti_sol_origi}
	0\leq T,\;\Phi\leq K,\quad 0\leq N\leq C\left(T_f\right),\quad\text{a.e. in}\;\;%\left(t,x\right)\in
	\left(0,T_f\right)\times\Omega.
	\end{equation}
	
%	Hence we can write $f_i\left(T_+,N_+,\Phi_+\right)=f_i\left(T,N,\Phi\right)$ for $i=1,2,3$.
\end{obs}

\begin{teo}[Existence of weak-strong solution of $\left(\ref{probregul}\right)$-$\left(\ref{condinicio3}\right)$]\label{Teo_existencia}
	There exists a weak-strong solution $\left(T,N,\Phi\right)$ of system $\left(\ref{probregul}\right)$-$\left(\ref{condinicio3}\right)$ in the sense of Definition $\ref{defi_Probl_Dif_No_Lineal2}$.

\end{teo}

%\subsubsection{Existence of weak-strong solution of problem $\boldsymbol{\left(\ref{probregul}\right)$-$\left(\ref{condinicio3}\right)}$}
\subsubsection*{Proof of Theorem \ref{Teo_existencia}}
%\begin{proof}

We define the following operator

$$
\begin{array}{cccccc}
\mathbf{R}:& \left(L^2\left(0,T_f;L^2\left(\Omega\right)\right)\right)^3&\longrightarrow& \left(L^2\left(0,T_f;L^2\left(\Omega\right)\right)\right)^3\\
\\
&\left(\widetilde{T},\widetilde{N},\widetilde{\Phi}\right)&\longrightarrow&\left(T,N,\Phi\right)=\mathbf{R}\left(\widetilde{T},\widetilde{N},\widetilde{\Phi}\right)
\end{array}
$$
where $\left(T,N,\Phi\right)$ is the weak-strong solution of the linear and decoupled problem

\begin{equation}\label{probregul2}
\left\{\begin{array}{ccl}
\dfrac{\partial T}{\partial t} -\nabla\cdot\left(\left(\kappa_1\;P\left(\widetilde{\Phi}_+^K,\widetilde{T}_+^K\right)+\kappa_0\right)\nabla\;T\right)& = & f_1\left(\widetilde{T}_+^K,\widetilde{N}_+^{C\left(T_f\right)},\widetilde{\Phi}_+^K\right)\\
&&\\
\dfrac{\partial N}{\partial t} -\epsilon\;\Delta\;N& = & f_2\left(\widetilde{T}_+^K,\widetilde{N}_+^{C\left(T_f\right)},\widetilde{\Phi}_+^K\right)\\
&&\\
\dfrac{\partial \Phi}{\partial t} -\epsilon\;\Delta\;\Phi& = &f_3\left(\widetilde{T}_+^K,\widetilde{N}_+^{C\left(T_f\right)},\widetilde{\Phi}_+^K\right)
\end{array}\right.\end{equation}
subject to $\left(\ref{condifronte3}\right)$ and $\left(\ref{condinicio3}\right)$. Observe that thanks to $\left(\ref{esti_sol_origi}\right)$, a weak-strong solution of $\left(\ref{probregul}\right)$-$\left(\ref{condinicio3}\right)$ is a fixed point of $\mathbf{R}$. Therefore, we look for a fixed point of $\mathbf{R}$ using Leray-Schauder's theorem \ref{Leray}.
%Then, a weak-strong solution of system $\left(\ref{probregul}\right)$-$\left(\ref{condinicio3}\right)$ is a fixed point of $\mathbf{R}$. Therefore, in order to prove the existence of weak-strong solution of problem $\left(\ref{probregul}\right)$-$\left(\ref{condinicio3}\right)$, we will prove the hypotheses of Theorem \ref{Leray}.

\begin{lem}\label{buenadefi}
	The operator  $\mathbf{R}$ is well defined from $\left(L^2\left(0,T_f;L^2\left(\Omega\right)\right)\right)^3$ to itself. 
\end{lem}
\begin{proof}
Using $\left(\ref{cota_funcion_P}\right)$, 
\begin{equation}\label{acotacion_P}
0\leq P\left(\left(\widetilde{\Phi}\left(t,x\right)\right)_+^K,\left(\widetilde{T}\left(t,x\right)\right)_+^K\right)\leq1,\quad \text{a.e.}\;\;(t,x)\in \left(0,T_f\right)\times \Omega.
\end{equation}

On the other hand, one has
\begin{equation}\label{acotacion_f}
\Big\| f_i\left(\widetilde{T}_+^K,\widetilde{N}_+^{C\left(T_f\right)},\widetilde{\Phi}_+^K\right)\Big\|_{L^\infty\left(0,T_f;L^\infty\left(\Omega\right)\right)}\leq\mathcal{C}_i\quad\forall\;i=1,2,3,
\end{equation}
with $C_i$ independent of $\widetilde{T}$, $\widetilde{N}$ and $\widetilde{\Phi}$. Hence, we can apply Theorem $\ref{teoexisunicweak}$ to conclude that there exists a unique weak solution of $\left(\ref{probregul2}\right)$ with the following regularity
$$\left(T,N,\Phi\right)\in \mathcal{W}_2\times \mathcal{S}_2\times \mathcal{S}_2.$$
% \left(L^\infty\left(0,T_f;L^2\left(\Omega\right)\right)\cap L^2\left(0,T_f;H^1\left(\Omega\right)\right)\right)^3,\quad \left(T_t,N_t,\Phi_t\right)\in \left(L^2\left(0,T_f;H^1\left(\Omega\right)'\right)\right)^3$$ 

In particular,
$$\left(T,N,\Phi\right)\in\left(L^2\left(0,T_f;L^2\left(\Omega\right)\right)\right)^3.$$

\end{proof}

\begin{lem}\label{compacto}
	The operator $\mathbf{R}$ is compact from $\left(L^2\left(0,T_f;L^2\left(\Omega\right)\right)\right)^3$ to itself.
\end{lem}
\begin{proof}

	Let $\left(\widetilde{T},\widetilde{N},\widetilde{\Phi}\right)\in \left(L^2\left(0,T_f;L^2\left(\Omega\right)\right)\right)^3$. Then, applying the same argument of Lemma $\ref{buenadefi}$ and estimates $\left(\ref{acotacion_P}\right)$ and $\left(\ref{acotacion_f}\right)$, we prove that there exists a unique $\left(T,N,\Phi\right)=\mathbf{R}\left(\widetilde{T},\widetilde{N},\widetilde{\Phi}\right)$ such that $\left(T,N,\Phi\right)$ is solution of $\left(\ref{probregul2}\right)$ with the following estimates:
	
	\begin{equation}\label{cotasW2S2}
	\begin{array}{c}
	\Big\| T\Big\|_{\mathcal{W}_2}\leq %C\left(\Big\| T_0\Big\|_{L^2\left(\Omega\right)},\Big\| f_1\left(\widetilde{T}_+^K,\widetilde{N}_+^{C\left(T_f\right)},\widetilde{\Phi}_+^K\right)\Big\|_{{L}^2\left(0,T_f;L^2\left(\Omega\right)\right)}\right)\leq 
	C\left(\Big\| T_0\Big\|_{L^2\left(\Omega\right)},K,C\left(T_f\right)\right),\\
	\\
	\Big\| N\Big\|_{\mathcal{S}_2}\leq %C\left(\Big\| N_0\Big\|_{H^1\left(\Omega\right)},\Big\| f_2\left(\widetilde{T}_+^K,\widetilde{N}_+^{C\left(T_f\right)},\widetilde{\Phi}_+^K\right)\Big\|_{{L}^2\left(0,T_f;L^2\left(\Omega\right)\right)}\right)\leq 
	C\left(\Big\| N_0\Big\|_{L^2\left(\Omega\right)},K,C\left(T_f\right)\right),\\
	\\
	\Big\| \Phi\Big\|_{\mathcal{S}_2}\leq %C\left(\Big\| \Phi_0\Big\|_{H^1\left(\Omega\right)},\Big\| f_3\left(\widetilde{T}_+^K,\widetilde{N}_+^{C\left(T_f\right)},\widetilde{\Phi}_+^K\right)\Big\|_{{L}^2\left(0,T_f;L^2\left(\Omega\right)\right)}\right)\leq 
	C\left(\Big\| \Phi_0\Big\|_{L^2\left(\Omega\right)},K,C\left(T_f\right)\right).
	\end{array}
	\end{equation}
	
%	$$\parallel T\parallel_{{L}^\infty\left(0,T_f;H^1\left(\Omega\right)\right)}\leq\mathcal{C}_1\left(\parallel T_0^{\varepsilon}\parallel_{H^1\left(\Omega\right)},\parallel f_1\parallel_{{L}^2\left(0,T_f;L^2\left(\Omega\right)\right)}\right)$$
%	$$\parallel N\parallel_{{L}^\infty\left(0,T_f;H^1\left(\Omega\right)\right)}\leq\mathcal{C}_2\left(\parallel N_0\parallel_{H^1\left(\Omega\right)},\parallel f_2\parallel_{{L}^2\left(0,T_f;L^2\left(\Omega\right)\right)}\right)$$
%	$$\parallel \Phi\parallel_{{L}^\infty\left(0,T_f;H^1\left(\Omega\right)\right)}\leq\mathcal{C}_3\left(\parallel \Phi_0\parallel_{H^1\left(\Omega\right)},\parallel f_3\parallel_{{L}^2\left(0,T_f;L^2\left(\Omega\right)\right)}\right)$$
%with $\overline{C}_i>0$ for $i=1,2,3$. 
Hence, $\left(T,N,\Phi\right)$ is bounded in $\mathcal{W}_2\times \mathcal{S}_2\times \mathcal{S}_2$. %where $W_2$ and $S_2$ are the spaces defined in $\left(\ref{W}\right)$ and $\left(\ref{S_2}\right)$, respectively. 
Applying Aubin-Lions Theorem, we conclude that the embedding 
$$\mathcal{W}_2\times \mathcal{S}_2\times \mathcal{S}_2\hookrightarrow \left(L^2\left(0,T_f;L^2\left(\Omega\right)\right)\right)^3$$
 is compact. Thus, $\mathbf{R}$ is compact from $\left(L^2\left(0,T_f;L^2\left(\Omega\right)\right)\right)^3$ to itself.

\end{proof}

\begin{lem}\label{continuo}
	The operator $\mathbf{R}: \left(L^2\left(0,T_f;L^2\left(\Omega\right)\right)\right)^3\longrightarrow \left(L^2\left(0,T_f;L^2\left(\Omega\right)\right)\right)^3$ is continuous.
\end{lem} 
\begin{proof}

Given 
\begin{equation}\label{R_cont}
\left(\widetilde{T}_n,\widetilde{N}_n,\widetilde{\Phi}_n\right)\rightarrow\left(\widetilde{T},\widetilde{N},\widetilde{\Phi}\right) \in  \left(L^2\left(\left(0,T_f\right);L^2\left(\Omega\right)\right)\right)^3,
\end{equation} 
we are going to check that
$$\left(T_n,N_n,\Phi_n\right):=\mathbf{R}\left(\widetilde{T}_n,\widetilde{N}_n,\widetilde{\Phi}_n\right)\rightarrow\mathbf{R}\left(\widetilde{T},\widetilde{N},\widetilde{\Phi}\right):=\left(T,N,\Phi\right)\;\;\text{in}\;\;  \left(L^2\left(\left(0,T_f\right);L^2\left(\Omega\right)\right)\right)^3.$$

Since $\left(T_n,N_n,\Phi_n\right)=\mathbf{R}\left(\widetilde{T}_n,\widetilde{N}_n,\widetilde{\Phi}_n\right)$ is solution of $\left(\ref{probregul2}\right)$, from %with a similar argument of 
$\left(\ref{cotasW2S2}\right)$ we obtain  % Lemma $\ref{compacto}$ 
that $\left(T_n,N_n,\Phi_n\right)$ is bounded in $\mathcal{W}_2\times \mathcal{S}_2\times \mathcal{S}_2$.
\\

By Aubin-Lions Theorem the embeddings $\mathcal{W}_2\hookrightarrow L^2\left(0,T_f;L^2\left(\Omega\right)\right)$ and $\mathcal{S}_2\hookrightarrow L^2\left(0,T_f;H^1\left(\Omega\right)\right)$ are compact, hence there exists a subsequence $\left(T_{n_k},N_{n_k},\Phi_{n_k}\right)\in \mathcal{W}_2\times \mathcal{S}_2\times \mathcal{S}_2$ and a limit $\left(T^*,N^*,\Phi^*\right)\in \mathcal{W}_2\times \mathcal{S}_2\times \mathcal{S}_2$ such that

$$\begin{array}{rcl}
\mathbf{R}\left(\widetilde{T}_{n_k},\widetilde{N}_{n_k},\widetilde{\Phi}_{n_k}\right)=\left(T_{n_k},N_{n_k},\Phi_{n_k}\right)&\stackbin[k\rightarrow\infty]{}{\;\;\Large\rightharpoonup\;\;}&\left(T^*,N^*,\Phi^*\right)\;\;\text{weakly in}\;\; \mathcal{W}_2\times \mathcal{S}_2\times \mathcal{S}_2,\\
\\
\mathbf{R}\left(\widetilde{T}_{n_k},\widetilde{N}_{n_k},\widetilde{\Phi}_{n_k}\right)=\left(T_{n_k},N_{n_k},\Phi_{n_k}\right)&\stackbin[k\rightarrow\infty]{}{\;\;\Large\rightarrow\;\;}&\left(T^*,N^*,\Phi^*\right)\;\;\text{strongly in}\;\;\left(L^2\left(0,T_f,L^2\left(\Omega\right)\right)\right)^3
\end{array}$$
and
%and strongly in $\left(L^2\left(0,T_f,L^2\left(\Omega\right)\right)\right)^3$ and
$$\left(N_{n_k},\Phi_{n_k}\right)\stackbin[k\rightarrow\infty]{}{\;\;\Large\rightarrow\;\;}\left(N^*,\Phi^*\right)\;\;\text{strongly in}\;\; \left(L^2\left(0,T_f,H^1\left(\Omega\right)\right)\right)^2.$$

 In particular,
$$\left(\left(T_{n_k}\right)_t,\left(N_{n_k}\right)_t,\left(\Phi_{n_k}\right)_t\right)\stackbin[k\rightarrow\infty]{}{\;\;\Large\rightharpoonup\;\;}\left(\left(T^*\right)_t,\left(N^*\right)_t,\left(\Phi^*\right)_t\right)\text{ weakly in } \left(L^2\left(0,T_f;\left(H^1\left(\Omega\right)\right)'\right)\right)^3,$$
$$\left(\left(N_{n_k}\right)_t,\left(\Phi_{n_k}\right)_t\right)\stackbin[k\rightarrow\infty]{}{\;\;\Large\rightharpoonup\;\;}\left(\left(N^*\right)_t,\left(\Phi^*\right)_t\right)\text{ weakly in } \left(L^2\left(0,T_f;L^2\left(\Omega\right)\right)\right)^2,$$
and

$$\left(\nabla T_{n_k},\nabla N_{n_k},\nabla \Phi_{n_k}\right)\stackbin[k\rightarrow\infty]{}{\;\;\Large\rightharpoonup\;\;}\left(\nabla T^*,\nabla N^*,\nabla\Phi^*\right)\text{ weakly in } \left(L^2\left(0,T_f;L^2\left(\Omega\right)\right)\right)^3.$$

Using the pointwise convergence $$\left(\widetilde{T}_n\left(t,x\right),\widetilde{N}_n\left(t,x\right),\widetilde{\Phi}_n\left(t,x\right)\right)\to\left(\widetilde{T}\left(t,x\right),\widetilde{N}\left(t,x\right),\widetilde{\Phi}\left(t,x\right)\right),\quad\text{a.e.}\;\;\left(t,x\right)\in \left(0,T_f\right)\times \Omega$$ 
one also has 
$$\left(\left(\widetilde{T}_n\left(t,x\right)\right)_+^K,\left(\widetilde{N}_n\left(t,x\right)\right)_+^{C\left(T_f\right)},\left(\widetilde{\Phi}_n\left(t,x\right)\right)_+^K\right)\to\left(\left(\widetilde{T}\left(t,x\right)\right)_+^K,\left(\widetilde{N}\left(t,x\right)\right)_+^{C\left(T_f\right)},\left(\widetilde{\Phi}\left(t,x\right)\right)_+^K\right)$$
a.e. $\left(t,x\right)\in \left(0,T_f\right)\times \Omega$.
\\

Since $\Big\|P\left(\widetilde{\Phi}_+^K,\widetilde{T}_+^K\right)\Big\|_{L^\infty\left(0,T_f;L^\infty\left(\Omega\right)\right)}\leq1$ and $P\left(\widetilde{\Phi},\widetilde{T}\right)$ is continuous in $\mathbb{R}^2$,
%by $\left(\ref{R_cont}\right)$, 
applying dominated convergence Theorem, we can deduce that
\begin{equation}\label{conver_P}
P\left(\left(\widetilde{\Phi}_{n_k}\right)_+^K,\left(\widetilde{T}_{n_k}\right)_+^K\right)\stackbin[k\rightarrow\infty]{}{\longrightarrow} P\left(\widetilde{\Phi}_+^K,\widetilde{T}_+^K\right)\text{ in } L^p\left(0,T_f;L^p\left(\Omega\right)\right),\;\;\forall p<\infty.
\end{equation}

Since $\Big\| f_1\left(\widetilde{T}_+^K,\widetilde{N}_+^{C\left(T_f\right)},\widetilde{\Phi}_+^K\right)\Big\|_{L^\infty\left(0,T_f;L^\infty\left(\Omega\right)\right)}\leq C$ and $\left(\ref{R_cont}\right)$, applying dominated convergence Theorem, we deduce that
$$f_i\left(\left(\widetilde{T}_{n_k}\right)_+^K,\left(\widetilde{N}_{n_k}\right)_+^{C\left(T_f\right)},\left(\widetilde{\Phi}_{n_k}\right)_+^{K}\right)\stackbin[k\rightarrow\infty]{}{\longrightarrow}f_i\left(\widetilde{T}_+^{K},\widetilde{N}_+^{C\left(T_f\right)},\widetilde{\Phi}_+^K\right)$$
in $L^p\left(0,T_f;L^p\left(\Omega\right)\right)$ for all $p<\infty$ and for $i=1,2,3$. 
\\

On the other hand, $\nabla T_{n_k}\stackbin[k\rightarrow\infty]{}{\;\;\Large\rightharpoonup\;\;}\nabla T^*$ weakly in $L^2\left(0,T_f;L^2\left(\Omega\right)\right)$. Thus, we obtain \newline $P\left(\left(\widetilde{\Phi}_{n_k}\right)_+^K,\left(\widetilde{T}_{n_k}\right)_+^K\right)\;\nabla T_{n_k}$ is bounded in $L^2\left(0,T_f;L^2\left(\Omega\right)\right)$. Consequently,

$$P\left(\left(\widetilde{\Phi}_{n_k}\right)_+^K,\left(\widetilde{T}_{n_k}\right)_+^K\right)\;\nabla T_{n_k}\stackbin[k\rightarrow\infty]{}{\;\;\Large\rightharpoonup\;\;}P\left(\widetilde{\Phi}_+^K,\widetilde{T}_+^K\right)\;\nabla T^*\text{ weakly in } \left(L^2\left(0,T_f;L^2\left(\Omega\right)\right)\right)^3.$$
\\

Thus, passing to the limit in the problem satisfied by $(T_{n_k},N_{n_k},\Phi_{n_k})$, we have that $\left(T^*,N^*,\Phi^*\right)=\mathbf{R}\left(\widetilde{T},\widetilde{N},\widetilde{\Phi}\right)$ and since the solution $\mathbf{R}\left(\widetilde{T},\widetilde{N},\widetilde{\Phi}\right)$ of $\left(\ref{probregul2}\right)$ is unique, we conclude the convergence of the whole sequence, that is,

$$\mathbf{R}\left(\widetilde{T}_n,\widetilde{N}_n,\widetilde{\Phi}_n\right)=\left(T_n,N_n,\Phi_n\right)\rightarrow\mathbf{R}\left(\widetilde{T},\widetilde{N},\widetilde{\Phi}\right)=\left(T,N,\Phi\right)\text{ in }\left(L^2\left(0,T_f;L^2\left(\Omega\right)\right)\right)^3.$$
\end{proof}

Now we introduce a notation for vectorial norms. Given a space $X$ and $f,\;g,\;h\in X$, 

$$\big\|f,g,h\big\|_{X}^2=\big\|f\big\|_{X}^2+\big\|g\big\|_{X}^2+\big\|h\big\|_{X}^2.$$

\begin{lem}\label{acotacion}
	If $\left(T,N,\Phi\right)=\lambda\;\mathbf{R}\left(T,N,\Phi\right)$ for some $\lambda\in\left[0,1\right]$, then $$\big\|T,N,\Phi\big\|_{L^2\left(0,T_f;L^2\left(\Omega\right)\right)}\leq C$$
	with $C>0$ independent of $\lambda\in\left[0,1\right]$.% where $\big\|T,N,\Phi\big\|_{L^2\left(0,T_f;L^2\left(\Omega\right)\right)}=\big\|T\big\|_{L^2\left(0,T_f;L^2\left(\Omega\right)\right)}+\big\|N\big\|_{L^2\left(0,T_f;L^2\left(\Omega\right)\right)}+\big\|\Phi\big\|_{L^2\left(0,T_f;L^2\left(\Omega\right)\right)}$.
%	$$\big\|T\big\|_{L^2\left(0,T_f;L^2\left(\Omega\right)\right)}\leq\mathcal{C}_1$$
%	$$\big\|N\big\|_{L^2\left(0,T_f;L^2\left(\Omega\right)\right)}\leq\mathcal{C}_2,$$
%	$$\big\|\Phi\big\|_{L^2\left(0,T_f;L^2\left(\Omega\right)\right)}\leq\mathcal{C}_3,$$
%	
%	with $\mathcal{C}_i>0$ for $i=1,2,3$ independent of $\lambda\in\left[0,1\right]$.
	
%	$$\big\|T\big\|_{L^2\left(0,T_f;L^2\left(\Omega\right)\right)}\leq\mathcal{C}_1,$$
%	$$\big\|N\big\|_{L^2\left(0,T_f;L^2\left(\Omega\right)\right)}\leq\mathcal{C}_2,$$
%	$$\big\|\Phi\big\|_{L^2\left(0,T_f;L^2\left(\Omega\right)\right)}\leq\mathcal{C}_3,$$
%	\\
%	with $\mathcal{C}_i>0$ independent of $\lambda$ for $i=1,2,3$.
	%The fixed points of $\lambda\;\mathbf{R}$ are bounded in $\left(L^2\left(0,T_f;L^2\left(\Omega\right)\right)\right)^3$, independently of $\lambda\in\left[0,1\right]$.
\end{lem}
\begin{proof}
	
For $\lambda=0$ is trivial, hence we suppose $\lambda\in\left(0,1\right]$. Let $\left(T,N,\Phi\right)\in L^2\left(0,T_f;L^2\left(\Omega\right)\right)$ such that $\left(T,N,\Phi\right)=\lambda\;\mathbf{R}\left(T,N,\Phi\right)$. Then $\left(T,N,\Phi\right)$ is solution of a system similar to $\left(\ref{probregul}\right)$-$\left(\ref{condinicio3}\right)$ with $\lambda$ multiplying in the right hand side. Therefore, we can follow the same argument that in Lemma $\ref{estimaciones}$ to obtain that $0\leq T,\;\Phi\leq K$ and $0\leq N\leq C\left(T_f\right)$ a.e. $\left(0,T_f\right)\times\Omega$.
\\

Thus, $\left(T,N,\Phi\right)$ is bounded in $\left(L^\infty\left(0,T_f;L^\infty\left(\Omega\right)\right)\right)^3$ and also in $\left(L^2\left(0,T_f;L^2\left(\Omega\right)\right)\right)^3$ independently of $\lambda\in\left[0,1\right]$.

\end{proof}

Finally, from Lemmas $\ref{compacto}$, $\ref{continuo}$ and $\ref{acotacion}$ the operator $\mathbf{R}$ satisfies the hypotheses of Theorem $\ref{Leray}$. Thus, we conclude that the map $\mathbf{R}$ has a fixed point $\left(T_\epsilon,N_\epsilon,\Phi_\epsilon\right)$ which is a weak-strong solution of problem $\left(\ref{probregul}\right)$-$\left(\ref{condinicio3}\right)$.
%\end{proof}

\subsubsection*{Step 2. $\boldsymbol{\epsilon}$-independent estimates}

Once we have proved the existence of weak-strong solution for the regularized problem $\left(\ref{probregul}\right)$-$\left(\ref{condinicio3}\right)$, we are going to take $\epsilon\rightarrow0$ in order to obtain a weak-strong solution of problem $\left(\ref{problin}\right)$.
\\

We can deduce the following $\epsilon$ independent estimates for the solution  $\left(T_\epsilon,N_\epsilon,\Phi_\epsilon\right)$:

\begin{itemize}
	\item Following the proof of Lemma $\ref{estimaciones}$, we can obtain that

\begin{equation}\label{estimacionesepsilon}
0\leq T_\epsilon,\;\Phi_{\epsilon}\leq K \text{ and }0\leq N_\epsilon\leq C\left(T_f\right),\quad\text{a.e. in}\;\;\left(0,T_f\right)\times\Omega.
\end{equation}

\item Following the proof of Lemma $\ref{estimaciones}$ b) for the problems satisfied by $N_\epsilon$ and $\Phi_\epsilon$, we have the bounds 

$$\| N_\epsilon,\Phi_{\epsilon}\|_{{L}^\infty\left(0,T_f;L^2\left(\Omega\right)\right)}^{2}+\|\nabla\left(\sqrt{\epsilon}\; N_\epsilon\right),\nabla\left(\sqrt{\epsilon}\; \Phi_\epsilon\right)\|_{{L}^2\left(0,T_f;L^2\left(\Omega\right)\right)}^{2}\leq C\left(\| N_0,\Phi_0\|_{L^2\left(\Omega\right)},|\Omega|,K,T_f\right).$$

%$$\parallel N_\epsilon\parallel_{{L}^\infty\left(0,T_f;L^2\left(\Omega\right)\right)}^{2}+\parallel\nabla\left(\sqrt{\epsilon}\; N_\epsilon\right)\parallel_{{L}^2\left(0,T_f;L^2\left(\Omega\right)\right)}^{2}\leq\mathcal{C}\left(\parallel N_0\parallel_{L^2\left(\Omega\right)},|\Omega|,K,T_f\right),$$

%$$\parallel\Phi_\epsilon\parallel_{{L}^\infty\left(0,T_f;L^2\left(\Omega\right)\right)}^{2}+\parallel\nabla\left(\sqrt{\epsilon}\; \Phi_\epsilon\right)\parallel_{{L}^2\left(0,T_f;L^2\left(\Omega\right)\right)}^{2}\;\leq\mathcal{C}\left(\parallel \Phi_0\parallel_{L^2\left(\Omega\right)},|\Omega|,K,T_f\right).$$

Hence, 

\begin{equation}\label{estimacionesepsilon2}
%\left\{
%\begin{array}{c}
%
%\left(N_\epsilon,\Phi_\epsilon\right)\text{ is bounded in } L^\infty\left(0,T_f;L^2\left(\Omega\right)\right),
%\\
%\\
\left(\sqrt{\epsilon}\;\nabla N_\epsilon,\;\sqrt{\epsilon}\;\nabla\Phi_\epsilon\right)\text{ is bounded in } L^2\left(0,T_f;L^2\left(\Omega\right)\right).
%\end{array}\right.
\end{equation}

\item From Lemma $\ref{estimaciones}$ b), we obtain that

$$T_\epsilon\;\;\text{is bounded in}\;\;L^\infty\left(0,T_f;L^2\left(\Omega\right)\right)\cap L^2\left(0,T_f;H^1\left(\Omega\right)\right).$$

\item From $\left(\ref{estimacionesepsilon2}\right)$, we obtain the bounds

\begin{equation}\label{estimacionesepsilon3}
\left(\sqrt{\epsilon}\;\Delta\;\Phi_\epsilon,\; \sqrt{\epsilon}\;\Delta\;N_\epsilon\right)\text{ in } L^2\left(0,T_f;\left(H^1\left(\Omega\right)\right)'\right).
\end{equation}

\item Moreover, from $\left(\ref{probregul}\right)$ %bounding the time derivatives from the $\left(\ref{problin}\right)$ problem%using that $\left(T_\epsilon\right)_t,\;\left(N_\epsilon\right)_t,\;\left(\Phi_\epsilon\right)_t$ are bounded in $L^2\left(0,T_f;L^2\left(\Omega\right)\right)$, 
we obtain that

\begin{equation}\label{estimacionesepsilon4}
\left\{
\begin{array}{c}
\left(T_\epsilon\right)_t\text{ is bounded in } L^2\left(0,T_f;\left(H^1\left(\Omega\right)\right)'\right),\\
\\
\left(N_\epsilon\right)_t,\;\left(\Phi_\epsilon\right)_t\text{ are bounded in } %L^2\left(0,T_f;\left(H^1\left(\Omega\right)\right)'\right)\cap 
L^\infty\left(0,T_f;L^\infty\left(\Omega\right)\right)
\end{array}\right.
\end{equation}

because $f_i\left(\left(\widetilde{T}_\epsilon\right)_+^K,\left(\widetilde{N}_\epsilon\right)_+^{C\left(T_f\right)},\left(\widetilde{\Phi}_\epsilon\right)_+^K\right)$ is bounded in $L^\infty\left(0,T_f;L^\infty\left(\Omega\right)\right)$ for $i=1,2,3$.
\end{itemize}

We will see the following additional estimate.

\begin{lem}\label{lemaLinfyh1}
	Assume $N_0,\;\Phi_0\in H^1\left(\Omega\right)$, then $N_\epsilon$, $\Phi_\epsilon$ are bounded in $L^\infty\left(0,T_f;H^1\left(\Omega\right)\right)$.
\end{lem}

\begin{proof}
	We only make the proof for $N_\epsilon$ because for $\Phi_{\epsilon}$ is similar. Multiplying the $N_\epsilon$ equation by $-\Delta N_\epsilon\in L^2\left(0,T_f;L^2\left(\Omega\right)\right)$ and integrating in $\Omega$, we obtain

\begin{equation}\label{adding1}
\dfrac{1}{2}\;\dfrac{d}{dt}\|\nabla N_\epsilon\|_{L^2\left(\Omega\right)}^2+\epsilon\;\|\Delta N_\epsilon\|_{L^2\left(\Omega\right)}^2\;dx
=\int_{\Omega}f_2\left(\left(T_\epsilon\right)_+^K,\left(N_\epsilon\right)_+^{C\left(T_f\right)},\left(\Phi_\epsilon\right)_+^K\right)\left(-\Delta N_\epsilon\right)\;dx
\end{equation}
where the right hand side of $\left(\ref{adding1}\right)$ after integrating by parts can be bounded as follows

\begin{equation}\label{acotacion_laplaciano}
\begin{array}{c}
\displaystyle\int_{\Omega}f_2\left(\left(T_\epsilon\right)_+^K,\left(N_\epsilon\right)_+^{C\left(T_f\right)},\left(\Phi_\epsilon\right)_+^K\right)\left(-\Delta N_\epsilon\right)\;dx
%
%$$\displaystyle=\int_{\Omega}\dfrac{\partial f_2}{\partial T}\left(\left(T_\epsilon\right)_0^K,\left(N_\epsilon\right)_0^{C_{T_f}},\left(\Phi_\epsilon\right)_0^K\right)\; \chi_{\left\{0\leq T_\epsilon\leq K\right\}}\;\nabla T_\epsilon\cdot\nabla N_\epsilon\;dx+$$
%
%$$\displaystyle+\int_{\Omega}\dfrac{\partial f_2}{\partial N}\left(\left(T_\epsilon\right)_0^K,\left(N_\epsilon\right)_0^{C_{T_f}},\left(\Phi_\epsilon\right)_0^K\right)\; \chi_{\left\{0\leq N_\epsilon\leq C_{T_f}\right\}}|\nabla \left(N_\epsilon\right)|^2\;dx+$$
%
%$$\displaystyle+\int_{\Omega}\dfrac{\partial f_2}{\partial \Phi}\left(\left(T_\epsilon\right)_0^K,\left(N_\epsilon\right)_0^{C_{T_f}},\left(\Phi_\epsilon\right)_0^K\right)\; \chi_{\left\{0\leq \Phi_{\epsilon}\leq K\right\}}\;\nabla \Phi_\epsilon\cdot\nabla N_\epsilon\;dx\leq$$
%
\displaystyle\leq C\left(\|\nabla T_\epsilon\cdot\;\nabla N_\epsilon\|_{L^2\left(\Omega\right)}%\;\parallel\nabla N_\epsilon\parallel_{L^2\left(\Omega\right)}
+\|\nabla N_\epsilon\|_{L^2\left(\Omega\right)}^2+\right.
\\
\\
\left.+\|\nabla \Phi_\epsilon\cdot\;\nabla N_\epsilon\|_{L^2\left(\Omega\right)}%\;\parallel\nabla N_\epsilon\parallel_{L^2\left(\Omega\right)}
\right)\leq %C\left(\|\nabla T_\epsilon,\;\nabla N_\epsilon,\;\nabla\Phi_\epsilon\|_{L^2\left(\Omega\right)}^2\right)
C\left(1+\|\nabla T_\epsilon\|_{L^2\left(\Omega\right)}^2\right)\|\nabla N_\epsilon,\;\nabla\Phi_\epsilon\|_{L^2\left(\Omega\right)}^2.%+\parallel\nabla \Phi_\epsilon\parallel_{L^2\left(\Omega\right)}^2+\parallel\nabla N_\epsilon\parallel_{L^2\left(\Omega\right)}^2\right).%+$$
\end{array}
\end{equation}

Here, we have used that every partial derivative $\dfrac{\partial f_2}{\partial T}$, $\dfrac{\partial f_2}{\partial N}$ and $\dfrac{\partial f_2}{\partial \Phi}$ evaluated at\newline $(\left(T_\epsilon\right)_+^K,\left(N_\epsilon\right)_+^{C\left(T_f\right)},\left(\Phi_\epsilon\right)_+^K)$ is bounded in $L^\infty\left(0,T_f;L^\infty\left(\Omega\right)\right)$
%$$\dfrac{\partial f_2}{\partial T}\left(\left(T_\epsilon\right)_+^K,\left(N_\epsilon\right)_+^{C\left(T_f\right)},\left(\Phi_\epsilon\right)_+^K\right)\text{ is bounded in } L^\infty\left(0,T_f;L^\infty\left(\Omega\right)\right),$$
%
%$$\dfrac{\partial f_2}{\partial N}\left(\left(T_\epsilon\right)_+^K,\left(N_\epsilon\right)_+^{C\left(T_f\right)},\left(\Phi_\epsilon\right)_+^K\right)\text{ is bounded in } L^\infty\left(0,T_f;L^\infty\left(\Omega\right)\right),$$
%
%$$\dfrac{\partial f_2}{\partial \Phi}\left(\left(T_\epsilon\right)_+^K,\left(N_\epsilon\right)_+^{C\left(T_f\right)},\left(\Phi_\epsilon\right)_+^K\right)\text{ is bounded in } L^\infty\left(0,T_f;L^\infty\left(\Omega\right)\right),$$
and the fact that $\big|\nabla\left(T_\epsilon\right)_{+}^{K}\big|\leq\big|\nabla T_\epsilon\big|$ and the same for $\nabla\left(N_\epsilon\right)_{+}^{C\left(T_f\right)}$ and $\nabla\left(\Phi_{\epsilon}\right)_{+}^{K}$. Taking into account this estimate in $\left(\ref{adding1}\right)$, we obtain that

\begin{equation}\label{cotaH2}
\begin{array}{c}
	\dfrac{1}{2}\;\dfrac{d}{dt}\|\nabla N_\epsilon,\;\nabla\Phi_{\epsilon}\|_{L^2\left(\Omega\right)}^2%+\parallel\nabla \Phi_\epsilon\parallel_{L^2\left(\Omega\right)}^2
	
	+ \displaystyle\epsilon\|\Delta N_\epsilon,\;\Delta \Phi_\epsilon\|_{L^2\left(\Omega\right)}^2\leq
	C\left(1+\|\nabla T_\epsilon\|_{L^2\left(\Omega\right)}^2\right)\|\nabla N_\epsilon,\;\nabla\Phi_\epsilon\|_{L^2\left(\Omega\right)}^2.
	%\mathcal{C}\parallel\nabla T_\epsilon,\;\nabla N_\epsilon,\;\nabla\Phi_{\epsilon}\parallel_{L^2\left(\Omega\right)}^2%+\parallel\nabla \Phi_\epsilon\parallel_{L^2\left(\Omega\right)}^2+\parallel\nabla N_\epsilon\parallel_{L^2\left(\Omega\right)}^2
	\end{array}
	\end{equation}

Since $\nabla\;T_\epsilon$ is bounded in $L^2\left(0,T_f;L^2\left(\Omega\right)\right)$, applying Gronwall Lemma, we deduce that

%$$\left(\parallel\nabla N_\epsilon\parallel_{L^2\left(\Omega\right)}^2+\parallel\nabla \Phi_\epsilon\parallel_{L^2\left(\Omega\right)}^2\right)\leq\widetilde{\mathcal{C}}\quad\forall t\in\left(0,T_f\right)$$
%\\
%Hence, on the one hand we have obtained that 

$$\left(\nabla\;N_\epsilon,\;\nabla\;\Phi_{\epsilon}\right)\text{ is bounded in } L^\infty\left(0,T_f;L^2\left(\Omega\right)\right).$$

%and on the ohter hand we had that 

%$$N_\epsilon,\;\Phi_\epsilon\in L^\infty\left(0,T_f;L^2\left(\Omega\right)\right)$$

%Thus, we get
Hence,

$$\left(N_\epsilon,\;\Phi_\epsilon\right)\text{ is bounded in } L^\infty\left(0,T_f;H^1\left(\Omega\right)\right).$$

Finally, integrating in time the inequality $\left(\ref{cotaH2}\right)$, we obtain the following bounds
$$%\parallel \nabla\;N_\epsilon,\nabla\;\Phi_\epsilon\parallel_{{L}^\infty\left(0,T_f;L^2\left(\Omega\right)\right)}^{2}+
\parallel\Delta\left(\sqrt{\epsilon}\; N_\epsilon\right),\Delta\left(\sqrt{\epsilon}\; \Phi_\epsilon\right)\parallel_{{L}^2\left(0,T_f;L^2\left(\Omega\right)\right)}^{2}\leq C.$$
%$$\parallel\nabla\;\Phi_\epsilon\parallel_{{L}^\infty\left(0,T_f;L^2\left(\Omega\right)\right)}^{2}+\parallel\Delta\left(\sqrt{\epsilon}\; \Phi_\epsilon\right)\parallel_{{L}^2\left(0,T_f;L^2\left(\Omega\right)\right)}^{2}\leq\mathcal{C}_2.$$

Hence one has the bound of $\left(\sqrt{\epsilon}\;N_\epsilon,\sqrt{\epsilon}\;\Phi_\epsilon\right)$ in $L^2\left(0,T_f;H^2\left(\Omega\right)\right)$.

\end{proof}	

\subsubsection*{Step 3. Taking limits as $\boldsymbol{\epsilon} \to 0$}
Using $\left(\ref{estimacionesepsilon}\right)$, $\left(\ref{estimacionesepsilon2}\right)$, $\left(\ref{estimacionesepsilon3}\right)$, $\left(\ref{estimacionesepsilon4}\right)$ and Lemma $\ref{estimaciones}$ b), we can conclude that there exists a subsequence $\left(T_{\epsilon},N_\epsilon,\Phi_{\epsilon}\right)\in \mathcal{W}_2$, with $N_{\epsilon},\;\Phi_{\epsilon} \in L^\infty\left(0,T_f;H^1\left(\Omega\right)\right)$ and a limit $\left(T,N,\Phi\right)$ such that as $\epsilon\rightarrow0$,

$$\left\{
\begin{array}{cc}
\vspace{0.2cm}

T_{\epsilon}\rightharpoonup T &\text{  weakly in }  \mathcal{W}_2,\\ 
\\
\left(N_{\epsilon},\Phi_{\epsilon}\right)\stackrel[]{*}{\rightharpoonup} \left(N,\Phi\right) &\text{  weakly * in }  \left(L^\infty\left(0,T_f;H^1\left(\Omega\right)\right)\right)^2,\\
\\
\left(T_\epsilon\right)_t\rightharpoonup T_t &\text{  weakly in  }L^2\left(0,T_f;\left(H^1\left(\Omega\right)\right)'\right),\\
\\
\left(\left(N_\epsilon\right)_t,\left(\Phi_\epsilon\right)_t\right)\stackrel[]{*}{\rightharpoonup}  \left( N_t,\Phi_t\right) &\text{  weakly * in  }\left(L^\infty\left(0,T_f;L^\infty\left(\Omega\right)\right)\right)^2,\\
\\
\nabla\;T_\epsilon \rightharpoonup \nabla\;T &\text{  weakly in  } L^2\left(0,T_f;L^2\left(\Omega\right)\right),\\
\\
\left(\sqrt{\epsilon}\;\Delta\;N_{\epsilon},\sqrt{\epsilon}\;\Delta\;\Phi_{\epsilon}\right)\rightharpoonup \left(\theta_1,\theta_2\right)&\text{  weakly in  } L^2\left(0,T_f;L^2\left(\Omega\right)\right).\\

\end{array}
\right.$$

In particular, 
$$\left(\epsilon\;\Delta N_\epsilon,\epsilon\;\Delta \Phi_\epsilon\right)=\left(\sqrt{\epsilon}\left(\sqrt{\epsilon}\;\Delta N_{\epsilon}\right),\sqrt{\epsilon}\left(\sqrt{\epsilon}\;\Delta\Phi_{\epsilon}\right)\right)\rightharpoonup\left(0,0\right)\text{  weakly in } L^2\left(\left(0,T_f\right);L^2\left(\Omega\right)\right).$$

From Aubin-Lions compactness
\begin{equation}\label{compacidad_fuerte}
\left\{
\begin{array}{cl}
\vspace{0.2cm}
T_{\epsilon}\rightarrow T &\text{  strong in }  L^2\left(0,T_f;L^2\left(\Omega\right)\right)\cap\mathcal{C}^0\left(0,T_f;\left(H^1\left(\Omega\right)\right)'\right),\\
\left(N_{\epsilon},\Phi_{\epsilon}\right)\rightarrow\left(N,\Phi\right) &\text{  strong in }  \left(\mathcal{C}^0\left(0,T_f;L^2\left(\Omega\right)\right)\right)^2.
\end{array}
\right.
\end{equation}
%$$\epsilon\;\Delta \Phi_\epsilon=\sqrt{\epsilon}\left(\sqrt{\epsilon}\;\Delta\Phi_{\epsilon}\right)\rightharpoonup0\text{  weakly in } L^2\left(\left(0,T_f\right);L^2\left(\Omega\right)\right).$$

Now, we will take limits in the nonlinear diffusion term in $L^2\left(0,T_f;L^2\left(\Omega\right)\right)$. On the one hand, we have that $\kappa_1\;P\left(\Phi_{\epsilon},T_\epsilon\right)+\kappa_0$ is continuous in $\mathbb{R}^2$ and it is bounded in $L^\infty\left(0,T_f,L^\infty\left(\Omega\right)\right)$ and for $\left(\ref{compacidad_fuerte}\right)$, we obtain that $\left(T_\epsilon,\Phi_{\epsilon}\right)\rightarrow\left(T,\Phi\right)$ a.e. in $\left(0,T_f\right)\times\Omega$. Hence, using dominated convergence Theorem

\begin{equation}\label{convergP}
\left(\kappa_1\;P\left(\left(\Phi_{\epsilon}\right)_+^K,\left(T_\epsilon\right)_+^K\right)+\kappa_0\right)\rightarrow\left(\kappa_1\;P\left(\Phi_+^K,T_+^K\right)+\kappa_0\right)\;\text{  in }L^p\left(0,T_f;L^p\left(\Omega\right)\right),\;\;\forall p<\infty.
\end{equation}

On the other hand, $\nabla\;T_\epsilon \rightharpoonup \nabla\;T$ weakly in $L^2\left(0,T_f;L^2\left(\Omega\right)\right)$. 
\\

Hence, since $\left(\kappa_1\;P\left(\left(\Phi_{\epsilon}\right)_+^K,\left(T_\epsilon\right)_+^K\right)+\kappa_0\right)\nabla\;T_\epsilon$ is bounded in $L^2\left(0,T_f;L^2\left(\Omega\right)\right)$, one has

$$\left(\kappa_1\;P\left(\left(\Phi_{\epsilon}\right)_+^K,\left(T_\epsilon\right)_+^K\right)+\kappa_0\right)\nabla\;T_\epsilon\rightharpoonup\left(\kappa_1\;P\left(\Phi_+^K,T_+^K\right)+\kappa_0\right)\nabla\;T\;\text{  weakly in  } L^2\left(0,T_f;L^2\left(\Omega\right)\right).$$
\\

Finally, for all $\varphi\in L^2\left(0,T_f;H^1\left(\Omega\right)\right)$  we conclude that

$$\displaystyle\int_{0}^{T_f}\int_{\Omega}\left(\left(T_\epsilon\right)_t,\left(N_\epsilon\right)_t,\left(\Phi_\epsilon\right)_t\right)\;\varphi\;dx\;dt\rightarrow\displaystyle\int_{0}^{T_f}\int_{\Omega}\left(T_t,N_t,\Phi_t\right)\;\varphi\;dx\;dt,$$

$$\displaystyle\int_{0}^{T_f}\int_{\Omega}\left(\kappa_1\;P\left(\left(\Phi_{\epsilon}\right)_+^K,\left(T_\epsilon\right)_+^K\right)+\kappa_0\right)\;\nabla T_\epsilon\cdot\nabla\varphi\;dx\;dt\rightarrow\displaystyle\int_{0}^{T_f}\int_{\Omega}\left(\kappa_1\;P\left(\Phi_+^K,T_+^K\right)+\kappa_0\right)\;\nabla T\cdot\nabla\varphi\;dx\;dt,$$

$$\displaystyle\int_{0}^{T_f}\int_{\Omega}\left(\sqrt{\epsilon}\left(\sqrt{\epsilon}\Delta N_{\epsilon}\right),\sqrt{\epsilon}\left(\sqrt{\epsilon}\Delta\Phi_{\epsilon}\right)\right)\varphi\;dx\;dt\rightarrow\left(0,0\right),$$

$$\displaystyle\int_{0}^{T_f}\int_{\Omega}f_i\left(\left(T_\epsilon\right)_+,\left(N_\epsilon\right)_\epsilon,\left(\Phi_\epsilon\right)_\epsilon\right)\varphi\;dx\;dt\rightarrow\displaystyle\int_{0}^{T_f}\int_{\Omega}f_i\left(T_+,N_+,\Phi_+\right)\varphi\;dx\;dt,$$
\\
para $i=1,2,3$.
\\

Taking limits as $\epsilon\rightarrow0$ in $\left(\ref{probregul2}\right)$, we deduce that $\left(T,N,\Phi\right)$ is a weak-strong solution of $\left(\ref{problin}\right)$ (which is in addition a weak-strong solution of problem $\left(\ref{probOriginal}\right)$-$\left(\ref{condinicio}\right)$) where the convergence for $\left(\ref{condinicio}\right)$ is obtained thanks to $\left(\ref{compacidad_fuerte}\right)$.
\end{proof}

\section{Asymptotic behaviour}
Once we have proved the existence of weak-strong solution of $\left(\ref{probOriginal}\right)$-$\left(\ref{condinicio}\right)$ for any finite time $T_f>0$, we are going to study the asymptotic behaviour of the solution as $t\to \infty$.. In order to obtain the equilibrium points, we solve the following nonlinear algebraic system

$$f_1\left(T,N,\Phi\right) = 0,\quad f_2\left(T,N,\Phi\right) =0,\quad f_3\left(T,N,\Phi\right) =0.$$
%$$\left\{\begin{array}{rcl}
%f_1\left(T,N,\Phi\right) & = &0,\\
%&&\\
%f_2\left(T,N,\Phi\right)& = & 0,\\
%&&\\
%f_3\left(T,N,\Phi\right) & = &0.
%\end{array}\right.$$

%From $f_2\left(T,N,\Phi\right)=0$ we obtain that
%
%%$$\alpha\;T\;\sqrt{1-\left(P\left(\Phi,T\right)\right)^2}+\delta\;T\;\Phi+\beta\;N\left(T+\Phi\right)=0\Rightarrow
%$$\left\{\begin{array}{lcl}
%T\;\sqrt{1-\left(P\left(\Phi,T\right)\right)^2}=0,\\
%\\
%T\;\Phi=0,\\
%\\
%N\left(T+\Phi\right)=0.
%\end{array}
%\right.$$	
%
%From $T\;\sqrt{1-\left(P\left(\Phi,T\right)\right)^2}=0$ and $T\;\Phi=0$, we have $T=0$. From the third condition, $N\left(T+\Phi\right)=0$, we obtain that $N=0$ or $\Phi=0$.
Following the same argument used in \cite[Section 4.1]{Romero_2020}, the equilibria of $\left(\ref{probOriginal}\right)$ are
\vspace{-0.1cm}
\begin{equation}\label{solequilibrio}
\begin{split}
\text{\textbullet}\;\; P_1&=\left\{\left(0,0,0\right)\right\}.\\
\text{\textbullet}\;\; P_2&=\left\{\left(0,N,0\right),\quad N>0\right\}.\\
\text{\textbullet}\;\; P_3&=\left\{\left(0,0,\Phi\right),\quad\Phi>0\right\}.
\end{split}
\end{equation}

\begin{obs}
	Observe that $P_1\cup P_2\cup P_3$ is a continuum of equilibria points.
\end{obs}

\begin{obs}
	In the following results, we assume sometimes the hypothesis $N_0(x)>0$ for $x\in \overline{\Omega}$. However, this condition can be relaxed for $N(t_*,x)$ for some $t_*\ge 0$, by considering the problem starting in $t=t_*$.
	%The lower bounds of $N_0(x)$ imposed in Lemmas $\ref{lemaFi_0}$, $\ref{lemaestabilidad1}$ and $\ref{lemaestabilidad2}$ can be relaxed for $N(t_*,x)$ for some $t_*\ge 0$, by considering the problem starting in $t=t_*$.
\end{obs}
Now, we present a result of pointwise convergence to zero of the vasculature.%about the solution of $\left(\ref{probOriginal}\right)$-$\left(\ref{condinicio}\right)$.  %The first result about the kind of equilibria points that we can deduce it is the following.

\begin{lem}\label{lemaFi_0}
	Given $\epsilon>0$ and a solution $\left(T,N,\Phi\right)$ of $\left(\ref{probOriginal}\right)$-$\left(\ref{condinicio}\right)$, if there exists $\widetilde{\Omega}\subset\Omega$ with $\mid\widetilde{\Omega}\mid>0$ such that $0<\epsilon\leq N_0\left(x\right)$ a.e. $x\in \widetilde \Omega$, one has $\Phi\left(t,x\right)\rightarrow0$ when $t\rightarrow+\infty$ a.e. $x\in \widetilde \Omega$.
	
\end{lem}
The proof of this result is rather similar to \cite[Lemma 12]{Romero_2020} with the difference that due to the fact that $\Phi\left(t,x\right),\;N\left(t,x\right)\in L^\infty\left(0,T_f;H^1\left(\Omega\right)\right)$, we prove Lemma $\ref{lemaFi_0}$ using a subdomain $\widetilde{\Omega}\subset\Omega$ with positive measure instead of a pointwise argument for every $x\in \Omega$. %The main difference in this result is the regularity of $\Phi\left(t,x\right)$ and $N\left(t,x\right)$.
\\

As consequence of Lemma $\ref{lemaFi_0}$ and that $t\mapsto N(t,\cdot)$ is increasingly a.e. $x\in\Omega$, we  deduce:
\begin{col}\label{P3_inestable}
	The equilibria solution $P_3$ %=\left\{\left(0,0,\Phi_*\left(x\right)\right),\quad \Phi_*\in\mathcal{C}^0\left(\overline{\Omega}\right)\;\;\text{with}\;\;\Phi_*>0\;\;\text{in}\;\;\overline{\Omega}\right\}$$ 
	is unstable.
\end{col}

Now, we prove a comparison result that provides a uniform bound for the solution of a nonlinear diffusion equation which we will use later:

\begin{lem}\label{lemaTmenorS}
	Let $\Omega\subseteq\mathbb{R}^n$ a bounded set of class $\mathcal{C}^2$, and $0<T_f<+\infty$. Given the following problems
	
	\begin{equation}\label{ec1}
	\left\{\begin{array}{l}
	T_t-\nabla\left(\nu\left(t,x,T\right)\cdot\nabla\;T\right)=f\left(t,x,T\right)\;\;\text{in}\;\;\left(0,T_f\right)\times\Omega,\\
	\\
	T\left(0,x\right)=T_0\left(x\right)\;\;\text{in}\;\;\Omega,\\
	\\
	\dfrac{\partial T}{\partial \text{n}}\Bigg\vert_{\partial\Omega}=0\;\;\text{in}\;\;\left(0,T_f\right)\times\partial\Omega,
	\end{array}
	\right.
	\end{equation}
	\\
	with $\nu\left(\cdot,\cdot,T\right)\in L^\infty\left(0,T_f;L^\infty\left(\Omega\right)\right)$ $\forall\; T\in\mathbb{R}$ a given non-negative function and $f\left(\cdot,\cdot,T\right)\in L^2\left(0,T_f;H^1\left(\Omega\right)\right)$ $\forall\; T\in\mathbb{R}$ and
	
	\begin{equation}\label{ec2}
	\left\{\begin{array}{l}
	y_t=g\left(t,y\right)\;\;\text{in}\;\;\left(0,T_{\max}\right),\\
	\\
	y\left(0\right)=y_0\\
	\end{array}
	\right.
	\end{equation}
	 with $0<T_{\max}<+\infty$ and $g\in C^0([0,T_{\max}]\times \mathbb{R})$ and locally lipschitz with respect $y$. Suppose that $\left(\ref{ec1}\right)$ has a weak solution $T\in \mathcal{W}_2\cap L^\infty\left(0,T_f;L^\infty\left(\Omega\right)\right)$ in $\left(0,T_f\right)\times\Omega$,  %where $W_2$ is the space defined in $\left(\ref{W}\right)$ 
	 and $\left(\ref{ec2}\right)$ has a unique solution $y\in C^1\left(\left[0,T_{\max}\right]\right)$ in $\left[0,T_{\max}\right]$. If $T_0\left(x\right)\leq y_0$ a.e. $x\in\Omega$ and 
	 \begin{equation}\label{f_menor_g}
 f\left(t,x,p\right)\leq g\left(t,p\right),\;\;\text{ a.e.}\;\;\left(t,x\right)\in\left(0,T_*\right)\times\Omega,\;\;\forall\;p\in\mathbb{R}
	 \end{equation}
	 %with $t\in\left(0,T^*\right)$, where $T^*=\min\left\{T_f,T_{\max}\right\}$, a.e $x\in\Omega$ and for all $p$ in $\mathbb{R}$, %and the maximal classical solution of $\left(\ref{ec2}\right)$ $y=y(t)$ is defined in $\left(0,T_{\max}\right)$ 
	with $T_*=\min\left\{T_f,T_{\max}\right\}$. Then,
	 $$T\left(t,x\right)\leq y\left(t\right),\;\;\text{ a.e.}\; \left(t,x\right)\in\left(0,T_*\right)\times\Omega.$$
	 %for any $T$ a weak solution of problem $\left(\ref{ec1}\right)$ in $\left(0,T_f\right)$ and $y$ the solution of problem $\left(\ref{ec2}\right)$ in $\left(0,T_{\max}\right)$.
%	  Hence, $\left(\ref{ec1}\right)$ admits an unique weak solution $T\in L^\infty\left(0,T_f;L^2\left(\Omega\right)\right)\cap L^2\left(0,T_f;H^1\left(\Omega\right)\right)$ such that $T\left(t,x\right)\rightarrow0$ exponentially and uniformly in $\Omega$ as $t\rightarrow+\infty$.
\end{lem}

\begin{proof}
	
	%For any $T$ a weak solution of problem $\left(\ref{ec1}\right)$ in $\left(0,T_f\right)$ and $y$ the solution of problem $\left(\ref{ec2}\right)$ in $\left(0,T_{\max}\right)$
	%By Picard's Theorem we obtain that problem $\left(\ref{ec2}\right)$ has a unique solution $y\in C^1\left(0,T_{\max}\right)$ in $\left(0,\mathcal{T}\right)$ with $0<\mathcal{T}\leq T_{\max}$. 
	%\\
	
	Let $T=T(t,x)$ a weak solution of $\left(\ref{ec1}\right)$ in $\left(0,T_f\right)$ and $y=y(t)$ the classical solution of $\left(\ref{ec2}\right)$ in $\left[0,T_{\max}\right]$ and we consider the problem which satisfies the difference $T-y$,
	
	\begin{equation}\label{ec3}
	\left\{\begin{array}{l}
	\left(T-y\right)_t-\nabla\cdot\left(\nu\left(t,x,T\right)\nabla\left(T-y\right)\right)=f\left(t,x,T\right)-g\left(t,y\right)\quad\text{in}\;\;\left(0,T_*\right)\times\Omega,\\
	\\
	T\left(0,x\right)-y\left(0\right)=T_0\left(x\right)-y_0\quad\text{a.e.}\;\;x\in\Omega,\\
	\\
	\dfrac{\partial (T-y)}{\partial \text{n}}\Bigg\vert_{\partial\Omega}=0\;\;\text{in}\;\;\left(0,T^*\right)\times\partial\Omega,
	\end{array}
	\right.
	\end{equation}
	\\
	
	Multiplying the first equation of $\left(\ref{ec3}\right)$ by $\left(T-y\right)_+$ and integrating in $\Omega$ and using $\left(\ref{f_menor_g}\right)$, we obtain that
	
	$$\dfrac{1}{2}\dfrac{d}{dt}\int_{\Omega} \left(T-y\right)_+^2\;dx+\int_{\Omega}\nu\left(t,x,T\right)|\nabla\left(T-y\right)_+^2|dx=\int_{\Omega} \left(f\left(t,x,T\right)-g\left(t,y\right)\right)\left(T-y\right)_+\;dx\leq$$
	
	$$\leq\int_{\Omega} \left(g\left(t,T\right)-g\left(t,y\right)\right)\left(T-y\right)_+\;dx\leq L_{\widetilde{K}}\int_{\Omega}\left(T-y\right)_+^2\;dx$$
	since the graph of $T(t,x)$ and $y\left(t\right)$ belong to a compact set $\widetilde{K}\subset\mathbb{R}$ because $T\in L^\infty\left(0,T_f;L^\infty\left(\Omega\right)\right)$ and $y\in C^1\left(\left[0,T_{\max}\right]\right)$ and hence $L_{\widetilde{K}}$ is a lipschitz constant of this compact set. Thus, we deduce
	$$\| \left(T-y\right)_+\left(t\right)\|_{L^2\left(\Omega\right)}^2\leq \|\left(T_0\left(x\right)-y_0\right)_+\|_{L^2\left(\Omega\right)}^2\;e^{2\;L_{\widetilde{K}}\;t}=0,$$
     %since we have supposed that $T_0\left(x\right)\leq y_0$. 
     hence, $T\left(t,x\right)\leq y\left(t\right)$ a.e. $\left(t,x\right)\in\left(0,T_*\right)\times\Omega$.

\end{proof}

Now, using Lemma $\ref{lemaTmenorS}$, we are going to deduce the same results for the asymptotic behaviour of any solution $\left(T,N,\Phi\right)$ of $\left(\ref{probOriginal}\right)$-$\left(\ref{condinicio}\right)$ which we proved in \cite[Lemmas 13 and 15]{Romero_2020}, where uniform convergence for $\left(T,N,\Phi\right)$ was obtained.
%In the following result, adding further conditions on some parameters of the problem, we can deduce in the behaviour of the solution of the system $\left(\ref{probOriginal}\right)$-$\left(\ref{condinicio}\right)$ as $t\rightarrow+\infty$.
\begin{lem}\label{lemaestabilidad1}

Given a solution $\left(T,N,\Phi\right)$ of $\left(\ref{probOriginal}\right)$-$\left(\ref{condinicio}\right)$ such that $$N_0\left(x\right)\ge N_0^{min}>0\;\;\text{ a.e.}\;\;x\in \Omega$$ and assume that

\begin{equation}\label{CondiNlejosKdelta}
\delta\geq\dfrac{\gamma}{K}.
\end{equation}

Then, %for all $t\geq0$ and $x\in\overline{\Omega}$:

\begin{equation}\label{cota_exponencial_Phi}
%\|\Phi\|_{L^\infty\left(0,T_f;L^\infty\left(\Omega\right)\right)}
0\leq\Phi\left(t,x\right)\leq \|\Phi_0\|_{L^\infty\left(\Omega\right)}\;e^{-\beta_2\;N_{0}^{\min}\;t},\;\; \text{a.e.}\;\;\left(t,x\right)\in\left(0,+\infty\right)\times\Omega.
\end{equation}

%where $\displaystyle N_{0}^{\min}=\min_{x\in\Omega} N_0\left(x\right)$. 
In addition, it holds that if $\beta_1\neq\beta_2$, then
\begin{equation}\label{cota_exponecial_T_beta1_dif_beta2}
0\le T(t,x)\leq\| T_0\|_{L^\infty\left(\Omega\right)}\;e^{-\beta_1\;N_0^{\min}\;t}+\dfrac{\rho\;\| \Phi_0\|_{L^\infty\left(\Omega\right)}}{\left(\beta_1-\beta_2\right)N_0^{\min}}\left(e^{-\beta_2\;N_0^{\min}\;t}-e^{-\beta_1\;N_0^{\min}\;t}\right),\;\;\text{a.e.}\;\; (t,x)\in (0,+\infty)\times \Omega,
%0\le T(t,x)\leq\left(\| T_0\|_{L^\infty\left(\Omega\right)}-\dfrac{\rho\;\| \Phi_0\|_{L^\infty\left(\Omega\right)}}{\left(\beta_1-\beta_2\right)N_0^{\min}}\right)e^{-\beta_1\;N_0^{\min}\;t}+\dfrac{\rho\;\| \Phi_0\|_{L^\infty\left(\Omega\right)}}{\left(\beta_1-\beta_2\right)N_0^{\min}}e^{-\beta_2\;N_0^{\min}\;t},\;\;\text{a.e.}\;\; (t,x)\in (0,+\infty)\times \Omega,
\end{equation}
whereas if $\beta_1=\beta_2$, then
\begin{equation}\label{cota_exponecial_T_beta1_igual_beta2}
0\le T(t,x)\leq\left(\| T_0\|_{L^\infty\left(\Omega\right)}+\rho\;\| \Phi_0\|_{L^\infty\left(\Omega\right)}\;t\right)\;e^{-\beta_1\;N_0^{\min}\;t},\;\;\text{a.e.} \;\;(t,x)\in (0,+\infty)\times \Omega.
\end{equation}

%there exists $\mu\in\left(0,\beta_2\;N_{0}^{\min}\right)$ such that	
%
%%$$\parallel T\parallel_{L^\infty\left(0,T_f;L^\infty\left(\Omega\right)\right)}
%\begin{equation}\label{cota_exponencial_T}
%0\le T(t,x)\leq M\;e^{-\mu\;t},\;\;\text{a.e.} \;\;(t,x)\in (0,+\infty)\times \Omega
%\end{equation}
%
%with $M=\max\left\{\| T_0\|_{L^\infty\left(\Omega\right)},\;\dfrac{\rho\;\| \Phi_0\|_{L^\infty\left(\Omega\right)}}{\beta_2\;N_{0}^{\min}-\mu}\right\}>0$. 
Moreover, there exists $N_{\max}>\|N_0\|_{L^\infty\left(\Omega\right)}$ such that
$$N\left(t,x\right)\leq N_{\max},\;\;\text{ a.e.}\; \left(t,x\right)\in\left(0,+\infty\right)\times\Omega.$$

%\begin{obs}
%	There exists $N_*\in L^\infty\left(\Omega\right)$ with $N_{\max}\geq N_*\geq N_0$ such that 
%	
%	$$N\left(t,x\right)\rightarrow N_*\left(x\right)\;\;\text{as}\;\; t\rightarrow+\infty,\;\;\forall x\in\Omega.$$
%\end{obs}
\end{lem}
\begin{proof}
%
% Using $\left(\ref{CondiNlejosKdelta}\right)$, that $T,\;\Phi\geq0$ and $N\geq N_0$, we can bound
%\\
%$$f_3\left(T,N,\Phi\right)\leq\dfrac{\gamma}{K}\;\Phi\;T-\delta\;\Phi\;T-\beta\;\Phi\;N\leq-\left(\delta-\dfrac{\gamma}{K}\right)\Phi\;T -\beta\;\Phi\;N\leq$$ $$\leq-\beta\;\Phi\;N\leq-\beta\;\Phi\;N_0.$$
%\\
%Hence
%
%\begin{equation}\label{AcotacionF}
%\left\{\begin{array}{l}
%\dfrac{\partial\Phi}{\partial t}\leq -\beta\;\Phi\;N_{0}\left(x\right)\;\;\text{   in   }\;\;\left[0,+\infty\right)\times\overline{\Omega,}\\
%\\
%\Phi\left(0,x\right)=\Phi_0\left(x\right)\leq \parallel \Phi_0\parallel_{L^\infty\left(\Omega\right)}\;\;\text{  in  }\;\;\overline{\Omega},
%\end{array}\right.
%\end{equation}
%
%
%Using $N_{0}^{\min}>0$, we conclude that
%%Solving $\left(\ref{AcotacionF}\right)$, we get a negative exponential upper bound for $\Phi\left(t,x\right)$.
%
%\begin{equation}\label{Fi_acotada}
%\Phi\left(t,x\right)\leq\Phi_0\left(x\right)\;e^{-\beta\;N_{0}\left(x\right)\;t}\leq \parallel \Phi_0\parallel_{L^\infty\left(\Omega\right)}\;e^{-\beta\;N_{0}^{\min}\;t}\rightarrow0
%\end{equation}
%
%as $t\rightarrow+\infty$ uniformly in $x\in\overline{\Omega}$.
%\\
%
%Using $\left(\ref{Fi_acotada}\right)$ and $T,\;\Phi\geq0$ and $N\geq N_0$ we can bound $f_1\left(T,N,\Phi\right)$ as follows: 
%$$f_1\left(T,N,\Phi\right)\leq \rho\;P\left(\Phi,T\right)\;T-\alpha\;T\sqrt{1-P^2\left(\Phi,T\right)}-\beta\;N\;T\leq$$
%$$\leq\rho\;\Phi-\beta\;N\;T\leq\rho\;\parallel \Phi_0\parallel_{L^\infty\left(\Omega\right)}\;e^{-\beta\;N_{0}^{\min}\;t}-\beta\;N_{0}^{\min}\;T.$$
%
To prove $\left(\ref{cota_exponencial_Phi}\right)$ we repeat the same argument for the exponential convergence of $\Phi\left(t,x\right)$ to zero in $L^\infty\left(0,T_f;L^\infty\left(\Omega\right)\right)$ made in \cite[Lemma 13]{Romero_2020}. To prove $\left(\ref{cota_exponecial_T_beta1_dif_beta2}\right)$ and $\left(\ref{cota_exponecial_T_beta1_igual_beta2}\right)$, we bound $f_1\left(T,N,\Phi\right)$ using $\left(\ref{cota_exponencial_Phi}\right)$ as follows
$$f_1\left(T,N,\Phi\right)\leq \rho\;\| \Phi_0\|_{L^\infty\left(\Omega\right)}\;e^{-\beta_2\;N_{0}^{\min}\;t}-\beta_1\;N_{0}^{\min}\;T,$$
and we apply Lemma $\ref{lemaTmenorS}$ taking the following linear differential problem
\begin{equation}\label{edo}
\left\{\begin{array}{l}
y_t=\rho\;\| \Phi_0\|_{L^\infty\left(\Omega\right)}\;e^{-\beta_2\;N_{0}^{\min}\;t}-\beta_1\;N_{0}^{\min}\;y\;\quad\text{in}\;\;\left(0,+\infty\right),\\
\\
y\left(0\right)=\| T_0\|_{L^\infty\left(\Omega\right)}.\\
\end{array}
\right.
\end{equation}

Solving $\left(\ref{edo}\right)$ we obtain that if $\beta_1\neq\beta_2$,
$$y\left(t\right)=\| T_0\|_{L^\infty\left(\Omega\right)}\;e^{-\beta_1\;N_0^{\min}\;t}+\dfrac{\rho\;\| \Phi_0\|_{L^\infty\left(\Omega\right)}}{\left(\beta_1-\beta_2\right)N_0^{\min}}\left(e^{-\beta_2\;N_0^{\min}\;t}-e^{-\beta_1\;N_0^{\min}\;t}\right),\quad\text{in}\;\;\left(0,+\infty\right),$$
and if $\beta_1=\beta_2$,
$$y\left(t\right)=\left(\| T_0\|_{L^\infty\left(\Omega\right)}+\rho\;\| \Phi_0\|_{L^\infty\left(\Omega\right)}\;t\right)\;e^{-\beta_1\;N_0^{\min}\;t},\quad\text{in}\;\;\left(0,+\infty\right).$$

Hence, we obtain that $$T\left(t,x\right)\leq y\left(t\right),\;\;\text{ a.e.}\left(t,x\right)\in\left(0,+\infty\right)\times \Omega.$$

%Now, we can find a super solution of $y\left(t\right)$ with the form $$\overline{y}\left(t\right)=M\;e^{-\mu\;t}$$
%with $\mu\in \left(0,\beta_1\;N_{0}^{\min}\right)$ and $M=\max\left\{\| T_0\|_{L^\infty\left(\Omega\right)},\;\dfrac{\rho\;\| \Phi_0\|_{L^\infty\left(\Omega\right)}}{\beta_2\;N_{0}^{\min}-\mu}\right\}>0$. Consequently, $\overline{y}$ is a super solution of $T\left(t,x\right)$ and we have
%
%\begin{equation}\label{T_acotada}
%T\left(t,x\right)\leq \overline{y}\left(t\right)= M\;e^{-\mu\;t},\;\; \text{ a.e.}\;\; \left(t,x\right)\in\left(0,+\infty\right)\times\Omega.
%\end{equation}
%Therefore, we get the following parabolic problem 
%
%
%\begin{equation}\label{AcotacionT}
%\left\{\begin{array}{l}
%\dfrac{\partial T}{\partial t}-\Delta\;T= \rho\;\parallel \Phi_0\parallel_{L^\infty\left(\Omega\right)}\;e^{-\beta\;N_{0}^{\min}\;t}-\beta\;N_{0}^{\min}\;T\text{  in  }\left(0,+\infty\right)\times\overline{\Omega},\\
%\\
%T\left(0,x\right)= \parallel T_0\parallel_{L^\infty\left(\Omega\right)}\text{  in  }\Omega,\\
%\\
%\dfrac{\partial T}{\partial n}\Bigg\vert_{\partial\Omega}=0.
%\end{array}
%\right.
%\end{equation} 
%
%
%
%$$T\left(t,x\right)\leq \parallel T_0\parallel_{L^\infty\left(\Omega\right)}e^{-\left(\beta\;N_{0}^{\min}-\rho\right)\;t}\rightarrow0$$
%uniformly as $t\rightarrow+\infty$ for $x\in\Omega$.
%%\\
%

Finally, we get the bound $N\left(t,x\right)\le N_{\max}$ as in \cite[Lemma 13]{Romero_2020} using the upper uniform bounds obtained for $T\left(t,x\right)$ and $\Phi\left(t,x\right)$ previously in $\left(\ref{cota_exponencial_Phi}\right)$ and $\left(\ref{cota_exponecial_T_beta1_dif_beta2}\right)$ or $\left(\ref{cota_exponecial_T_beta1_igual_beta2}\right)$.
\end{proof}
 	In the following result, we study the situation when $N_0\left(x\right)$ is close to $K$ in the whole domain $\Omega$.
  \begin{lem}\label{lemaestabilidad2}
 	%Given $\epsilon>0$ small enough such that $K-N_0\left(x\right)\leq\epsilon$ 
 	Assuming $N_0\left(x\right)\geq K-\epsilon$ a.e. $x\in\Omega$ for $\epsilon$ small enough and a weak-strong solution $\left(T,N,\Phi\right)$ of $\left(\ref{probOriginal}\right)$-$\left(\ref{condinicio}\right)$, then,

 	\begin{equation}\label{cota_expnencial_T_2}
 	0\leq T\left(t,x\right)\leq \| T_0\|_{L^\infty\left(\Omega\right)}\;e^{-\left(\beta_1\left(K-\epsilon\right)-\rho\dfrac{\epsilon}{K}\right)\;t},\;\;\text{a.e.}\;\;\left(t,x\right)\in\left(0,+\infty\right)\times\Omega,
 	\end{equation}
 	and
 	\begin{equation}\label{cota_expnencial_Phi_2}
 	0\leq\Phi\left(t,x\right)\leq \| \Phi_0\|_{L^\infty\left(\Omega\right)}\;e^{-\left(\beta_2\left(K-\epsilon\right)-\gamma\dfrac{\epsilon}{K}\right)\;t},\;\;\text{a.e.}\;\;\left(t,x\right)\in\left(0,+\infty\right)\times\Omega,
 	\end{equation}

 	In addition, if $\rho\dfrac{\epsilon}{K}-\beta_1\left(K-\epsilon\right)<0$ and $\gamma\dfrac{\epsilon}{K}-\beta_2\left(K-\epsilon\right)<0$ then, there exists $N_{\max}>\|N_0\|_{L^\infty\left(\Omega\right)}$ such that
 	\begin{equation}\label{cota_expnencial_N_2}
 	N\left(t,x\right)\leq N_{\max},\;\;\text{a.e.}\;\left(t,x\right)\in\left(0,+\infty\right)\times\Omega.
 	\end{equation}

 \end{lem}

 \begin{proof}
 	
 	Since $N$ is increasing in time, we get
 	$$N\left(t,x\right)\geq N_0\left(x\right)> K-\epsilon\;\;\text{a.e.}\; \left(t,x\right)\in\left(0,+\infty\right)\times\Omega.$$
 	
 	Using now that $T,\;\Phi\geq 0$, and
 	
 	$$1-\dfrac{T+N+\Phi}{K}\leq 1-\dfrac{N}{K}<1-\dfrac{K-\epsilon}{K}=\dfrac{\epsilon}{K}$$
 	therefore, 
 	
 	$$\dfrac{\partial T}{\partial t}-\nabla\cdot\left(\left(\kappa_1\;P\left(\Phi,T\right)+\kappa_0\right)\nabla
 	T\right)=f_1\left(T,N,\Phi\right)\leq\rho\;T\;\dfrac{\epsilon}{K}-\beta_1\;\left(K-\epsilon\right)\;T=\left(\rho\;\dfrac{\epsilon}{K}-\beta_1\left(K-\epsilon\right)\right)T.$$
 	
 	Hence, we apply Lemma $\ref{lemaTmenorS}$ with $y_0=\| T_0\|_{L^\infty\left(\Omega\right)}$ and $g\left(t,y\right)=\left(\rho\;\dfrac{\epsilon}{K}-\beta_1\left(K-\epsilon\right)\right)y$ to obtain that
 	$$T\left(t,x\right)\leq y\left(t\right)=\| T_0\|_{L^\infty\left(\Omega\right)}\;e^{-\left(\beta_1\left(K-\epsilon\right)-\rho\dfrac{\epsilon}{K}\right)\;t},\;\;\text{a.e.}\; \left(t,x\right)\in\left(0,+\infty\right)\times\Omega.$$
% 	 where $y=y(t)$ is the solution of
% 	 
% 	 %\begin{equation}\label{ec2}
% 	 $$\left\{\begin{array}{l}
% 	 y_t=g\left(t,y\right)\;\;\text{in}\;\;\left(0,T_{\max}\right),\\
% 	 \\
% 	 y\left(0\right)=y_0.\\
% 	 \end{array}
% 	 \right.$$
 	 %\end{equation}
 	 
% 	 where $y\left(t\right)$ has the form
% 	 $$y\left(t\right)=\| T_0\|_{L^\infty\left(\Omega\right)}\;e^{-\left(\beta_1\left(K-\epsilon\right)-\rho\dfrac{\epsilon}{K}\right)\;t}.$$ 
 	 
% 	 Hence, if $\left(\rho\;\dfrac{\epsilon}{K}-\beta_1\left(K-\epsilon\right)\right)<0$, $\overline{y}\left(t\right)$ is a super solution of 
% 	 $T\left(t,x\right)$ satisfies $\left(\ref{cota_expnencial_T_2}\right)$.
%     \\
% 	 \begin{equation}\label{acotacion_T2}
% 	 T\left(t,x\right)\leq \overline{y}\left(t\right)=\| T_0\|_{L^\infty\left(\Omega\right)}\;e^{-\left(\beta_1\left(K-\epsilon\right)-\rho\dfrac{\epsilon}{K}\right)\;t}\rightarrow0
% 	 \end{equation}
% 	 
% 	 uniformly for $x\in\Omega$ as $t\rightarrow+\infty$.
% 	 \\

Now we repeat the same argument made in \cite[Lemma 15]{Romero_2020} to prove the uniform exponential convergence of $\Phi\left(t,x\right)$ to zero in $L^\infty\left(0,T_f;L^\infty\left(\Omega\right)\right)$ given in $\left(\ref{cota_expnencial_Phi_2}\right)$ and for the bound of $N\left(t,x\right)$ given in $\left(\ref{cota_expnencial_N_2}\right)$ using the upper uniform bounds $\left(\ref{cota_expnencial_T_2}\right)$ and $\left(\ref{cota_expnencial_Phi_2}\right)$ already proved for $T\left(t,x\right)$ and $\Phi\left(t,x\right)$.
	\end{proof}

	\begin{obs}
	In Lemmas $\ref{lemaestabilidad1}$ and $\ref{lemaestabilidad2}$ using that $N\left(\cdot,x\right)$ is increasing in time, there exists $N_*\in L^\infty\left(\Omega\right)$ with $N_{\max}\geq N_*\geq N_0$ a.e. in $\Omega$ such that 
	
	$$N\left(t,x\right)\rightarrow N_*\left(x\right)\;\;\text{as}\;\; t\rightarrow+\infty,\;\;\text{a.e.}\; x\in\Omega.$$
\end{obs}

\section{A FE numerical scheme}\label{esquema_espacio}

In this part, we build an uncoupled and linear fully discrete scheme of $\left(\ref{probOriginal}\right)$-$\left(\ref{condinicio}\right)$ by means of a Implicit-Explicit (IMEX) Finite Difference in time approximation and $P_1$ continuous finite element with "mass-lumping" in space. This scheme will preserve the pointwise and energy estimates that appear in Lemmas $\ref{estimaciones}$ and $\ref{lemaLinfyh1}$ considering acute triangulations. In a forthcoming paper, we will use this numerical scheme to show simulations related to different kinds of glioblastoma growth.
\\

Now we introduce the hypotheses required along this section.

\begin{enumerate}[a)]
	\item Let $0<T_f<+\infty$ and a bounded set $\Omega\subseteq\mathbb{R}^2$ or $\mathbb{R}^3$ with polygonal or polyhedral lipschitz-continuous boundary. We consider the uniform time partition $$\displaystyle\left(0,T_f\right]=\bigcup^{K_f-1}_{k=0}\left(t_{k},t_{k+1}\right],$$ with $t_k=k\:dt$ where $K_f\in\mathbb{N}$ and $dt=\dfrac{T_f}{K_f}$ is the time step.
	\item  Let $\left\{\mathcal{T}_h\right\}_{h>0}$ be a family of shape-regular, quasi-uniform triangulations of $\overline{\Omega}$ formed by acute N-simplexes (triangles in $2$D and tetrahedral in $3$D), such that $$\overline{\Omega}=\displaystyle\bigcup_{\mathcal{K}\in\mathcal{T}_h}\mathcal{K},$$ where $h=\displaystyle\max_{\mathcal{K}\in\mathcal{T}_h}h_{\mathcal{K}}$, with $h_{\mathcal{K}}$ being the diameter of $\mathcal{K}$. Further, let $\mathcal{N}_h=\left\{a_i\right\}_{i\in I}$ be the set of all the nodes of $\mathcal{T}_h$. 
	\item Conforming piecewise linear, finite element spaces associated to $\mathcal{T}_h$ are assumed for approximating $H^1\left(\Omega\right)$:
	$$N_h=\left\{n_h\in\mathcal{C}^0\left(\overline{\Omega}\right)\;\;:\;\;n_h\vert_{\mathcal{K}}\in\mathcal{P}_1\left(\mathcal{K}\right)\;\;\forall\; \mathcal{K}\in\mathcal{T}_h\right\}$$
	whose Lagrange basis is denoted by $\left\{\varphi_a\right\}_{a\in\mathcal{N}_h}$.
\end{enumerate}

Let $I_h:\mathcal{C}^0\left(\overline{\Omega}\right)\rightarrow N_h$ be the nodal interpolation operator and consider the discrete inner product
$$\left(n_h,\overline{n}_h\right)_h=\int_{\Omega}I_h\left(n_h\cdot\overline{n}_h\right)=\sum_{a\in \mathcal{N}_h}n_h\left(a\right)\;\overline{n}_h\left(a\right)\int_{\Omega}\varphi_a,\quad\forall n_h,\overline{n}_h\in N_h$$
which induces the discrete norm $\| n_h\|_h=\sqrt{\left(n_h,n_h\right)_h}$ defined on $N_h$ (that is equivalent to $L^2\left(\Omega\right)$-norm).
\\

Thus, in each time step, we consider the following linear uncoupled numerical scheme for the model $\left(\ref{probOriginal}\right)$: given $T^k_h, N^k_h, \Phi^k_h\in N_h$, find $T_{h}^{k+1}, N_{h}^{k+1}, \Phi_{h}^{k+1}\in N_h$ in a decoupled way (first $T$, then $\Phi$ and finally $N$) satisfying
\begin{equation}\label{eqT_space}
\begin{array}{ccl}
\left(\delta_t T_{h}^{k+1}, v\right)_h+\left(\left(\kappa_1\;P\left(\Phi_{h}^{k},T_{h}^{k}\right)+\kappa_0\right)\nabla\;T_{h}^{k+1},\nabla v\right)&=&\left(\widehat{f}_1\left(T_{h}^{k},T_{h}^{k+1},N_{h}^{k},\Phi_{h}^{k}\right),v\right)_h
\end{array}
\end{equation}

\begin{equation}\label{eqN_space}
\begin{array}{ccl}
%\left(\delta_t N_{h}^{k+1}, v_2\right)_h&=&\left(\widetilde{f}_2\left(T_{h}^{k},T_{h}^{k+1},N_{h}^{k},\Phi_{h}^{k},\Phi_{h}^{k+1}\right),v_2\right)_h\\
\delta_t N_{h}^{k+1}\left(a\right) &=&\widehat{f}_2\left(T_{h}^{k}\left(a\right),T_{h}^{k+1}\left(a\right),N_{h}^{k}\left(a\right),\Phi_{h}^{k}\left(a\right),\Phi_{h}^{k+1}\left(a\right)\right)\\
\end{array}
\end{equation}

\begin{equation}\label{eqF_space}
\begin{array}{ccl}
%\left(\delta_t \Phi_{h}^{k+1}, v_3\right)_h&=&\left(\widetilde{f}_3\left(T_{h}^{k},T_{h}^{k+1},N_{h}^{k},\Phi_{h}^{k},\Phi_{h}^{k+1}\right),v_3\right)_h
\delta_t \Phi_{h}^{k+1}\left(a\right) &=&\widehat{f}_3\left(T_{h}^{k}\left(a\right),T_{h}^{k+1}\left(a\right),N_{h}^{k}\left(a\right),\Phi_{h}^{k}\left(a\right),\Phi_{h}^{k+1}\left(a\right)\right)\\
\end{array}
\end{equation}
$\forall v\in N_h$ and $\forall a\in\mathcal{N}_h$. We have denoted
$$\delta_t T_{h}^{k+1}=\dfrac{T_h^{k+1}-T_h^{k}}{dt}$$ 
and similarly for $\delta_t N_h^{k+1}$ and $\delta_t \Phi_h^{k+1}$. The approximation of the initial conditions are taken as
\begin{equation}\label{condinicio_espacio}
T^0_{h}=I_h\left(T_0\right)\in N_h,\;\;N^0_{h}=I_h\left(N_0\right)\in N_h ,\;\;\Phi^0_{h}=I_h\left(\Phi_0\right)\in N_h
\end{equation}
where we consider for simplicity that $T_0,\;N_0,\Phi_0\in\mathcal{C}^0\left(\overline{\Omega}\right)$. 
\\

Finally, the functions $\widehat{f}_i$ for $i=1,2,3$ which appear in $\left(\ref{eqT_space}\right)$, $\left(\ref{eqN_space}\right)$ and $\left(\ref{eqF_space}\right)$, have the following definitions:

\begin{equation}\label{f1_space}
\begin{array}{ll}
\widehat{f}_1\left(T_{h}^{k},T_{h}^{k+1},N_{h}^{k},\Phi_{h}^{k}\right)=&\rho\;P\left(\Phi^k_h,T^k_h\right)\left[ \;T^{k}_h\left(1-\dfrac{T^{k+1}_h}{K}\right)-T^{k+1}_h\left(\dfrac{N^{k}_h+\Phi^{k}_h}{K}\right)\right]-\\
\\
\displaystyle &-\alpha\;T^{k+1}_h\sqrt{1-P\left(\Phi^k_h,T^k_h\right)^2}-\beta_1\;N^k_h\;T^{k+1}_h,
\end{array}
\end{equation}

\begin{equation}\label{f2_space}
\begin{array}{ll}
\widehat{f}_2\left(T_{h}^{k},T_{h}^{k+1},N_{h}^{k},\Phi_{h}^{k},\Phi_{h}^{k+1}\right)=&\displaystyle \alpha\;T^{k+1}_h\sqrt{1-P\left(\Phi^k_h,T^k_h\right)^2} + \beta_1\;N^{k}_h\;T^{k+1}_h+ \delta\;T^{k+1}_h\;\Phi^{k+1}_h+\\
\\
& +\beta_2\;N^{k}_h\;\Phi^{k+1}_h,\\
\end{array}
\end{equation}	

\begin{equation}\label{f3_space}
\begin{array}{ll}
\widehat{f}_3\left(T_{h}^{k},T_{h}^{k+1},N_{h}^{k},\Phi_{h}^{k},\Phi_{h}^{k+1}\right)=&\gamma\;\dfrac{T^{k+1}_h}{K}\;\sqrt{1-P\left(\Phi^k_h,T^k_h\right)^2}\left[\Phi^{k}_h\left(1-\dfrac{\Phi^{k+1}_h}{K}\right)-\right.\\
\\
&\displaystyle\left.-\Phi^{k+1}_h\left(\dfrac{T^k_h+N^k_h}{K}\right)\right]-  \delta\;T^{k+1}_h\;\Phi^{k+1}_h-\beta_2\;N^{k}_h\;\Phi^{k+1}_h. \\
\end{array}
\end{equation}

The discretization of $\left(\ref{f1_space}\right)$-$\left(\ref{f3_space}\right)$ is based in two main ideas:
\begin{enumerate}
	\item We take an approximation of the negative reaction terms in a linear semi-implicit form %($T_h^{k+1}$ in $\left(\ref{f1_space}\right)$ and $\Phi_h^{k+1}$ in $\left(\ref{f3_space}\right)$) 
	and an explicit approximation of the positive reaction terms.
	\item The sum of non-logistic reaction terms of $\left(\ref{f1_space}\right)$-$\left(\ref{f3_space}\right)$ cancels, as in the continuous case.

\end{enumerate}

\begin{obs}
	Observe that $\left(\ref{eqN_space}\right)$ and $\left(\ref{eqF_space}\right)$ can be rewritten in a variational sense as follows:
	
	\begin{equation}\label{eqN_space2}
	\begin{array}{ccl}
	%\left(\delta_t N_{h}^{k+1}, v_2\right)_h&=&\left(\widetilde{f}_2\left(T_{h}^{k},T_{h}^{k+1},N_{h}^{k},\Phi_{h}^{k},\Phi_{h}^{k+1}\right),v_2\right)_h\\
	\left(\delta_t N_{h}^{k+1},v_2\right)_h &=&\left(\widehat{f}_2\left(T_{h}^{k},T_{h}^{k+1},N_{h}^{k},\Phi_{h}^{k},\Phi_{h}^{k+1}\right),v_2\right)_h\\
	\end{array}
	\end{equation}
	
	\begin{equation}\label{eqF_space2}
	\begin{array}{ccl}
	%\left(\delta_t \Phi_{h}^{k+1}, v_3\right)_h&=&\left(\widetilde{f}_3\left(T_{h}^{k},T_{h}^{k+1},N_{h}^{k},\Phi_{h}^{k},\Phi_{h}^{k+1}\right),v_3\right)_h
	\left(\delta_t \Phi_{h}^{k+1},v_3\right)_h &=&\left( \widehat{f}_3\left(T_{h}^{k},T_{h}^{k+1},N_{h}^{k},\Phi_{h}^{k},\Phi_{h}^{k+1}\right),v_3\right)_h\\
	\end{array}
	\end{equation}
	$\forall v_i\in N_h$ for $i=2,3$.
\end{obs}
\subsection{A priori energy estimates}

In this part, we are going to get a priori energy estimates for the fully discrete solution $T_h^{k+1}$, $N_h^{k+1}$ and $\Phi_{h}^{k+1}$ of $\left(\ref{eqT_space}\right)$, $\left(\ref{eqN_space}\right)$ and $\left(\ref{eqF_space}\right)$ which are independent of $(h, k)$. 
The following two lemmas are based on the hypothesis of acute triangulations to get a discrete maximum principle, see  \cite{Ciarlet_1973}.

\begin{lem}[Lower bounds; positivity]\label{positivo_space}
	%Suppose the hypotheses defined in section \ref{esquema_espacio} and 
	Let $T_h^k,\;N_h^k,\;\Phi_{h}^k\in N_h$ such that $0\leq T_h^k,\;N_h^k,\;\Phi_{h}^k$ in $\Omega$. Then, $T_h^{k+1},\;N_h^{k+1}\;\Phi_{h}^{k+1}\geq0$ in $\Omega$.
	
\end{lem}
\begin{proof}
	Let $I_h((T_h^{k+1})_-)\in N_h$ be defined as 
	$$I_h\left(\left(T_h^{k+1}\right)_-\right)=\sum_{a\in \mathcal{N}_h}\left(T_h^{k+1}\left(a\right)\right)_-\varphi_a,$$
	where $\left(T_h^{k+1}\left(a\right)\right)_-=\min\left\{0,T_h^{k+1}\left(a\right)\right\}$. Analogously, one defines $I_h((T_h^{k+1})_+)\in N_h$ as
	$$I_h\left(\left(T_h^{k+1}\right)_+\right)=\sum_{a\in \mathcal{N}_h}\left(T_h^{k+1}\left(a\right)\right)_+\varphi_a,$$
	where $\left(T_h^{k+1}\left(a\right)\right)_+=\max\left\{0,T_h^{k+1}\left(a\right)\right\}$. Notice that $T_h^{k+1}=I_h((T_h^{k+1})_-)+I_h((T_h^{k+1})_+)$.
	\\
	
	On choosing $v=I_h((T_h^{k+1})_-)$ in $\left(\ref{eqT_space}\right)$, it follows that
	
	\begin{equation}\label{T_space_positiva}
	\begin{array}{c}
	\dfrac{1}{dt}\Big\|\left(T_h^{k+1}\right)_-\Big\|_h^2 +\left(\left(\kappa_1\;P\left(\Phi_{h}^{k},T_{h}^{k}\right)+\kappa_0\right)\nabla T_h^{k+1},\nabla I_h\left(\left(T_h^{k+1}\right)_-\right)\right)\leq\\
	\\
	\leq\left(\widehat{f}_1\left(T_{h}^{k},T_{h}^{k+1},N_{h}^{k},\Phi_{h}^{k}\right),\left(T_h^{k+1}\right)_-\right)_h,
	\end{array}
	\end{equation}
	where we have used in the left hand side that in every node $a\in \mathcal{N}_h$, $$\delta_t T_{h}^{k+1}\left(a\right)\cdot\left(T_h^{k+1}\left(a\right)\right)_-=\dfrac{1}{dt}\left(\Big|\left(T_{h}^{k+1}\left(a\right)\right)_-\Big|^2-T_{h}^{k}\left(a\right)\cdot\left(T_h^{k+1}\left(a\right)\right)_-\right)\geq\dfrac{1}{dt}\left(\Big|\left(T_{h}^{k+1}\left(a\right)\right)_-\Big|^2\right)$$ using that $T_{h}^{k}\left(a\right)\geq0$ and $\left(T_h^{k+1}\left(a\right)\right)_-\leq0$. Now, we can make the following
	$$%\begin{array}{ll}
	\left(\left(\kappa_1\;P\left(\Phi_{h}^{k},T_{h}^{k}\right)+\kappa_0\right)\nabla T_h^{k+1},\nabla I_h\left(\left(T_h^{k+1}\right)_-\right)\right) =$$
	$$=\left(\left(\kappa_1\;P\left(\Phi_{h}^{k},T_{h}^{k}\right)+\kappa_0\right)\nabla I_h\left(\left(T_h^{k+1}\right)_-\right),\nabla I_h\left(\left(T_h^{k+1}\right)_-\right)\right)+$$
	$$+\left(\left(\kappa_1\;P\left(\Phi_{h}^{k},T_{h}^{k}\right)+\kappa_0\right)\nabla I_h\left(\left(T_h^{k+1}\right)_+\right),\nabla I_h\left(\left(T_h^{k+1}\right)_-\right)\right)
	=$$
	$$=\Big\|\left(\kappa_1\;P\left(\Phi_{h}^k,T_h^k\right)+\kappa_0\right)^{1/2}\nabla I_h\left(\left(T_h^{k+1}\right)_-\right)\Big\|_{L^2\left(\Omega\right)}^2+$$
	$$\displaystyle+\sum_{a\neq\widetilde{a}\in \mathcal{N}_h} \left(T_h^{k+1}\left(a\right)\right)_- \left(T_h^{k+1}\left(\widetilde{a}\right)\right)_+\left(\left(\kappa_1\;P\left(\Phi_{h}^{k},T_{h}^{k}\right)+\kappa_0\right)\nabla\varphi_a,\nabla\varphi_{\widetilde{a}}\right).
	$$
	%\end{array}$$
	
	Hence, using that $\left(T_h^{k+1}\left(a\right)\right)_- \;\left(T_h^{k+1}\left(\widetilde{a}\right)\right)_+\leq0$ if $a\neq\widetilde{a}$, $\left(\kappa_1\;P\left(\Phi_{h}^{k},T_{h}^{k}\right)+\kappa_0\right)$ is a nonnegative function and that
	%$$\left(\kappa_1\;P\left(\Phi_{h}^{k},T_{h}^{k}\right)+\kappa_0\right)$$
	$$\nabla\varphi_a\cdot\nabla\varphi_{\widetilde{a}}\leq0\quad\forall a\neq\widetilde{a}\in \mathcal{N}_h$$ (owing to an acute triangulation is assumed), we deduce,% from $\left(\ref{T_space_positiva}\right)$, 
	\begin{equation}\label{T_diagonal}
	\begin{array}{c}
	\left(\left(\kappa_1\;P\left(\Phi_{h}^{k},T_{h}^{k}\right)+\kappa_0\right)\nabla T_h^{k+1},\nabla I_h\left(\left(T_h^{k+1}\right)_-\right)\right)\geq\\
	\\
	\geq\Big\|\left(\kappa_1\;P\left(\Phi_{h}^k,T_h^k\right)+\kappa_0\right)^{1/2}\nabla I_h\left(\left(T_h^{k+1}\right)_-\right)\Big\|_{L^2\left(\Omega\right)}^2.
	\end{array}
	\end{equation}
	
	Adding $\left(\ref{T_diagonal}\right)$ in $\left(\ref{T_space_positiva}\right)$, it holds that
	
	\begin{equation}\label{positividad_T_numerico}
	\begin{array}{c}
	\dfrac{1}{dt}\Big\|\left(T_h^{k+1}\right)_-\Big\|_h^2 +\Big\|\left(\kappa_1\;P\left(\Phi_{h}^k,T_h^k\right)+\kappa_0\right)^{1/2}\nabla I_h\left(\left(T_h^{k+1}\right)_-\right)\Big\|_{L^2\left(\Omega\right)}^2\leq\\
	\\
	\leq\left(\widehat{f}_1\left(T_{h}^{k},T_{h}^{k+1},N_{h}^{k},\Phi_{h}^{k}\right),\left(T_h^{k+1}\right)_-\right)_h\leq 0.
		\end{array}
		\end{equation}
		
	For the last inequality above, we used that in every node $a\in \mathcal{N}_h$ we have, due to the form of $\widehat{f}_1$ given in $\left(\ref{f1_space}\right)$, the following 
	$$\rho\;P\left(\Phi^k_h\left(a\right),T^k_h\left(a\right)\right) \left(T_h^{k}\left(a\right)\right)\left(T_h^{k+1}\left(a\right)\right)_-\leq0$$
	% C_1\sum_{a\in N_h}I_h\left(\left(T_h^{k+1}\left(a\right)\right)_-\right)\;T_h^{k}\left(a\right)\leq0$$ 
	and
	$$- \left(\rho\;P\left(\Phi^k_h\left(a\right),T^k_h\left(a\right)\right)\left(\dfrac{T^k_h\left(a\right)+N^{k}_h\left(a\right)+\Phi^{k}_h\left(a\right)}{K}\right)+\alpha\sqrt{1-P\left(\Phi^k_h\left(a\right),T^k_h\left(a\right)\right)^2}\right.$$
	$$\left.+\beta_1\;N^k_h\left(a\right)\right)\left(T_h^{k+1}\left(a\right)\right)\left(T_h^{k+1}\left(a\right)\right)_-\leq 0.$$
	%-C_2 \sum_{a\in N_h}I_h\left(\left(T_h^{k+1}\left(a\right)\right)_-\right)\;T_h^{k+1}\left(a\right)\leq0.$$
	
	Therefore, from $\left(\ref{positividad_T_numerico}\right)$, $\left(T_h^{k+1}\right)_-\equiv0$ and this implies $T_h^{k+1}\geq0$ in $\Omega$.
	\\
	
	For $\left(\ref{eqF_space}\right)$, the same argument can be used and it is even easier. Thus, multiplying $\left(\ref{eqF_space}\right)$ by $(\Phi_h^{k+1}(a))_-$,
	%Hence
%	$$\dfrac{1}{2}\;\dfrac{1}{dt}\Big\|\left(\Phi_h^{k+1}\right)_-\Big\|_h^2\leq \left(\widetilde{f}_3\left(T_{h}^{k},T_{h}^{k+1},N_{h}^{k},\Phi_{h}^{k},\Phi_{h}^{k+1}\right),\left(\Phi_h^{k+1}\right)_-\right)_h\leq 0\Rightarrow \Phi_h^{k+1}\geq0.$$
%	$$\dfrac{1}{2}\;\dfrac{1}{dt}\Big\|\left(N_h^{k+1}\right)_-\Big\|_h^2\leq \left(\widetilde{f}_2\left(T_{h}^{k},T_{h}^{k+1},N_{h}^{k},\Phi_{h}^{k},\Phi_{h}^{k+1}\right),\left(N_h^{k+1}\right)_-\right)_h\leq 0\Rightarrow N_h^{k+1}\geq0.$$

	\begin{equation}\label{positividad_F_numerico}
	\dfrac{1}{dt}\;\left(\Phi_h^{k+1}\left(a\right)\right)_-^2 \leq\widehat{f}_3\left(T_{h}^{k}\left(a\right),T_{h}^{k+1}\left(a\right),N_{h}^{k}\left(a\right),\Phi_{h}^{k}\left(a\right),\Phi_{h}^{k+1}\left(a\right)\right)\;\left(\Phi_h^{k+1}\left(a\right)\right)_-\leq 0
	\end{equation}\
	since in every node $a\in \mathcal{N}_h$ we have, due to the form of $\widehat{f}_3$ given in $\left(\ref{f3_space}\right)$, the following 
	
	$$\gamma\;\dfrac{T^{k+1}_h\left(a\right)}{K}\;\sqrt{1-P\left(\Phi^k_h\left(a\right),T^k_h\left(a\right)\right)^2}\;\left(\Phi^{k}_h\left(a\right)\right)\left(\Phi_h^{k+1}\left(a\right)\right)_-\leq0$$
	% C_1\sum_{a\in N_h}I_h\left(\left(T_h^{k+1}\left(a\right)\right)_-\right)\;T_h^{k}\left(a\right)\leq0$$ 
	and
	$$- \left(\gamma\;\dfrac{T^{k+1}_h\left(a\right)}{K}\;\sqrt{1-P\left(\Phi^k_h\left(a\right),T^k_h\left(a\right)\right)^2}\left(\dfrac{T_h^k\left(a\right)+N^{k}_h\left(a\right)+\Phi_{h}^k\left(a\right)}{K}\right)+\delta\; T^{k+1}_h\left(a\right)\right.$$
	$$\left.+\beta_2\;N^k_h\left(a\right)\right)\left(\Phi_h^{k+1}\left(a\right)\right)\left(\Phi_h^{k+1}\left(a\right)\right)_-\leq 0.$$
	%-C_2 \sum_{a\in N_h}I_h\left(\left(T_h^{k+1}\left(a\right)\right)_-\right)\;T_h^{k+1}\left(a\right)\leq0.$$
	
	Therefore, from $\left(\ref{positividad_F_numerico}\right)$, $\left(\Phi_h^{k+1}\left(a\right)\right)_-\equiv0$ $\forall a\in\mathcal{N}_h$ and this implies $\Phi_h^{k+1}\geq0$ in $\Omega$.
	\\
	
	Finally, for $\left(\ref{eqN_space}\right)$  it is easy to obtain that 
	$$\dfrac{1}{dt}\;\left(N_h^{k+1}\left(a\right)\right)_-^2 \leq\widehat{f}_2\left(T_{h}^{k}\left(a\right),T_{h}^{k+1}\left(a\right),N_{h}^{k}\left(a\right),\Phi_{h}^{k}\left(a\right),\Phi_{h}^{k+1}\left(a\right)\right)\;\left(N_h^{k+1}\left(a\right)\right)_-\leq 0$$
	since $\widehat{f}_2\left(T_{h}^{k}\left(a\right),T_{h}^{k+1}\left(a\right),N_{h}^{k}\left(a\right),\Phi_{h}^{k}\left(a\right),\Phi_{h}^{k+1}\left(a\right)\right)\geq0$ in every node $a\in\mathcal{N}_h$ due to the form of $\widehat{f}_2$ given in $\left(\ref{f2_space}\right)$. Hence, $\left(N_h^{k+1}\left(a\right)\right)_-\equiv0$ $\forall a\in\mathcal{N}_h$ and this implies $N_h^{k+1}\geq0$ in $\Omega$.
	
\end{proof}

\begin{lem}[Upper bounds]\label{cota_superior_space}
	%Suppose the hypotheses defined in section \ref{esquema_espacio} and 
	Let $T_h^k,\;N_h^k,\;\Phi_{h}^k\in N_h$ such that $0\leq T_h^k,\;\Phi_{h}^k\leq K$ and $0\leq N_h^k$ in $\Omega$. Then one has 
	\begin{enumerate}[a)]
		\item $T^{k+1}_h$, $\Phi^{k+1}_h\leq K$ in $\Omega$.
		\item $N^k_h\leq N^{k+1}_h$ in $\Omega$.
		\item $N^k_h\leq \widetilde{C}\left(%\parallel N_h^0\parallel_{L^\infty\left(\Omega\right)},
		T_f\right)$ in $\Omega$, for all $k=1,\cdots,K_f$, with $\widetilde{C}$ independent of $\left(h,k\right)$.
	\end{enumerate}
	\begin{proof}
		\begin{enumerate}[a)]
			\item We argue in a similar fashion of Lemma \ref{positivo_space}. In this case, by writing $\left(\ref{eqT_space}\right)$ as
			
			$$\left(\delta_t \left(T_{h}^{k+1}-K\right), v\right)_h+\left(\left(\kappa_1\;P\left(\Phi_{h}^{k},T_{h}^{k}\right)+\kappa_0\right)\nabla\;\left(T_{h}^{k+1}-K\right),\nabla v\right)=\left(\widehat{f}_1\left(T_{h}^{k},T_{h}^{k+1},N_{h}^{k},\Phi_{h}^{k}\right),v\right)_h$$
			
			and taking $v=I_h((T_h^{k+1}-K)_+)$, it follows that
			
			$$
		\dfrac{1}{dt}\Big\|\left(T_h^{k+1}-K\right)_+\Big\|_h^2 +\Big\|\left(\kappa_1\;P\left(\Phi_{h}^k,T_h^k\right)+\kappa_0\right)^{1/2}\nabla I_h\left(\left(T_h^{k+1}-K\right)_+\right)\Big\|_{L^2\left(\Omega\right)}^2\leq$$
			$$\leq\left(\widehat{f}_1\left(T_{h}^{k},T_{h}^{k+1},N_{h}^{k},\Phi_{h}^{k}\right),\left(T_h^{k+1}-K\right)_+\right)_h\leq 0
			$$
			since in every node $a\in \mathcal{N}_h$ we have on one side that
			$$\delta_t \left(\left(T_{h}^{k+1}-K\right)\left(a\right)\right)\cdot\left(\left(T_h^{k+1}-K\right)\left(a\right)\right)_+=\Big|\left(\left(T_h^{k+1}-K\right)\left(a\right)\right)_+\Big|^2-$$
			$$-\left(\left(T_h^{k}-K\right)\left(a\right)\right)\cdot\left(\left(T_h^{k+1}-K\right)\left(a\right)\right)_+\geq\Big|\left(\left(T_h^{k+1}-K\right)\left(a\right)\right)_+\Big|^2$$
		using that $\left(\left(T_h^{k}-K\right)\left(a\right)\right)\leq0$ and $\left(\left(T_h^{k+1}-K\right)\left(a\right)\right)_+\geq0$. On other side, in every node $a\in \mathcal{N}_h$, due to the form of $\widehat{f}_1$ given in $\left(\ref{f1_space}\right)$, the following 
			
			$$\left(\rho\;P\left(\Phi^k_h\left(a\right),T^k_h\left(a\right)\right)T^k_h\left(a\right) \left(1-\dfrac{T^{k+1}_h\left(a\right)}{K}\right)\right)\left(\left(T_h^{k+1}-K\right)\left(a\right)\right)_+\leq0$$ 
			%$$\leq C_1\sum_{a\in N_h}I_h\left(\left(\left(T_h^{k+1}-K\right)\left(a\right)\right)_+\right)\;\left(1-\dfrac{T^{k+1}_h\left(a\right)}{K}\right)\leq0$$ 
			and
			$$- \left(\rho\;P\left(\Phi^k_h\left(a\right),T^k_h\left(a\right)\right)\left(\dfrac{N^{k}_h\left(a\right)+\Phi^{k}_h\left(a\right)}{K}\right)+\alpha\sqrt{1-P\left(\Phi^k_h\left(a\right),T^k_h\left(a\right)\right)^2}+\right.$$
			$$\left.+\beta_1\;N^k_h\left(a\right)\right)\left(T_h^{k+1}\left(a\right)\right)\left(\left(T_h^{k+1}-K\right)\left(a\right)\right)_+\leq0.$$
			%$$\leq -C_2 \sum_{a\in N_h}I_h\left(\left(\left(T_h^{k+1}-K\right)\left(a\right)\right)_+\right)\;T_h^{k+1}\left(a\right)\leq0.$$
			
			Hence, $\left(T_h^{k+1}-K\right)_+\equiv0$ and this implies $T_h^{k+1}\leq K$ in $\Omega$.
			\\
			
			With a similar reasoning, now for $\left(\ref{eqF_space}\right)$, we get 
			$$
			\dfrac{1}{dt}\;\left(\left(\Phi_h^{k+1}-K\right)\left(a\right)\right)_+^2 \leq\widehat{f}_3\left(T_{h}^{k}\left(a\right),T_{h}^{k+1}\left(a\right),N_{h}^{k}\left(a\right),\Phi_{h}^{k}\left(a\right),\Phi_{h}^{k+1}\left(a\right)\right)\;\left(\Phi_h^{k+1}\left(a\right)-K\right)_+\leq 0
			$$
			%since for $\Phi_h^{k+1}\leq K$ the right side is negative. 
			Hence, $\left(\Phi_h^{k+1}\left(a\right)-K\right)_+\equiv0$ $\forall a\in\mathcal{N}_h$ and this implies $\Phi_h^{k+1}\leq K$ in $\Omega$.
			
			\item Using that $T_h^k,\;T_h^{k+1},\;\Phi_{h}^k,\;\Phi_{h}^{k+1},\;N_h^k\geq0$, we can estimate $\left(\ref{eqN_space}\right)$ as follows %using the bounds got before for $T_h^k,\,T_h^{k+1},\;\Phi_h^k$ and $\Phi_h^{k+1}$ so we will do
			
			\begin{equation}\label{boundN}
			\begin{array}{ll}
			\displaystyle \dfrac{N_h^{k+1}\left(a\right)-N_h^k\left(a\right)}{dt} =&\alpha\;T_h^{k+1}\left(a\right)\sqrt{1-P\left(\Phi_h^k\left(a\right),T_h^k\left(a\right)\right)^2} + \beta_1\;N_h^{k}\left(a\right)\;T_h^{k+1}\left(a\right)+\\
			\\
			&+\delta\;T_h^{k+1}\left(a\right)\;\Phi_h^{k+1}\left(a\right)+ \beta_2\;N_h^{k}\left(a\right)\;\Phi_h^{k+1}\left(a\right)\geq 0.% \mathcal{C}_1\;N^k+\mathcal{C}_2\Rightarrow\\
			%&\Rightarrow \delta_t N^{k+1}\leq \mathcal{C}_1\;N^k+\mathcal{C}_2
			\end{array}\end{equation}
			
			Hence,
			$$N_h^k\left(a\right)\leq N_h^{k+1}\left(a\right)\quad\forall k=0,\ldots, K_f-1,\;\;\text{in}\;\;\Omega.$$
			
			\item Using that $0\le T_h^k,T_h^{k+1}, \Phi_h^k,\Phi_h^{k+1}\leq K$ in $\Omega$ and for all $k=0,\ldots K_f$, we can bound $\left(\ref{eqN_space}\right)$ in the following way
			
			$$\displaystyle \dfrac{N_h^{k+1}\left(a\right)-N_h^k\left(a\right)}{dt}\leq C_1\;N_h^k\left(a\right)+C_2,\;\;\text{in}\;\;\Omega.$$
			
			Applying discrete Gronwall inequality pointwise for every $a\in \mathcal{N}_h$, it holds that $\forall k=1,\ldots K_f$
			
			\begin{equation}\label{cotaN_space}
			N_h^{k}\left(a\right)\leq N_h^0(a)\;e^{C_1\;k\;dt}+C_2\;\dfrac{e^{C_1\;k\;dt}-1}{C_1}\leq C\left(\| N_h^0\|_{L^\infty\left(\Omega\right)}, T_f\right)=\widetilde{C}\left(T_f\right).
			\end{equation}
			
			Thus, we have deduced an exponential upper bound for $N_h^{k}$, with a similar expression that in the continuous estimate obtained in Lemma $\ref{estimaciones}$a), which depends on the initial data of necrosis and the final time $T_f$ and is independent of $dt$ and $\left(h,k\right)$. %In particular, $N_h^k\leq C\left(T_f\right)$ is satisfied for $k=K_f$, hence we get an upper bound for all $N_h^k$ in $\Omega$ independent of $\left(h,k\right)$.
			
		\end{enumerate}
	\end{proof}
\end{lem}

Moreover, some a priori energy estimates will be obtained. To get these estimates, we define the piecewise functions
$$T_h^{dt}=\left\{
\begin{array}{ll}
T_h^{k+1}&\text{if}\;\;t\in\left(\left.t_k,t_{k+1}\right.\right],\\
\\
T_h^0&\text{if}\;\;t=0,
\end{array}
\right.
$$
and the same for $N_h^{dt}$ and $\Phi_{h}^{dt}$.
\begin{lem}\label{l2h1_numerico}
	Given $T_h^k, \;N_h^k,\; \Phi_h^k\in N_h$ such that $0\leq T^k_h,\;\Phi^k_h\leq K$ and $0\leq N_h^k\leq \widetilde{C}\left(T_f\right)$ in $\Omega$ with $\widetilde{C}\left(T_f\right)$ the upper finite bound defined in $\left(\ref{cotaN_space}\right)$, then
	
	%$$\Big\|T_h^{k+1}\Big\|_{L^\infty\left(0,T_f;L^2\left(\Omega\right)\right)}+
	$$\displaystyle\|T_h^{dt}\|_{L^2\left(0,T_f;H^1\left(\Omega\right)\right)}^2=dt\;\sum_{k=1}^{K_f}\|T_h^{k}\|_{H^1\left(\Omega\right)}^2\leq C$$
	with $\mathcal{C}>0$ independent of $\left(h,dt\right)$.
	
\end{lem}
\begin{proof}
	Take $v=T_h^{k+1}$ in $\left(\ref{eqT_space}\right)$ (using that $\left(a-b\right)a=\dfrac{1}{2}\left(a^2-b^2+\left(a-b\right)^2\right)\geq\dfrac{1}{2}\left(a^2-b^2\right)$ $\forall a,b\in\mathbb{R}$) and estimating the right hand side, it holds that
	
	$$\dfrac{1}{2}\dfrac{1}{dt}\left(\Big\|T_h^{k+1}\Big\|_h^2-\Big\|T_h^{k}\Big\|_h^2\right)+\int_\Omega \left(\kappa_1\; P\left(\Phi^k_h,T^k_h\right)+\kappa_0\right)\Big|\nabla T^{k+1}_h\Big|^2\leq\rho\left(T_h^k,T_h^{k+1}\right)_h\leq\rho\;K^2\mid\Omega\mid.$$
	%$$\leq\rho\left(T_h^k,T_h^{k+1}\right)_h-\dfrac{\rho}{K}\left(T_h^k,\left(T_h^{k+1}\right)^2\right)_h.$$
	
	Applying Hölder and Young's inequalities for the last right term in every node $a\in \mathcal{N}_h$, %similar as we did in Lemma \ref{energy_estimate_time}, 
	and adding in all the time steps, we get the following energy estimate
	
	%$$\dfrac{1}{2}\Big\|T_h^{k+1}\Big\|_h^2+
	$$\displaystyle\kappa_0\;dt\sum_{k=0}^{K_f-1} \|\nabla\;T_h^{k+1}\|_{L^2\left(\Omega\right)}\leq \dfrac{1}{2}\Big\|T_h^{0}\Big\|_h^2+T_f\;\rho\;K^2\mid\Omega\mid.$$ %$$dt\;\sum_{k=0}^{K_f-1}\int_\Omega \left(\kappa_1\; P\left(\Phi^k_h,T^k_h\right)+\kappa_0\right)\Big|\nabla I_h \left(T^{k+1}_h\right)\Big|^2\leq \dfrac{1}{2}\Big\|T_h^{0}\Big\|_h^2+\dfrac{T_f}{2}\;\rho\;K^2\mid\Omega\mid$$
	hence the desired bound is deduced.
\end{proof}

%Finally, some energy estimates can be obtained for $N_h^{dt}$ and $\Phi_{h}^{dt}$. For that, we define the piecewise functions
%$$N_h^{dt}=\left\{
%\begin{array}{ll}
%N_h^{k+1}&\text{if}\;\;t\in\left(\left.t_k,t_{k+1}\right.\right],\\
%\\
%N_h^0&\text{if}\;\;t=0.
%\end{array}
%\right.
%$$
%and
%$$\Phi_h^{dt}=\left\{
%\begin{array}{ll}
%\Phi_h^{k+1}&\text{if}\;\;t\in\left(\left.t_k,t_{k+1}\right.\right],\\
%\\
%\Phi_h^0&\text{if}\;\;t=0.
%\end{array}
%\right.
%$$

Before presenting the energy estimate for $N_h^{dt}$ and $\Phi_{h}^{dt}$ in $L^\infty\left(0,T_f;H^1\left(\Omega\right)\right)$ we define the Laplacian in a discrete way using the discrete $L^2$ product, that is $-\Delta_h\;n_h\in N_h$ such that $\left(-\Delta_h\;n_h,\overline{n}_h\right)_h=\left(\nabla n_h,\nabla\overline{n}_h\right)$, $\forall\overline{n}_h\in N_h$. Now, we show a result of discrete Laplacian which we will use later.

\begin{lem}\label{desigualdad_inversa_laplaciano}
	Given $-\Delta_h\;n_h\in N_h$, it holds that $$\|-\Delta_h\;n_h\|_{L^2\left(\Omega\right)}\leq C\dfrac{1}{h}\|n_h\|_{H^1\left(\Omega\right)}\;\;\forall n_h\in N_h.$$
\end{lem}

\begin{proof}
	Choosing $-\Delta_h\;n_h\in N_h$ as test function in the definition of discrete Laplacian, we obtain that
	$$\|-\Delta_h\;n_h\|_h^2=\left(-\Delta_h\;n_h,-\Delta_h\;n_h\right)_h=\left(\nabla\;n_h,\nabla\left(-\Delta_h\;n_h\right)\right)\leq \|\nabla\;n_h\|_{L^2\left(\Omega\right)}\;\dfrac{1}{h}\;\|-\Delta_h\;n_h\|_{L^2\left(\Omega\right)}$$
	where we have used the inverse inequality $\|\overline{n}_h\|_{H^1\left(\Omega\right)}\leq C\dfrac{1}{h}\;\|\overline{n}_h\|_{L^2\left(\Omega\right)}$ $\forall\overline{n}\in N_h$. On other hand, we have that $\|\cdot\|_h$ and $\|\cdot\|_{L^2\left(\Omega\right)}$ are equivalent norms, hence $\|-\Delta_h\;n_h\|_h^2\ge C\|-\Delta_h\;n_h\|_{L^2\left(\Omega\right)}^2$.
	\\
	
	Finally, we deduce
	
	$$\|-\Delta_h\;n_h\|_{L^2\left(\Omega\right)}\leq C\dfrac{1}{h}\|\nabla\;n_h\|_{L^2\left(\Omega\right)}.$$

\end{proof}
\begin{lem}
	Given $T_h^k, \;N_h^k,\; \Phi_h^k\in N_h$ such that $0\leq T^k_h,\;\Phi^k_h\leq K$ and $0\leq N_h^k\leq \widetilde{C}\left(T_f\right)$ in $\Omega$ with $\widetilde{C}\left(T_f\right)$ the upper finite bound defined in $\left(\ref{cotaN_space}\right)$, then for small enough $dt$, one has
	
	%$$\Big\|T_h^{k+1}\Big\|_{L^\infty\left(0,T_f;L^2\left(\Omega\right)\right)}+
	$$\Big\|N_h^{dt},\Phi_{h}^{dt}\Big\|_{L^\infty\left(0,T_f;H^1\left(\Omega\right)\right)}\leq\mathcal{C}$$
	with $\mathcal{C}>0$, independent of $\left(h,dt\right)$.
	%and
	%$$\Big\|\Phi_h^{dt}\Big\|_{L^\infty\left(0,T_f;H^1\left(\Omega\right)\right)}\leq\mathcal{C}_2$$
	
	%with $\mathcal{C}_i>0$ for $i=1,2$, independent of $\left(h,k\right)$.
	
\end{lem}

\begin{proof}
	We make the proof for $N_h^{dt}$ since for $\Phi_{h}^{dt}$ is similar. By multiplying by $-\Delta_h\;N_h^{k+1}$  in $\left(\ref{eqN_space2}\right)$ , it holds that
	
	\begin{equation}\label{cotaLinfH1}
	\dfrac{1}{2}\dfrac{1}{dt}\left(\Big\|\nabla N_h^{k+1}\Big\|_{h}^2-\Big\|\nabla N_h^{k}\Big\|_{h}^2\right)\leq\left( \widehat{f}_2\left(T_{h}^{k},T_{h}^{k+1},N_{h}^{k},\Phi_{h}^{k},\Phi_{h}^{k+1}\right),-\Delta_h N_h^{k+1}\right)_h.
	\end{equation}

	For the right hand side, we use an extension of the Scott-Zhang interpolation operator $\mathcal{Q}_h$ from $L^2\left(\Omega\right)$ to $N_h$ (see \cite[Proposition 2.4]{Kisko_2019} and the references therein) in the following way
	%$P_h\left(\widetilde{f}_2\left(T_{h}^{k},T_{h}^{k+1},N_{h}^{k},\Phi_{h}^{k},\Phi_{h}^{k+1}\right)\right)\in N_h$ (that is, if $u\in L^2\left(\Omega\right)$, $P_h\left(u\right)\in N_h$ and $\left(u-P_h\left(u\right),\overline{u}\right)_h=0$, $\forall \overline{u}\in N_h$). Hence,
	
	\begin{equation}\label{cota_proyeccion}
	\begin{array}{c}
	\left( \widehat{f}_2,-\Delta_h N_h^{k+1}\right)_h=\left(\widehat{f}_2- \mathcal{Q}_h\left(\widehat{f}_2\right),-\Delta_h N_h^{k+1}\right)_h+\left(\mathcal{Q}_h\left(\widehat{f}_2\right),-\Delta_h N_h^{k+1}\right)_h
		\end{array}
	\end{equation}
	where we denoted $\widehat{f}_2=\widehat{f}_2\left(T_{h}^{k},T_{h}^{k+1},N_{h}^{k},\Phi_{h}^{k},\Phi_{h}^{k+1}\right)$ in order to simplify the notation. 
	\\
	
	Now, we bound $\left(\ref{cota_proyeccion}\right)$ using that $\|\widehat{f}_2-\mathcal{Q}_h\left(\widehat{f}_2\right)\|_{L^2\left(\Omega\right)}\leq C h\|\widehat{f}_2\|_{H^1\left(\Omega\right)}$, $\|\mathcal{Q}_h\left(\widehat{f}_2\right)\|_{H^1\left(\Omega\right)}\leq C\|\widehat{f}_2\|_{H^1\left(\Omega\right)}$ and Lemma $\ref{desigualdad_inversa_laplaciano}$ to obtain that
	
	$$
	\left(\widehat{f}_2- \mathcal{Q}_h\left(\widehat{f}_2\right),-\Delta_h N_h^{k+1}\right)_h+\left(\mathcal{Q}_h\left(\widehat{f}_2\right),-\Delta_h N_h^{k+1}\right)_h\leq\left(\widehat{f}_2- \mathcal{Q}_h\left(\widehat{f}_2\right),-\Delta_h N_h^{k+1}\right)_h+$$
	$$+\left(\nabla \mathcal{Q}_h\left(\widehat{f}_2\right),\nabla N_h^{k+1}\right)\leq C\;h\;\Big\| \widehat{f}_2\Big\|_{H^1\left(\Omega\right)}\;\dfrac{1}{h}\Big\|\nabla N_h^{k+1}\Big\|_{L^2\left(\Omega\right)}+C\;\Big\| \widehat{f}_2\Big\|_{H^1\left(\Omega\right)}\Big\|\nabla N_h^{k+1}\Big\|_{L^2\left(\Omega\right)}\leq$$
	$$\leq C \;\Big\| \widehat{f}_2 \Big\|_{H^1\left(\Omega\right)}\;\Big\|\nabla N_h^{k+1}\Big\|_{L^2\left(\Omega\right)}.$$
	
	In these circumstances, we can follow a similar argument to $\left(\ref{acotacion_laplaciano}\right)$ in an discrete way 
	$$C \;\Big\| \widehat{f}_2 \Big\|_{H^1\left(\Omega\right)}\;\Big\|\nabla N_h^{k+1}\Big\|_{L^2\left(\Omega\right)}\leq C\left(1+\Big\|\nabla T_h^{k},\nabla T_h^{k+1},\nabla N_h^{k},\nabla \Phi_h^{k},\nabla \Phi_h^{k+1}\Big\|_{L^2\left(\Omega\right)}\right)\Big\|\nabla N_h^{k+1}\Big\|_{L^2\left(\Omega\right)}\leq$$
	$$\leq C\left(1+\Big\|\nabla T_h^{k},\nabla T_h^{k+1},\nabla N_h^{k},\nabla N_h^{k+1},\nabla \Phi_h^{k},\nabla \Phi_h^{k+1}\Big\|_{L^2\left(\Omega\right)}^2\right).$$
	
	%Since $\widetilde{f}_2\left(T_{h}^{k},T_{h}^{k+1},N_{h}^{k},\Phi_{h}^{k},\Phi_{h}^{k+1}\right)$ and $\nabla\widetilde{f}_2\left(T_{h}^{k},T_{h}^{k+1},N_{h}^{k},\Phi_{h}^{k},\Phi_{h}^{k+1}\right)$ are bounded in $L^2\left(0,T_f;L^2\left(\Omega\right)\right)$, we get

	%$$\mathcal{C}_1 \;\Big\| \widetilde{f}_2\left(T_{h}^{k},T_{h}^{k+1},N_{h}^{k},\Phi_{h}^{k},\Phi_{h}^{k+1}\right) \Big\|_{H^1\left(\Omega\right)}\;\Big\|\nabla N_h^{k+1}\Big\|_{L^2\left(\Omega\right)}\leq$$
	%$$\leq\mathcal{C}_1\left(\Big\|\nabla T_h^{k}\Big\|_{L^2\left(\Omega\right)}+\Big\|\nabla T_h^{k+1}\Big\|_{L^2\left(\Omega\right)}+\Big\|\nabla N_h^{k}\Big\|_{L^2\left(\Omega\right)}+\Big\|\nabla \Phi_h^{k}\Big\|_{L^2\left(\Omega\right)}+\Big\|\nabla \Phi_h^{k+1}\Big\|_{L^2\left(\Omega\right)}\right)\Big\|\nabla N_h^{k+1}\Big\|_{L^2\left(\Omega\right)}\leq$$
	%$$\leq\widetilde{\mathcal{C}}_1\left(\Big\|\nabla T_h^{k}\Big\|_{L^2\left(\Omega\right)}^2+\Big\|\nabla T_h^{k+1}\Big\|_{L^2\left(\Omega\right)}^2+\Big\|\nabla N_h^{k}\Big\|_{L^2\left(\Omega\right)}^2+\Big\|\nabla \Phi_h^{k}\Big\|_{L^2\left(\Omega\right)}^2+\Big\|\nabla \Phi_h^{k+1}\Big\|_{L^2\left(\Omega\right)}^2\right).$$
	
	Hence,
	
	\begin{equation}\label{cotaLinfH1_N}
	\begin{array}{l}
	\dfrac{1}{2}\dfrac{1}{dt}\left(\Big\|\nabla N_h^{k+1}\Big\|_{h}^2-\Big\|\nabla N_h^{k}\Big\|_{h}^2\right)\leq C\left(1+\Big\|\nabla T_h^{k},\nabla T_h^{k+1},\nabla N_h^{k},\nabla N_h^{k+1},\nabla \Phi_h^{k},\nabla \Phi_h^{k+1}\Big\|_{L^2\left(\Omega\right)}^2\right).
	\end{array}
	\end{equation}
	
	We can obtain a similar expression for $\Phi_h^{dt}$
	
	\begin{equation}\label{cotaLinfH1_Phi}
	\begin{array}{l}
	\dfrac{1}{2}\dfrac{1}{dt}\left(\Big\|\nabla \Phi_h^{k+1}\Big\|_{h}^2-\Big\|\nabla \Phi_h^{k}\Big\|_{h}^2\right)\leq C\left(1+\Big\|\nabla T_h^{k},\nabla T_h^{k+1},\nabla N_h^{k},\nabla \Phi_h^{k},\nabla \Phi_h^{k+1}\Big\|_{L^2\left(\Omega\right)}^2\right).
	\end{array}
	\end{equation}	
	
	Adding $\left(\ref{cotaLinfH1_N}\right)$ and $\left(\ref{cotaLinfH1_Phi}\right)$, multiplying by $2\;dt$ and adding with respect $k=0,\ldots,\widetilde{k}-1$ with $0\leq\widetilde{k}\leq K_f$, we have (using that $\|\cdot\|_h$ is an equivalent norm to $L^2$)
	
	\begin{equation}\label{cotaLinfH1_N_Phi}
	\begin{array}{c}
	\displaystyle\Big\|\nabla N_h^{\widetilde{k}},\nabla \Phi_h^{\widetilde{k}}\Big\|_{L^2\left(\Omega\right)}^2%+\Big\|\nabla \Phi_h^{k+1}\Big\|_h^2\right)
	\leq C\left(T_f+dt\sum_{k=0}^{\widetilde{k}}\Big\|\nabla T_h^{k},\nabla N_h^{k},\nabla \Phi_h^{k}\Big\|_{L^2\left(\Omega\right)}^2\right)+C\;\Big\| \nabla N_h^{0},\nabla\Phi_{h}^0\Big\|_{L^2\left(\Omega\right)}^2.
	\end{array}
	\end{equation}
	
	We can apply discrete Gronwall Lemma for any $dt$ small enough such that $C\;dt\leq\delta_0<1$, to obtain
	$$\displaystyle\Big\|\nabla N_h^{\widetilde{k}},\nabla \Phi_h^{\widetilde{k}}\Big\|_{L^2\left(\Omega\right)}^2\leq\dfrac{C}{1-\delta_0}\left(T_f+dt\sum_{k=0}^{K_f}\Big\|\nabla T_h^{k}\Big\|_{L^2\left(\Omega\right)}^2+C\;\Big\|\nabla N_h^{0}, \nabla\Phi_h^{0}\Big\|_{L^2\left(\Omega\right)}^2\right)\;e^{\frac{C}{1-\delta_0}T_f}.$$
	
	Since $\nabla T_h^{dt}$ is bounded in $L^2\left(0,T_f;L^2\left(\Omega\right)\right)$, we deduce,
	 
	$$\left(\nabla N_h^{dt},\nabla \Phi_h^{dt}\right)\;\;\text{is bounded in}\;\; L^\infty\left(0,T_f;L^2\left(\Omega\right)\right).$$ 
	
	Hence,
	
	$$\left(N_h^{dt},\; \Phi_h^{dt}\right)\;\;\text{is bounded in}\;\; L^\infty\left(0,T_f;H^1\left(\Omega\right)\right).$$ 
	
\end{proof}

\subsection{Numerical Simulations}	
	
	The main goals of this section consist of:
	\begin{enumerate}
		\item Validate numerically the properties of the scheme $\left(\ref{eqT_space}\right)$-$\left(\ref{eqF_space}\right)$, namely, the pointwise and energy estimates. 

	   \item Compare $\left(\ref{eqT_space}\right)$-$\left(\ref{eqF_space}\right)$ with two simplifications schemes: The first one changing the time approximation for a completely explicit scheme and later changing the space approximation for the scheme $\left(\ref{eqT_space}\right)$-$\left(\ref{eqF_space}\right)$ without "mass-lumping".
\end{enumerate}

We start computing the lower and upper bounds of $T_h^{k+1}$ for these schemes. We consider $T_f=1$, time step $dt=10^{-2}$, mesh size $h=0.025$ and the parameters are taken as:

\begin{table}[H]
	\centering
	\begin{tabular}{c|c|c|c|c|c|c|c|c|c}
		
		Parameter& $\kappa_1$&  $\kappa_0$   &	$\rho$	& $\alpha$  &  $\beta_1$ & $\beta_2$ & $\gamma$ &$\delta$ & $K$   \\
		\hline
		Value  & $8\cdot10^{-5}$	& $8\cdot10^{-5}$ & $1$ & $0.8$ & $0.8$  & $0.8$ & $0.008$ & $0.8$&$1$  \\
		
	\end{tabular}
	\caption{\label{Table2} Parameters value.}
\end{table}

%Note that hypothesis $\left(\ref{CondiNlejosKdelta}\right)$ is satisfied, hence tumor and vasculature will vanish at infinity time. Indeed, 
We take the initial vasculature $\Phi_0\left(x\right)=0.5$ and the initial conditions for the tumor and necrosis given in Figure $\ref{tumor_inicial_max_min}$:

\begin{figure}[H]
	\centering
	\begin{subfigure}[b]{0.35\linewidth}
		\includegraphics[width=0.8\linewidth]{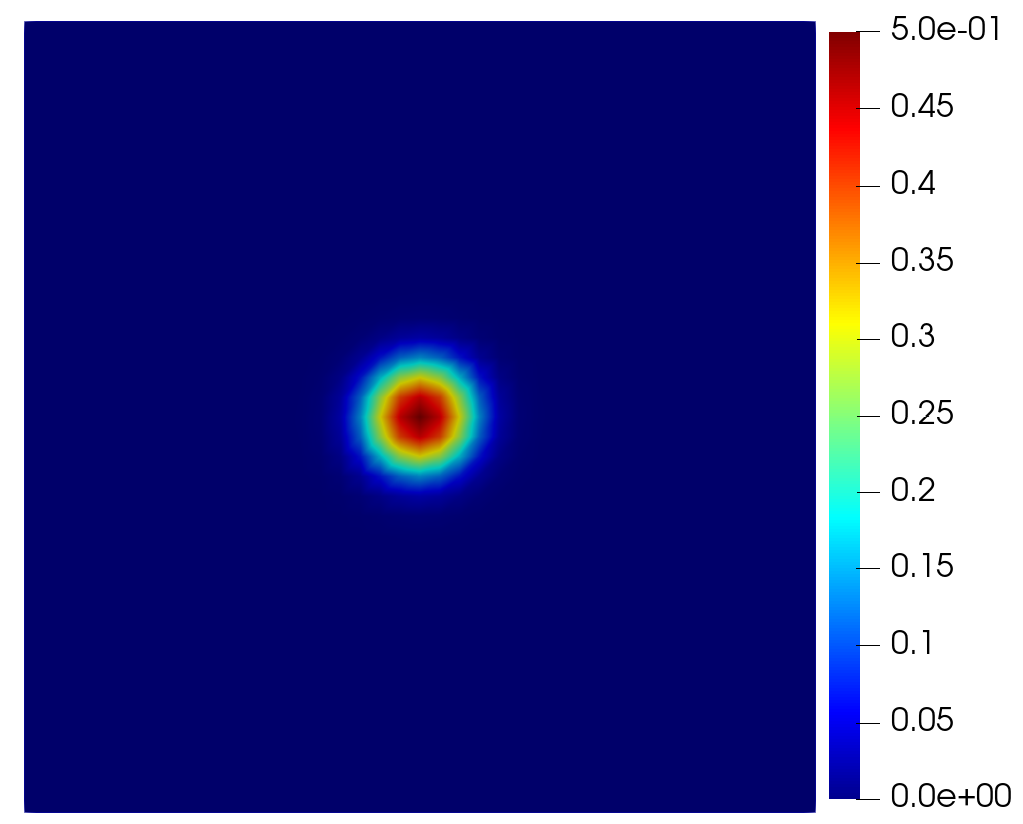}
		\centering
		\caption{ Initial tumor.}
	\end{subfigure}
	\hspace{0.5cm}
	\begin{subfigure}[b]{0.35\linewidth}
		\includegraphics[width=0.8\linewidth]{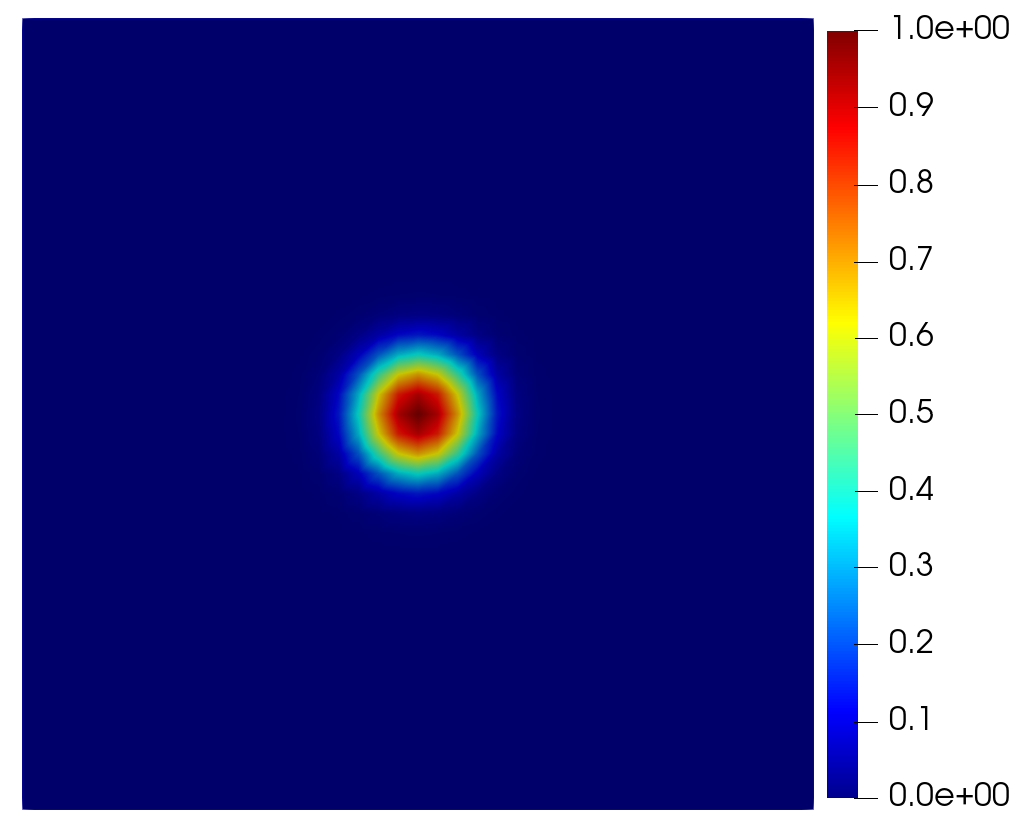}
		\centering
		\caption{ Initial necrosis.}
	\end{subfigure}
	\hspace{1cm}
	\caption{ Initial tumor and necrosis.}
	\label{tumor_inicial_max_min}
\end{figure}

We show in Figure $\ref{min_max}$ the minimum and maximum value of $T_h^{k+1}$ in the first $10$ time steps using IMEX and completely explicit scheme:

\begin{figure}[H]
				\includegraphics[width=17cm, height=6cm]{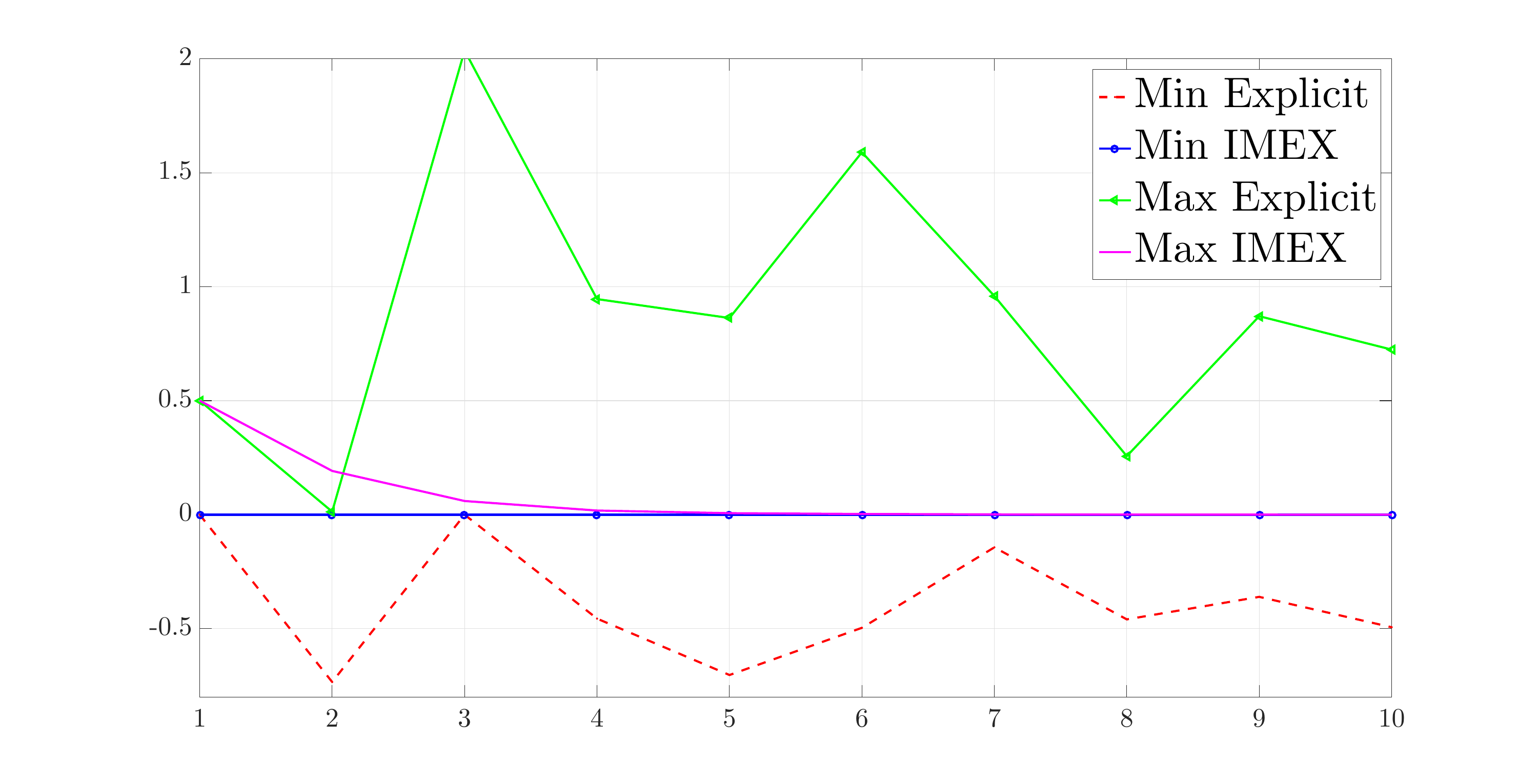}
		\centering
	\caption{Pointwise estimate for $T_h^{k+1}$ versus time using IMEX and completely explicit scheme.}
	\label{min_max}
\end{figure}

We observe that lower and upper bounds are not satisfied for the completely explicit scheme while for IMEX scheme we get the pointwise estimates proved in Lemmas $\ref{positivo_space}$ and $\ref{cota_superior_space}$. Moreover, taking the mesh size $h$ smaller, the completely explicit scheme has a similar behaviour. Hence, we can conclude that the explicit time approximation does not satisfy the maximum principle.
\\

In our second numerical simulation, we compare graphically the lower bound of $T_h^{k+1}$ for our scheme $\left(\ref{eqT_space}\right)$-$\left(\ref{eqF_space}\right)$ and for the same scheme $\left(\ref{eqT_space}\right)$-$\left(\ref{eqF_space}\right)$ but without "mass-lumping". We consider $T_f=1$, tiem step $dt=10^{-2}$, $h=0.1$ and the parameters are taken as:

\begin{table}[H]
	\centering
	\begin{tabular}{c|c|c|c|c|c|c|c|c|c}
		
		Parameter& $\kappa_1$&  $\kappa_0$   &	$\rho$	& $\alpha$  &  $\beta_1$ & $\beta_2$ & $\gamma$ &$\delta$ & $K$   \\
		\hline
		Value  & $8\cdot10^{-4}$	& $8\cdot10^{-4}$ & $1$ & $0$ & $0$  & $0$ & $0$ & $0$&$1$  \\
		
	\end{tabular}
	\caption{\label{Table4} Parameters value.}
\end{table}

We take again the initial vasculature $\Phi_0\left(x\right)=0.5$ and the initial conditions for the tumor and necrosis given in Figure $\ref{tumor_inicial_max_min}$. We show in Figure $\ref{min_ma_2x}$ the minimum value of $T_h^{k+1}$ in $40$ time step using IMEX with "mass-lumping" and IMEX without "mass-lumping":

\begin{figure}[H]
		\includegraphics[width=17cm, height=6cm]{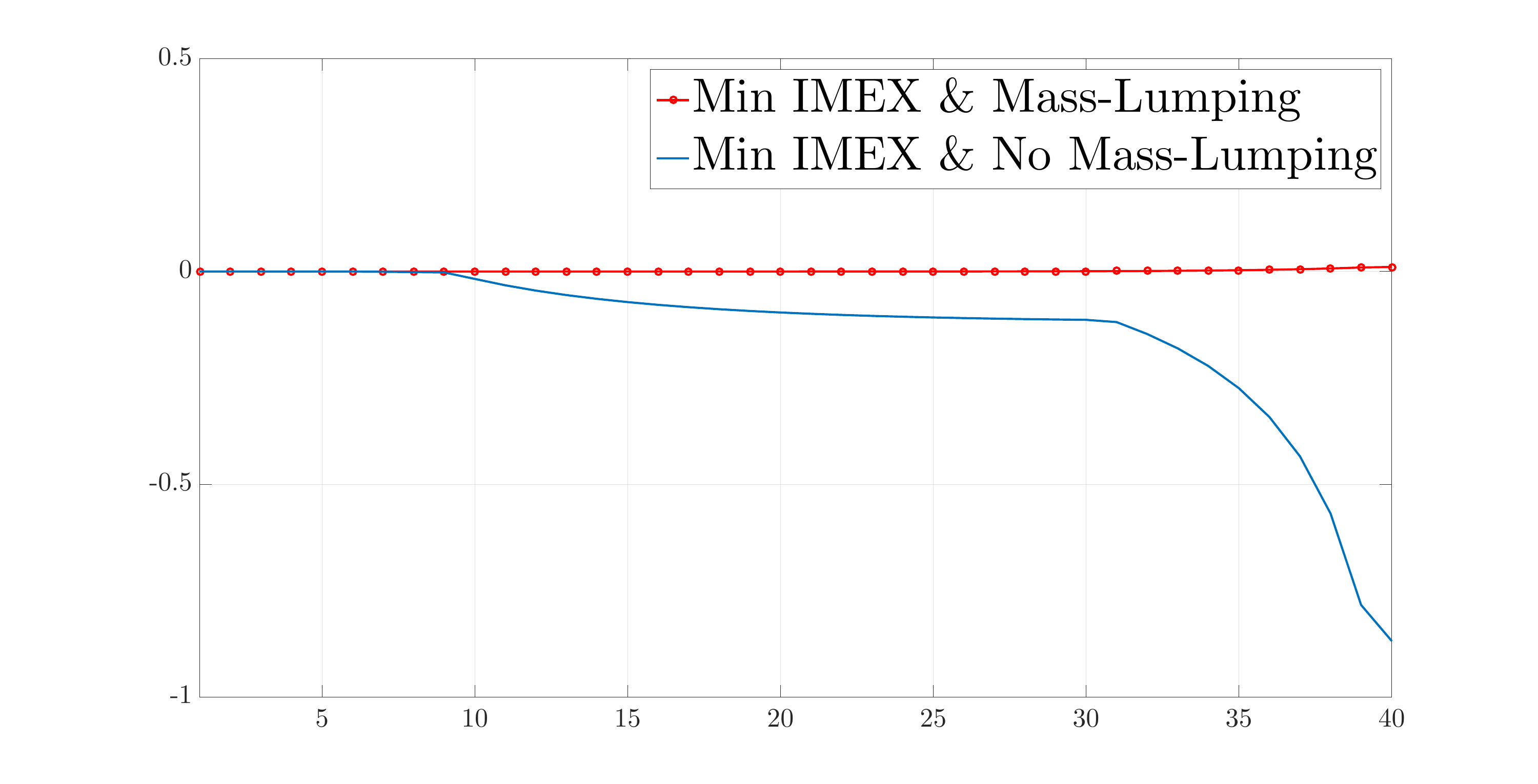}
	\centering
	\caption{Minimum value of $T_h^{k+1}$ using  IMEX with "mass-lumping" and IMEX without "mass-lumping".}
	\label{min_ma_2x}
\end{figure}

We observe how positivity is not satisfied for IMEX without "mass-lumping" while it is conserved for IMEX with "mass-lumping", in agreement with Lemma $\ref{positivo_space}$. Moreover, taking the time step $dt$ smaller, we do not get positivity for scheme $\left(\ref{eqT_space}\right)$-$\left(\ref{eqF_space}\right)$ without "mass-lumping". Hence, we can conclude that the space approximation without "mass-lumping" does not satisfy positivity.
\\ 

Thus, we have proved that the completely explicit scheme and the IMEX without "mass-lumping" do not satisfy positivity.
\\

Finally, we are going to check the energy estimate of $T_h^{k+1}$ obtained in Lemma $\ref{l2h1_numerico}$ for our scheme $\left(\ref{eqT_space}\right)$-$\left(\ref{eqF_space}\right)$ and for a completely explicit scheme and finite element with "mass-lumping". Now, we consider $T_f=0.01$, the mesh size $h=0.025$, the same initial condition than in Figure $\left(\ref{tumor_inicial_max_min}\right)$ and the parameters are taken as:

\begin{table}[H]
	\centering
\begin{tabular}{c|c|c|c|c|c|c|c|c|c}
	
	Parameter& $\kappa_1$&  $\kappa_0$   &	$\rho$	& $\alpha$  &  $\beta_1$ & $\beta_2$ & $\gamma$ &$\delta$ & $K$   \\
	\hline
	Value  & $2.9\cdot10^{-7}$	& $2.9\cdot10^{-7}$ & $1$ & $0.0029$ & $0.0029$  & $0$ & $0.0029$ & $0.00029$&$1$  \\
	
\end{tabular}
	\caption{\label{Table3}Parameters value.}
\end{table}

We show in Figure $\ref{L2H1}$ the value of $\displaystyle\|T_h^{dt}\|_{L^2\left(0,T_f;H^1\left(\Omega\right)\right)}^2$ for the different $dt$ obtained with $$K_f=10,\;60,\;110,\;160,\;210,\;260,\;310,\;360,\;410,\;460,\;510.$$ using the IMEX and completely explicit scheme. 

\begin{figure}[H]
	\includegraphics[width=17cm, height=6cm]{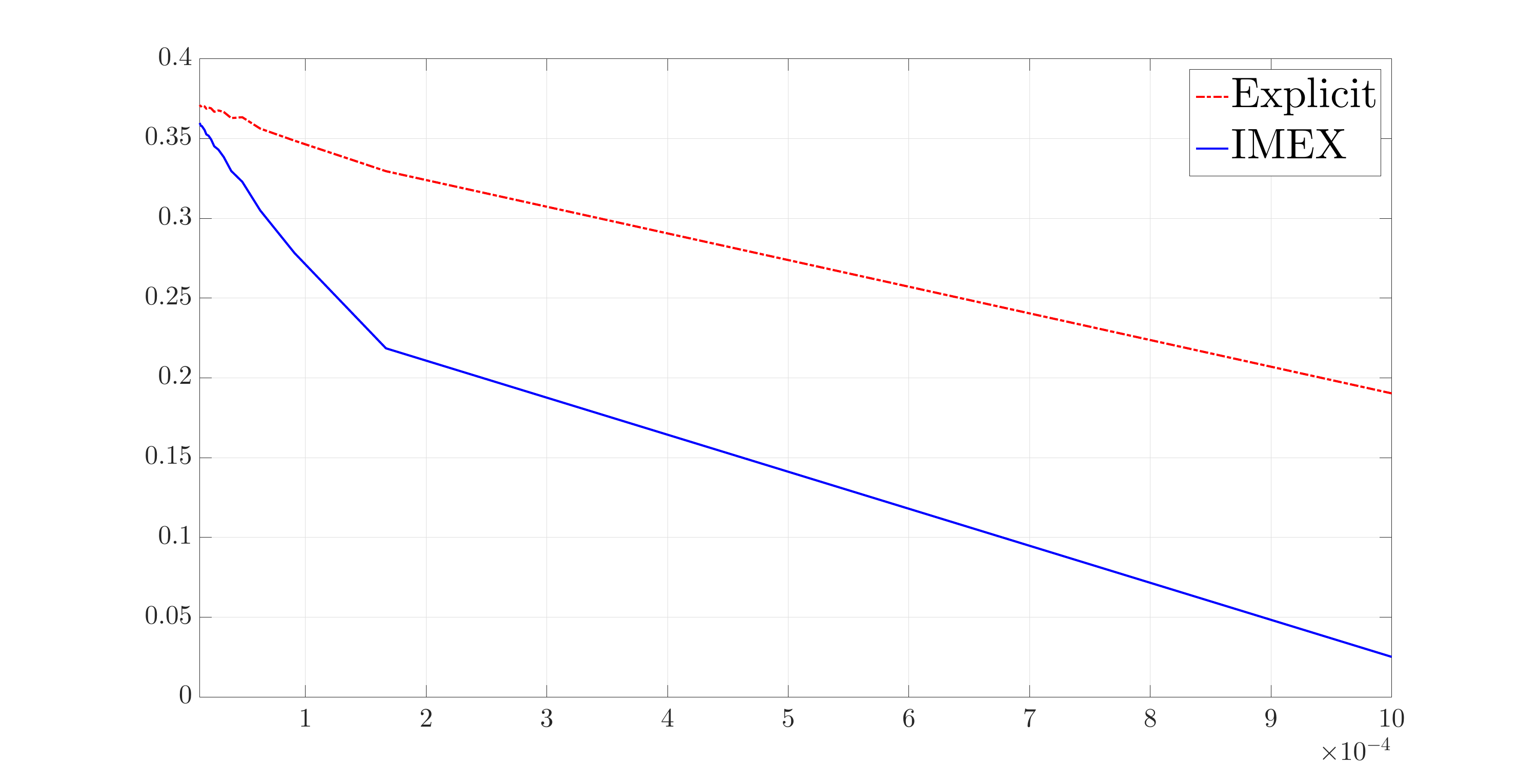}
	\centering
	\caption{Value of $\|T_h^{dt}\|_{L^2\left(0,T_f;H^1\left(\Omega\right)\right)}^2$ versus time using IMEX and completely explicit scheme.}
	\label{L2H1}
\end{figure}

We observe that the difference between the value of $\|T_h^{dt}\|_{L^2\left(0,T_f;H^1\left(\Omega\right)\right)}^2$ using these two schemes increases as $dt$ increases.
\\

	\begin{obs}
		We have presented some numerical simulations in order to verify the analytical results of Section $\ref{esquema_espacio}$. In a forthcoming paper, \cite{Romero_2021}, we will explore the behaviour of the model depending on the parameter using appropriate numerical simulations. In particular, we will study different situations such as tumor growth with vasculature non-uniformly distributed.
 Moreover, in all the above simulations, the hypothesis $\left(\ref{CondiNlejosKdelta}\right)$ is satisfied and hence tumor, $T$, and vasculature, $\Phi$, will vanish at infinity time. When the proliferating part of the tumor, $T$, goes to zero, only the necrotic part, $N$, remains. This situation represents that the tumor remains encapsulated, it could not longer grow.
	\end{obs}

\section*{Acknowledgments}
The authors would like to thank the reviewer for his/her fruitful comments.
%\bibliographystyle{ws-m3as}
%\bibliographystyle{abbrvnat} 
%\bibliography{reference}
	
\end{document}